\pdfoutput=1

\documentclass[12pt]{amsart}
\usepackage[usenames,dvipsnames]{xcolor}
\usepackage{mathtools,amssymb,amsthm,tikz,tikz-cd,multirow}
\usepackage{fullpage}
\usepackage{etoolbox}
\usepackage{extarrows}
\usepackage{enumitem}
\usepackage{makecell}
\usepackage{tensor}
\usepackage{tabu}

\usepackage[font=footnotesize,labelfont=bf, width=.9\textwidth]{caption}

\usepackage[protrusion=true,expansion=true]{microtype}
\setlist[itemize]{label={$\vcenter{\hbox{\tiny$\bullet$}}$}, leftmargin=2\parindent}
\setlist[enumerate,1]{label={\upshape(\roman*)}, leftmargin=2\parindent}

\DeclareMathAlphabet{\mathbbold}{U}{bbold}{m}{n}

\DeclareRobustCommand{\SkipTocEntry}[5]{}

\newcommand\GL{\operatorname{\mathbf{GL}}}

\newcommand\res{\operatorname{res}}

\newcommand\z{\mathbf{z}}

\newcommand{\assign}{:=}

\DeclarePairedDelimiter{\floor}{\lfloor}{\rfloor}
\DeclarePairedDelimiter\ket{\lvert}{\rangle}
\DeclarePairedDelimiter\bra{\langle}{\rvert}

\DeclareMathOperator{\DT}{DT}

\usetikzlibrary{decorations,decorations.pathreplacing,positioning,calc,tikzmark}

\usepackage[protrusion=true,expansion=true]{microtype}
\usepackage{hyperref}

\hypersetup{
    pdftitle={ },
    pdfauthor={Ben Brubaker, Valentin Buciumas, Daniel Bump, Henrik P. A. Gustafsson},
    linktocpage=true,           
    colorlinks=true,            
    linkcolor=blue!75!black,    
    citecolor=green!50!black,   
    filecolor=magenta,          
    urlcolor=blue!75!black      
}

\let\oldtocsubsection=\tocsubsection
\let\oldtocsubsubsection=\tocsubsubsection
\renewcommand{\tocsubsection}[2]{\hspace{1.8em}\oldtocsubsection{#1}{#2}}
\renewcommand{\tocsubsubsection}[2]{\hspace{3.6em}\oldtocsubsubsection{#1}{#2}}

\begin{document}

\title{Iwahori-metaplectic duality}
\author{Ben Brubaker}
\address{School of Mathematics, University of Minnesota, Minneapolis, MN 55455}
\email{brubaker@math.umn.edu}
\author{Valentin Buciumas}
\address{Korteweg--de Vries Institute for Mathematics, University of Amsterdam, Science Park 105-107,
1098 XG Amsterdam, The Netherlands}
\email{valentin.buciumas@gmail.com}
\author{Daniel Bump}
\address{Department of Mathematics, Stanford University, Stanford, CA 94305-2125}
\email{bump@math.stanford.edu}
\author{Henrik P. A. Gustafsson}
\address{Department of Mathematics and Mathematical Statistics, Umeå University, SE-901 87 Umeå, Sweden}
\email{henrik.gustafsson@umu.se}

\definecolor{green}{RGB}{0,180,0}
\definecolor{brown}{RGB}{120,0,120}
\colorlet{sgreen}{green!50!black}
\colorlet{sblue}{blue!50!black}
\colorlet{sred}{red!50!black}

\newtheorem{theorem}{Theorem}[section]
\newtheorem{lemma}[theorem]{Lemma}
\newtheorem{proposition}[theorem]{Proposition}
\newtheorem{assumption}[theorem]{Assumption}
\newtheorem{corollary}[theorem]{Corollary}

\newtheorem{maintheorem}{Theorem}
\renewcommand{\themaintheorem}{\Alph{maintheorem}}
\newtheorem{maincorollary}[maintheorem]{Corollary}
\renewcommand{\themaincorollary}{\Alph{maincorollary}}

\newtheorem{customthm}{Theorem}
\newenvironment{customtheorem}[1]{\renewcommand{\thecustomthm}{#1}\customthm}{\endcustomthm}

\theoremstyle{definition}
\newtheorem{remark}[theorem]{Remark}
\newtheorem{definition}[theorem]{Definition}

\newtheorem{example}[theorem]{Example}

\numberwithin{equation}{section}



\begin{abstract}
We construct a family of solvable lattice models whose partition functions include $p$-adic Whittaker functions for general linear groups from two very different sources: from Iwahori-fixed vectors and from metaplectic covers. Interpolating between them by Drinfeld twisting, we uncover unexpected relationships between Iwahori and metaplectic Whittaker functions. This leads to new Demazure operator recurrence relations for spherical metaplectic Whittaker functions. 
In prior work of the authors it was shown that the row transfer matrices of certain lattice models for spherical metaplectic Whittaker functions could be represented as ``half-vertex operators'' operating on the $q$-Fock space of Kashiwara, Miwa and Stern. In this paper the same is shown for all the members of this more general family of lattice models including the one representing Iwahori Whittaker functions.
\end{abstract}

\maketitle

\tableofcontents

\pagebreak

\section{Introduction}
This paper explores {\emph{solvable lattice models}} and their associated quantum groups appearing in the study of certain Whittaker functions --- special functions 
from the representation theory of $p$-adic algebraic groups. The lattice models in this paper consist of
grids, whose edges may be assigned data called {\emph{spins}} which in the
simplest cases are just $+$ or $-$, but which in other cases may be taken from
a larger set. Assigning spins to the edges results in a {\emph{state}} of the
system which is then assigned a {\emph{Boltzmann weight}}; the sum of the
Boltzmann weights over all states is the {\emph{partition function}}.
Solvability of the model means that the Boltzmann weights satisfy an extraordinary
relation called the {\emph{Yang-Baxter equation}}. This involves an auxiliary
vertex called an {\emph{R-matrix}} and the partition functions of solvable models
are amenable to study by means of the Yang-Baxter equation. As a consequence,
these partition functions satisfy algebraic relations leading to important
information such as Demazure-like recurrences, evaluations as determinants or
other explicit forms, Cauchy identities and branching rules. Although solvable lattice
models originally arose in statistical mechanics~\cite{Baxter, JimboMiwaAlgebraic}, they have recently found
applications in many other areas including integrable probability~\cite{BorodinPetrovIntegrableProbability,CorwinPetrov,BarraquandBorodinCorwinWheeler,AggarwalBorodinPetrovWheeler}, algebraic
combinatorics~\cite{ABW,BorodinWheelerColored,WheelerZinn-JustinCrelle,BBBGVertex} and the 
representation theory of $p$-adic groups~\cite{BBB,BBBGIwahori,BBBGMetahori}.

According to the modern paradigm, underlying every solvable lattice model
there is a quantum group. Edges of the lattice grid correspond to modules over
the quantum group, and because the module category of a quantum group is braided, 
the Yang-Baxter equation must hold. Spins on the edges enumerate a basis for the module,
and the Boltzmann weights encode endomorphisms in terms of this basis.
A further procedure, known as {\emph{Drinfeld
twisting}}, modifies the quantum group and its R-matrix but still
results in a solution to the Yang-Baxter equation.

As a first example, we consider the spherical Whittaker function for an unramified representation of $\GL_r (F)$, where $F$ is a non-archimedean local field. The
Shintani-Casselman-Shalika formula gives an explicit evaluation of every nonzero value
of this Whittaker function as a product of a Schur polynomial and a deformation of the 
denominator in the Weyl character formula.
Tokuyama's formula \cite{Tokuyama}, under a suitable bijection, expresses these same values as the
partition functions of a solvable lattice model. We will call this simplest
example the {\emph{Tokuyama model}}. For the Tokuyama model, the relevant quantum group is $U_q
(\widehat{\mathfrak{g}\mathfrak{l}} (1|1))$, the quantized enveloping algebra
of the simplest affine Lie superalgebra. 

The connection between the representation theory of this quantum group and the representation
theory of $\GL_r (F)$ that was our starting point is somewhat obscure;
{\emph{a priori}} they do not seem to be related. We can understand the appearance of the quantum group using either of two key ingredients in proofs of the Shintani-Casselman-Shalika formula: the properties of intertwining operators for the principal series
representations ({\cite{CasselmanShalika,BBFMatrix,BBBF}}) or of operators in the
Hecke algebra acting on the modules by convolution
({\cite{IwahoriMatsumoto,ShintaniWhittaker}}). Either way, the relevant
operators can be imitated by an R-matrix, leading to recurrence
relations for the lattice models that are identical to those satisfied by the
Whittaker functions. (This is a revisionist explanation that does not actually
reflect the original approach to Tokuyama's formula.)

There are two ways this result --- identifying the spherical Whittaker functions
with partition functions of the Tokuyama model --- may be generalized:
\begin{itemize}
  \item We may consider a more refined set of invariants than the spherical Whittaker
  function, namely a basis of Whittaker functions fixed by the Iwahori subgroup. There are $r! = | W |$ of
  these, where $W \cong S_r$ is the Weyl group. The values of these Iwahori
  Whittaker functions were again interpreted as the partition functions of
  solvable lattice models in~\cite{BBBGIwahori}. The
  underlying quantum group is a Drinfeld twist of $U_q
  (\widehat{\mathfrak{g}\mathfrak{l}} (r| 1))$. The models will be referred to
  as {\emph{Iwahori ice.}}
  
  \item We may consider Whittaker functions on a metaplectic cover
  $\widetilde{\GL}_r^{(n)} (F)$ of $\GL_r (F)$, a central extension of
  degree~$n$. These, too, may be expressed in terms of solvable lattice models
  called {\emph{metaplectic ice}}, as shown in
  {\cite{BBB}}. For these covers, there are $n^r$ such Whittaker models (each having a unique spherical Whittaker function up to normalization). The
  underlying quantum group is a Drinfeld twist of $U_q
  (\widehat{\mathfrak{g}\mathfrak{l}} (1| n))$, where the twisting data involves 
  $n$-th order Gauss sums that appear in the intertwining
  integrals for the principal series representations.
\end{itemize}
While the Yang-Baxter equation provides ways to analyze a solvable lattice model, in some special cases 
we may instead express its row transfer matrix (which describes all one-row partition functions) in terms of exponentiated
Hamiltonian operators, themselves module endomorphisms of a suitable Fock space.
The main theorem of~\cite{BBBGVertex} does this for the models representing spherical 
metaplectic Whittaker functions. 
When this can be done, the Yang-Baxter equation can be replaced, as the principal tool, 
by the algebraic formalism of the boson-fermion correspondence.
In~\cite{BBBGVertex} (as in~\cite{LLTRibbon,LamRibbon}), the Hamiltonians are operators
on the $q$-Fock space of Kashiwara, Miwa and Stern~\cite{KMS}.

This paper arose from our desire to do the same for Iwahori ice models, which we 
accomplish in Theorem~\ref{thm:T-hamiltonian} in Section~\ref{sec:results}.
Before setting the stage for Theorem~\ref{thm:T-hamiltonian}, we highlight several surprising results we discovered along the way.
Our initial inspiration came from~{\cite{BBBGMetahori}}, where we give lattice
models whose partition functions are Iwahori Whittaker functions on the
metaplectic groups $\widetilde{\GL}_r^{(n)} (F)$, thereby unifying the two cases described above using the associated quantum group $U_q
(\widehat{\mathfrak{g}\mathfrak{l}} (r|n))$. But perhaps even more importantly, the paper~{\cite{BBBGMetahori}} treated the two specializations to the non-metaplectic Iwahori case and to the spherical metaplectic case on similar footing, as evoked by the terminology for the spins on edges in the model which consisted of $r$ colors and $n$ supercolors (``scolors''). 
The Iwahori case amounts to restricting to only colors, while the metaplectic case amounts to restricting to only supercolors.
This paper fully explores that similarity by defining a general class of lattice models related by Drinfeld twisting flexible enough to include each case.
In doing so, we connect the Iwahori $U_q(\widehat{\mathfrak{gl}}(r|1))$ model to the metaplectic $U_q(\widehat{\mathfrak{gl}}(1|n))$ model via the duality alluded in the title.

These lattice models show that (non-metaplectic) Iwahori Whittaker functions and 
(spherical) metaplectic Whittaker functions are described by the same
mathematics. This is quite surprising from the point of view of $p$-adic representation theory, as Iwahori subgroups and metaplectic covering groups seem like very different animals. 

Indeed, at the very concrete level of special functions,
metaplectic Whittaker functions involve Gauss sums, and Iwahori
Whittaker functions do not. However the Gauss sum data can be
introduced into the partition functions by Drinfeld twisting,
so the $n$-th order Gauss sums that appear can be regarded as
variables that can be specialized in different ways. 

Another objection to any such connection is that the dimensions of these two vector spaces don't agree. There are $n^r$ spherical Whittaker
functions (for the covers we consider) but $r!$ Iwahori Whittaker
functions. To explain this discrepancy, we note that 
the $U_q(\widehat{\mathfrak{gl}}(r|1))$ models that
represent $\GL_r$ Iwahori Whittaker functions could
be replaced by models for
$U_q(\widehat{\mathfrak{gl}}(m|1))$ for any integer $m \geqslant r$; these models would have $m$ colors. This can be seen from
Figure~7 in~\cite{BBBGIwahori}, where we note that
the Boltzmann weights depend in inequalities between
colors, but in no way involve the value~$r$.
Furthermore, as is explained in~\cite{BBBGIwahori},
the models there can be used not only to represent
Iwahori Whittaker functions, but more generally
\emph{parahoric} Whittaker functions, by allowing multiple
uses of the same color in the boundary conditions.
If we use a palette of $m$ colors and work in the
generality that includes parahoric Whittaker functions,
we find that there are $m^r$ different models.
These represent the Iwahori and parahoric Whittaker
functions redundantly, but this point of view is good
for seeing the perfect parallel between Iwahori and
metaplectic models.

This is a good advertisement for the lattice model perspective, which allows one to make the bridge between 
the Iwahori $U_q (\widehat{\mathfrak{g}\mathfrak{l}} (r| 1))$ 
models and the metaplectic $U_q (\widehat{\mathfrak{g}\mathfrak{l}} (1| n))$ models
and to recognize their closely related quantum groups.
Our contention expressed in Theorems~\ref{thm:specializations} and~\ref{thm:YBE} is that if
we interchange the roles of $n$ and $r$, that is, colors and supercolors, these models are essentially the
same, a fact with many consequences for explicit connections between Iwahori Whittaker
functions and metaplectic Whittaker functions (see Theorems~\ref{thm:DW},~\ref{thm:Tokuyama}, and~Corollary~\ref{cor:Cauchy}
in the next section). We will refer to this surprising relationship uncovered in this paper as {\emph{Iwahori-metaplectic duality}}.
We are using the term ``duality'' here in the physics sense, referring
to two seemingly different objects that are secretly equivalent.

Indeed this Iwahori-metaplectic duality allows us to transfer facts about one type of Whittaker function to the other. This is precisely what
allows us to adapt methods of~\cite{BBBGVertex} to prove our original aim --- a description of Iwahori Whittaker functions in terms of Hamiltonian operators in Theorem~\ref{thm:T-hamiltonian}. But it also has interesting applications in the other direction, from Iwahori Whittaker functions to metaplectic Whittaker functions.
Recall the important role of \emph{Demazure operators} in the theory of
Iwahori Whittaker functions~{\cite{BBL,BBBGIwahori}}. Now by exploiting Iwahori-metaplectic duality, one expects to find dual Demazure
recurrences among the values of the $n^r$ {\emph{spherical metaplectic}} Whittaker functions. 
We will explain how in Section~\ref{sec:demaction}.
These recurrences, and even the base case, are different from previous explicit formulas for metaplectic Whittaker functions, and particularly from the Demazure recurrences for the $r!n^r$
{\emph{Iwahori}} metaplectic Whittaker functions in
{\cite{ChintaGunnellsPuskas,PatnaikPuskasIwahori,BBBGMetahori,SSV,SSV2}}. Moreover,
Iwahori vectors and Whittaker functions have connections with Kazhdan-Lusztig
theory and the geometry of Bott-Samelson resolutions
({\cite{RogawskiIwahori,KazhdanLusztigDeligneLanglands,ReederCompositio,BBL,MihalceaSuWhittaker,AMSS}}).
We might therefore expect to find similar connections with spherical metaplectic
Whittaker functions, though we refrain from exploring that here.

\addtocontents{toc}{\SkipTocEntry}
\subsection*{Acknowledgements}
We thank Solomon Friedberg for helpful discussions and comments.
Brubaker was supported by NSF grant DMS-2101392.
Buciumas was supported by NSERC Discovery RGPIN-2019-06112, the endowment of the M.V. Subbarao Professorship in Number Theory and the Netherlands Organization for Scientific Research (NWO) project number 613.009.126.
Gustafsson was supported by the Swedish Research Council (Vetenskapsr\aa det) grant 2018-06774.
\vfil

\section{Statement of results} \label{sec:results}

In this paper we construct a family of lattice models consisting of colored paths on a two-dimensional grid.
For a fixed positive integer $m$ let $\mathcal{P}_m$ denote an ordered set of colors $c_1 < c_2 < \cdots < c_m$ called a palette.
Each path is assigned a color from this palette.

We will consider two versions of the lattice models differing only in their boundary conditions: a finite grid with paths entering at the top boundary and exiting at the right boundary, and a grid with an infinite number of columns with paths entering at the top boundary and exiting at the bottom boundary.
The finite systems will be related to Whittaker functions and are discussed in Section~\ref{sec:family} while the infinite systems will be related to Fock space operators in Section~\ref{sec:Fock-space}.

Rows are numbered $1$ to $r$ from top to bottom and for the finite system we number the columns $0$ to $N-1$ from right to left (for some $N$ large enough).
Each row $i$ is assigned a nonzero complex parameter $z_i$.
The columns for the infinite system are also assigned increasing numbers to the left starting with $0$ at some arbitrarily chosen column.

Each vertex of the grid, which is called a T-vertex, has four adjacent edges and a state of the lattice model given some boundary conditions is an assignment of colors to edges according to the following rules.
A horizontal edge can be assigned either no color or any one color in~$\mathcal{P}_m$.
A vertical edge can be assigned either no color or a predetermined color~$c_j$ dictated by the column number $k$ such that $j + k \equiv 0 \bmod m$ and because of this restriction we say that $c_j$ is the column color for that column, or the vertex color of that vertex, even if none of its vertical edges are colored.
That is, the color for each column repeats in blocks of $c_1, \ldots, c_m$ from left to right ending with $c_m$ at column number $0$ which is the rightmost column for the finite system.
For an admissible state the edges of each vertex must also have color assignments according to the set of vertex configurations shown in Table~\ref{tab:boltzmann-weights}.
For an example of a state in a finite system see Figure~\ref{fig:example}.

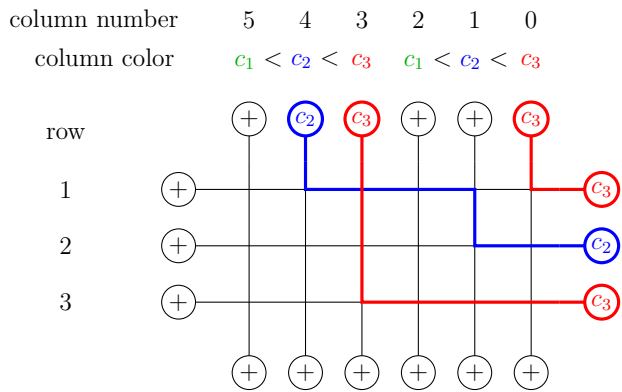
\begin{figure}[htpb]
  \centering

  \tikzstyle{state}=[circle, draw, fill=white, minimum size=16pt, inner sep=0pt, outer sep=0pt]
\tikzstyle{selection}=[thick, color=gray, fill=lightgray!30, rounded corners, text=black]
\tikzstyle{selection-line}=[selection, rounded corners=0pt, fill=none, shorten <=2pt, shorten >=2pt]

  \begin{tikzpicture}[baseline=0.5cm, scale=0.75, every node/.append style={scale=0.8}]
  
  \draw (0,0) node[state] {$+$} -- (6.75,0);
  \draw (0,1) node[state] {$+$} -- (6.75,1);
  \draw (0,2) node[state] {$+$} -- (6.75,2);
  \foreach \x in {0,...,5}{
    \draw (\x+1.25, -1.25) node[state] {$+$} -- (\x+1.25,2.5);
    \node at (6.25-\x,5) {$\x$};
  }
 
  \node[align=right] at (1.25, 5) {\llap{column number\hspace{1cm}}\phantom{$c_1$}};
  \node at (-2,3) {row};
  \node at (-2,2) {$1$};
  \node at (-2,1) {$2$};
  \node at (-2,0) {$3$};

  \draw[blue, very thick] (2.25,2.5) -- (2.25,2) -- (5.25,2) -- (5.25,1) -- (6.75,1);
  \draw[red, very thick]  (3.25,2.5) -- (3.25,0) -- (6.75,0);
  \draw[red, very thick]  (6.25,2.5) -- (6.25,2) -- (6.75,2);

  \draw (0+1.25, 3.25) node[state, label={[label distance=4mm]above:\llap{\color{black}column color\hspace{1cm}}$\color{green}c_1$ \color{black}\rlap{$<$}}] {$+$} -- (0+1.25, 2.5);
  \draw[blue,very thick] (1+1.25, 3.25) node[state, label={[label distance=4mm]above:$\color{blue}c_2$ \color{black}\rlap{$<$}}] {$c_2$} -- (1+1.25, 2.5);
  \draw[red, very thick] (2+1.25, 3.25) node[state, label={[label distance=4mm]above:$\color{red}c_3$}] {$c_3$} -- (2+1.25, 2.5);
  \draw (3+1.25, 3.25) node[state, label={[label distance=4mm]above:$\color{green}c_1$ \color{black}\rlap{$<$}}] {$+$} -- (3+1.25, 2.5);
  \draw (4+1.25, 3.25) node[state, label={[label distance=4mm]above:$\color{blue}c_2$ \color{black}\rlap{$<$}}] {$+$} -- (4+1.25, 2.5);
  \draw[red, very thick] (5+1.25, 3.25) node[state, label={[label distance=4mm]above:$\color{red}c_3$}] {$c_3$} -- (5+1.25, 2.5);
  \draw[red,  very thick] (7.5, 0) node[very thick, state] {$c_3$} -- (6.75, 0);
  \draw[blue, very thick] (7.5, 1) node[very thick, state] {$c_2$} -- (6.75, 1);
  \draw[red,  very thick] (7.5, 2) node[very thick, state] {$c_3$} -- (6.75, 2);
\end{tikzpicture}

  \caption{Example of a state for a finite system. For readability we have marked each boundary edge by a circled label with its color $c_i$, or with a plus sign if the edge is not colored. When possible, we will often refrain from similarly labelling each interior edge to avoid clutter.}
  \label{fig:example}
\end{figure}

For nonzero complex parameters $\Phi$ and $\alpha_{i,j}$ with $i,j \in \{1,2,\ldots,m\}$ such that $\alpha_{i,j} \alpha_{j,i} = 1$ and $\alpha_{i,i} = 1$, we define a Boltzmann weight associated to each vertex configuration according to the last row of Table~\ref{tab:boltzmann-weights} where $z$ is the row parameter $z_i$ for row $i$.
For convenience we have extended the parameters $\alpha_{i,j}$ to all $i,j \in \mathbb{Z}$ by $m$-periodicity such that $\alpha_{-i,-j} = \alpha_{m-i,m-j}$.
This parametrization with negative $\alpha$-indices is chosen such that $\alpha_{i,j}$ matches the likewise named parameter of \cite{BBBGVertex} which is a Drinfeld twist of the quantum Fock space and is discussed in Section~\ref{sec:Fock-space}.
The minus signs appear when translating from models with colored paths to models with supercolored paths as discussed in Section~\ref{sec:Delta}.
See in particular Table~\ref{tab:super-boltzmann-weights}.

\begin{table}[htpb]
  \centering
  \caption{T-vertex configurations and their Boltzmann weights at a column with color $c_j$ for our family of lattice models.}
  \label{tab:boltzmann-weights}
\begin{tabular}{|c|c|c|c|c|c|}
  \hline
  $\texttt{a}_1$ & $\texttt{a}_2$ & $\texttt{b}_1$ & $\texttt{b}_2$ & $\texttt{c}_1$ & $\texttt{c}_2$ \\\hline\hline
  
  \begin{tikzpicture} 
    \draw [] (0,-1) -- (0,1) node[label=above:$c_j$] {};
    \draw [] (-1,0) -- (1,0);
  \end{tikzpicture}
 
  &
  
  \begin{tikzpicture}
    \draw [blue, ultra thick] (0,-1) -- (0,1) node[label=above:$c_j$] {}; 
    \draw [red, ultra thick] (-1,0) -- (1,0) node[label=right:$c_i$] {}; 
  \end{tikzpicture}
 
  &
  
  \begin{tikzpicture}
    \draw [] (-1,0) -- (1,0); 
    \draw [red, ultra thick] (0,-1) -- (0,1) node[label=above:$c_j$] {};
  \end{tikzpicture}
 
  &
  
  \begin{tikzpicture}
    \draw [] (0,-1) -- (0,1) node[label=above:$c_j$] {};
    \draw [red, ultra thick] (-1,0) -- (1,0) node[label=right:$c_i$] {}; 
  \end{tikzpicture} 
 
  &
  
  \begin{tikzpicture}
    \draw [] (1,0) -- (0,0) -- (0,1) node[label=above:$c_j$] {};
    \draw [red, ultra thick] (0,-1) -- (0,0) -- (-1,0);
  \end{tikzpicture}
  
  &
  
  \begin{tikzpicture}
    \draw [] (-1,0) -- (0,0) -- (0,-1);
    \draw [red, ultra thick] (0,1) node[label=above:$c_j$] {} -- (0,0) -- (1,0);
  \end{tikzpicture}
  
  \\\hline
  $1$ & \scriptsize $\Phi \times
  \begin{cases}
    q z & i = j\\
    \alpha_{-i, -j} & i \neq j
    \end{cases}
    $ & $-\frac{\Phi}{q}$ & \scriptsize $\begin{cases}
    z & i = j\\
    1 & i \neq j
  \end{cases}$ & $ -\frac{\Phi}{q} (1 - q^2) z$ & $1$ \\\hline
\end{tabular}
\end{table} 
\noindent The overall Boltzmann weight for a state $\mathfrak{s}$ is defined as the product of the Boltzmann weights for each vertex and the partition function $Z(\mathfrak{S})$ for an ensemble (or system) $\mathfrak{S}$ of admissible states for given boundary conditions is the sum of the Boltzmann weights of the states
\begin{equation}
  \label{eq:partition-function}
  Z(\mathfrak{S}) := \sum_{\mathfrak{s} \in \mathfrak{S}} \operatorname{wt}(\mathfrak{s}), \qquad \operatorname{wt}(\mathfrak{s}) := \prod_{\mathclap{v \text{ vertex}}} \operatorname{wt}(\mathfrak{s}|_v).
\end{equation}

Even when our systems are infinite, the number of associated admissible states for given boundary data will be finite.
However, for the infinite systems one needs to pick a normalization for the Boltzmann weights of the states which each contains an infinite number of $\texttt{b}_1$ configurations of weight $(-\Phi/q)$.
As detailed in Section~\ref{sec:Fock-space}, we normalize with a vacuum-to-vacuum state such that the results are independent of the choice of the $0$ column.

We will in particular consider two specializations of these models:
\begin{itemize}
  \item the \emph{Iwahori specialization} for which $\Phi=q^{-1}$ and $\alpha_{-i,-j} = \left\{\begin{smallmatrix*}[l] 1/q &\ i<j\\ q  &\ i>j\end{smallmatrix*}\right\}$ for $1\leqslant i,j\leqslant m$,
  \item the \emph{metaplectic specialization} for which $\Phi = -q$, $\alpha_{i,j} = -g(i-j)/q$, 
\end{itemize}
where $g(x)$ is an $n$-th order Gauss sum.
Recall that $\alpha_{i,i}$ is always equal to one.
Denote the corresponding specializations of the partition function $Z(\mathfrak{S})$ by $Z(\mathfrak{S})|_\text{Iwahori}$ and $Z(\mathfrak{S})|_\text{metaplectic}$ respectively, and similarly for other quantities depending on $\Phi$ and $\alpha_{i,j}$.

\subsection{Results for finite systems}
The states we will consider consist of paths moving from the top boundary to the right boundary. Thus, the left and bottom boundaries are unoccupied.
The top boundary data of our system will consist of a strict partition $\mu$ with $r$ (distinct) nonnegative parts specifying the column numbers that are occupied for the top boundary, and an element $\sigma \in (\mathbb{Z}/m\mathbb{Z})^r$ specifying the indices for the colors on right boundary of the system from the top down. We denote the ensemble of admissible states with these boundary conditions for a palette $\mathcal{P}_m$ by $\mathfrak{S}^m_{\mu,\sigma}$.
See Figure~\ref{fig:boundary-conventions}.

\begin{figure}[htpb]
  \centering

\tikzstyle{state}=[circle, draw, fill=white, minimum size=16pt, inner sep=0pt, outer sep=0pt]
\tikzstyle{selection}=[thick, color=gray, fill=lightgray!30, rounded corners, text=black]
\tikzstyle{selection-line}=[selection, rounded corners=0pt, fill=none, shorten <=2pt, shorten >=2pt]

\begin{tikzpicture}[baseline=0.5cm, scale=0.75, every node/.append style={scale=0.8}]
  
  \draw[selection] (7,-.5) rectangle (8,3.2);
  \draw[selection] (0,3.75) rectangle (6.75,2.75);

  \draw (0,0) node[state] {$+$} -- (0.75,0) node[label=right:$z_3$] {};
  \draw (0,1) node[state] {$+$} -- (0.75,1) node[label=right:$z_2$] {};
  \draw (0,2) node[state] {$+$} -- (0.75,2) node[label=right:$z_1$] {};
  \foreach \x in {0,...,5}{
    \draw (\x+1.25, -1.25) node[state] {$+$} -- (\x+1.25,-.5);
    \node at (6.25-\x,5) {$\x$};
  }

  \node[align=right] at (1.25, 5) {\llap{column number\hspace{1cm}}\phantom{$c_1$}};
  \node at (-2,3) {row};
  \node at (-2,2) {$1$};
  \node at (-2,1) {$2$};
  \node at (-2,0) {$3$};
  
  \draw (0+1.25, 3.25) node[state, label={[label distance=2mm]left:{\color{black}$\mu$}}, label={[label distance=4mm]above:\llap{\color{black}column color\hspace{1cm}}$\color{green}c_1$ \color{black}\rlap{$<$}}] {$+$} -- (0+1.25, 2.5);
  \draw[blue,very thick] (1+1.25, 3.25) node[state, label={[label distance=4mm]above:$\color{blue}c_2$ \color{black}\rlap{$<$}}] {$c_2$} -- (1+1.25, 2.5);
  \draw[red, very thick] (2+1.25, 3.25) node[state, label={[label distance=4mm]above:$\color{red}c_3$}] {$c_3$} -- (2+1.25, 2.5);
  \draw (3+1.25, 3.25) node[state, label={[label distance=4mm]above:$\color{green}c_1$ \color{black}\rlap{$<$}}] {$+$} -- (3+1.25, 2.5);
  \draw (4+1.25, 3.25) node[state, label={[label distance=4mm]above:$\color{blue}c_2$ \color{black}\rlap{$<$}}] {$+$} -- (4+1.25, 2.5);
  \draw[red, very thick] (5+1.25, 3.25) node[state, label={[label distance=4mm]above:$\color{red}c_3$}] {$c_3$} -- (5+1.25, 2.5);
  \draw[red,  very thick] (7.5, 0) node[very thick, state] {$c_3$} -- (6.75, 0);
  \draw[blue, very thick] (7.5, 1) node[very thick, state] {$c_2$} -- (6.75, 1);
  \draw[red,  very thick] (7.5, 2) node[very thick, state, label={[label distance=2mm]above:{\color{black}$\sigma$}}] {$c_3$} -- (6.75, 2);
  \draw[very thick] (0.75,-.5) rectangle (6.75, 2.5);
\end{tikzpicture}

  \caption{Conventions for the grid and boundary data for a state with colored paths.
  The strict partition $\mu$ describes the occupied column numbers at the top boundary and $\sigma \in (\mathbb{Z}/m\mathbb{Z})^r$ denotes the colors at the right boundary by assigning the color $c_{\sigma_i}$ to row $i$ taking representatives in~$\{1, \ldots, m\}$.
  For readability, the boundary edges are marked by a circled label with its color $c_i$ or a plus sign if non-colored.
  The interior of the state is not shown, but represented by a white box.
  In the displayed example $\mu = (4,3,1)$ and $\sigma = (3,2,3)$.} 
  \label{fig:boundary-conventions}
\end{figure}
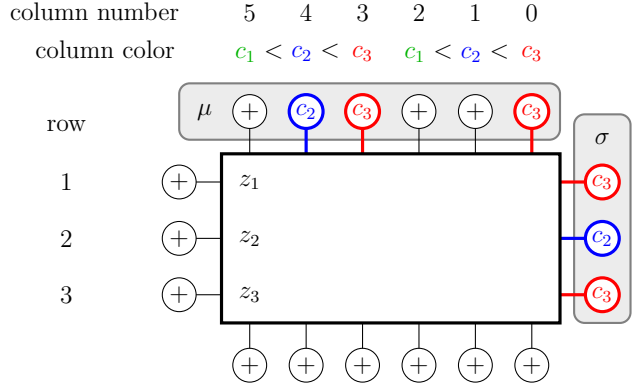

As one of our main results we show that the Iwahori specialization of our family of lattice models is equivalent to the so-called Iwahori ice model, and that the metaplectic specialization is equivalent to the so-called Delta metaplectic ice model.

\begin{maintheorem}
  \label{thm:specializations}
  The lattice model with color palette $\mathcal{P}_m$ defined by the weights in Table~\ref{tab:boltzmann-weights} has the following properties:
  \begin{enumerate}
    \item For $m=r$, its Iwahori specialization is equivalent to the Iwahori ice model defined in~\cite{BBBGIwahori} and its partition function computes a basis of Iwahori and parahoric Whittaker functions on $\GL_r(F)$ parametrized by the boundary data $\sigma$.
    \item For $m=n$, its metaplectic specialization is equivalent to the Delta metaplectic ice model defined in~\cite{mice, BBB} and its partition function computes a basis of spherical Whittaker functions on the metaplectic $n$-cover of $\GL_r(F)$ parametrized by the boundary data $\sigma$.
  \end{enumerate}
  In both cases the boundary data $\mu$ is related to the group argument of the Whittaker function.
\end{maintheorem}

For more details see Theorem~\ref{thm:specializations-prime} in Section~\ref{sec:specializations} which is a refinement of Theorem~\ref{thm:specializations}. The proof is postponed to Section~\ref{sec:specializations-proof} and paints a web of dualities between different lattice models in the literature as summarized by Figure~\ref{fig:web} at the end of Section~\ref{sec:specializations}.

Another of our main results is that each lattice model in the family is Yang-Baxter solvable, meaning that the Boltzmann weights satisfy a Yang-Baxter equation with an R-matrix originating from a specific quantum group.
This quantum group is a $\Phi$- and $\alpha_{i,j}$-dependent Drinfeld twist of the quantum group $U_q(\widehat{\mathfrak{gl}}(m|1))$.
This R-matrix can be visualized as another type of vertex, an R-vertex, adjoining only horizontal edges and dependent on two adjacent row parameters, e.g.\ $z_1$ and $z_2$.
Then the Yang-Baxter equations, which describe a relationship between partitions functions upon switching adjacent rows, can be illustrated as the equality of the partition functions of the following configurations where edge labels are circled.
\begin{equation}
  \label{eq:RTT}
  \begin{tikzpicture}[baseline, scale=0.75, every node/.style={scale=0.8}, round/.style={draw, circle, fill=white, inner sep=0pt, minimum size=16pt}, dot/.style={circle, fill, inner sep=0pt, minimum size=3pt}]
    \draw (-1,-1) node [round, label={[label distance=5mm]left:$z_1$}] {$a$} -- (0,0) node[dot, label={[label distance=2mm]above:$R$}] {} -- (1,1) node [round] {$i$} -- (2,1) node[dot, label=above right:${T}$] {} -- (3,1) node[round, label={[label distance=5mm]right:$z_1$}] {$d$};
    \draw (-1,1) node [round, label={[label distance=5mm]left:$z_2$}] {$b$} -- (1,-1) node [round] {$j$} -- (2,-1) node[circle, fill, inner sep=0pt, minimum size=3pt, label=above right:${T}$] {} -- (3,-1) node[round, label={[label distance=5mm]right:$z_2$}] {$e$};
    \draw (2,-2) node[round] {$f$} -- (2,0) node [round] {$k$} -- (2,2) node[round] {$c$};
  \end{tikzpicture}
  \qquad = \qquad
  \begin{tikzpicture}[xscale=-1, baseline, scale=0.75, every node/.style={scale=0.8}, round/.style={draw, circle, fill=white, inner sep=0pt, minimum size=16pt}, dot/.style={circle, fill, inner sep=0pt, minimum size=3pt}]
    \draw (-1,-1) node [round, label={[label distance=5mm]right:$z_2$}] {$e$} -- (0,0) node[dot, label={[label distance=2mm]above:$R$}] {} -- (1,1) node [round] {$l$} -- (2,1) node[dot, label=above right:${T}$] {} -- (3,1) node[round, label={[label distance=5mm]left:$z_2$}] {$b$};
    \draw (-1,1) node [round, label={[label distance=5mm]right:$z_1$}] {$d$} -- (1,-1) node [round] {$m$} -- (2,-1) node[circle, fill, inner sep=0pt, minimum size=3pt, label=above right:${T}$] {} -- (3,-1) node[round, label={[label distance=5mm]left:$z_1$}] {$a$};
    \draw (2,-2) node[round] {$f$} -- (2,0) node [round] {$n$} -- (2,2) node[round] {$c$};
  \end{tikzpicture}
\end{equation}
Here the boundary edges $a,b,c,d,e$ and $f$ are fixed while we sum over the internal edges $i,j,k$ and $l,m,n$ respectively.
Note that the row parameters $z_1$ and $z_2$ for the T-vertices are switched between the left-hand side and the right-hand side.

\begin{maintheorem}\label{thm:YBE}
  All members of the family of lattice models defined by the weights in Table~\ref{tab:boltzmann-weights} with a palette of $m$ colors are Yang-Baxter solvable with an R-matrix for the $U_q(\widehat{\mathfrak{gl}}(m|1))$ evaluation module, under a Drinfeld twist with parameters $\alpha_{i,j}$ and $\Phi$.

\end{maintheorem}

The theorem is further refined as Theorem~\ref{thm:YBE-prime} and proved in Section~\ref{sec:YBE}.

As a result, we may repeatedly apply the resulting Yang-Baxter equations in a familiar ``train argument'' to obtain the following recursive relations and functional equations on partition functions for any member of the family of lattice models.

\begin{maincorollary}
\label{maincor}
Given any $r \in \mathbb{N}$, $\mu \in \mathbb{Z}^r$ and $\sigma = (\sigma_1, \ldots, \sigma_r) \in (\mathbb{Z} / m \mathbb{Z})^r$, let $\mathfrak{S}_{\mu, \sigma}^m$ denote the corresponding system using Boltzmann weights from Table~\ref{tab:boltzmann-weights}. For any $i \in \{1,2, \ldots, r-1\}$ such that $\sigma_i \ne \sigma_{i+1}$,
\begin{equation}
  \label{eq:Demazure-op}
Z(\mathfrak{S}_{\mu, s_i \sigma}^m)(\z) = 
\frac{q^{-1} \alpha_{-\sigma_i, -\sigma_{i+1}}^{-1}}{1 - \mathbf{z}^{\alpha_i}} \left[ (q^2 \mathbf{z}^{\alpha_i} - 1) s_i + (1-q^2) 
\left\{\begin{smallmatrix*}[l] 1 &\ \sigma_i > \sigma_{i+1} \\ \mathbf{z}^{\alpha_i} &\ \sigma_{i+1} > \sigma_i\end{smallmatrix*}\right\}
\right]Z(\mathfrak{S}_{\mu, \sigma}^m)(\z).  
\end{equation}
If instead $\sigma_i = \sigma_{i+1}$, then we obtain the functional equation:
\begin{equation}
  Z(\mathfrak{S}_{\mu, \sigma}^m)(\z) = \frac{1 - q^2 \mathbf{z}^{\alpha_i}}{1 - q^2 \mathbf{z}^{-\alpha_i}} Z(\mathfrak{S}_{\mu, \sigma}^m)(s_i\z).
\end{equation}
Here $s_i$ denotes the simple reflection in $S_r$ which acts on $\sigma$ by permuting components and on the partition functions as elements of $\mathbb{C}(q)[\mathbf{z}]$ by permuting components of the spectral parameters $\mathbf{z}$.
\end{maincorollary}

The operator on the right-hand side of~\eqref{eq:Demazure-op} is a Demazure-type divided difference operator which will be discussed further in Sections~\ref{sec:p-adic-results} and~\ref{sec:demaction} and the corollary generalizes Proposition~7.1 in \cite{BBBGIwahori}. 
Because the proof proceeds similarly from Theorem~\ref{thm:YBE}, we omit it here. 

In particular, in Proposition~\ref{prop:sigma-recursion} we show how a partition function with right-boundary given by some $\sigma \in (\mathbb{Z}/m\mathbb{Z})^r$ determines any other partition function whose right-boundary is a permutation of $\sigma$.
Moreover, we show that if $\sigma$ has distinct parts, then there is a particularly simple partition function in this $\sigma$-orbit which has a single admissible state --- the ``ground state.'' This allows us to give a simple expression for any partition function with $\sigma$ having distinct parts in Theorem~\ref{thm:recursion}.
This was exploited in \cite{BBBGIwahori}, but now in light of Iwahori-metaplectic duality, can be applied to metaplectic groups where it gives some surprising results for metaplectic Whittaker functions. 

\subsection{Applications to \texorpdfstring{$p$}{p}-adic Whittaker functions}
\label{sec:p-adic-results}
As explained previously in Theorem~\ref{thm:specializations}, the lattice models introduced in this paper have an Iwahori and a metaplectic specialization which produce values of Iwahori Whittaker functions and metaplectic spherical Whittaker functions, respectively. This allows us to study both models and their generalization concurrently. As such, we are able to use tools introduced in the study of (nonmetaplectic) Iwahori and spherical Whittaker functions to derive new interesting results about metaplectic Whittaker functions. 

A distinguishing property of metaplectic groups is that Whittaker models for unramified principal series are not unique anymore. Let $F$ be local non-archimedean field with uniformizer~$\varpi$. 
In this paper we work with the same metaplectic $n$-fold cover of $\operatorname{GL}_r(F)$ that is considered in the papers~\cite{BBB,BBBGMetahori} and which is a central extension by the group of $n$:th roots of unity assumed to be in $F$.
For $\lambda = (\lambda_1,\ldots, \lambda_r)$, we denote by $\varpi^\lambda$ the diagonal element with entries $\varpi^{\lambda_1}, \ldots, \varpi^{\lambda_r}$. Let $\rho:=(r-1,\ldots, 1,0)$ and let $\mathbf{z} = (z_1, \ldots, z_r)$ with $z_i \in \mathbb{C}^\times$ denote an element of the complex dual torus $\hat T \cong (\mathbb{C}^\times)^r$.

For such a metaplectic cover, the dimension of the Whittaker model of the unramified principal series representation is $n^r$. Therefore, instead of working with a \emph{unique} unramified spherical Whittaker function as in the non-metaplectic case, we now deal with a basis of $n^r$ unramified spherical Whittaker functions. These Whittaker functions will be denoted by $\tilde\phi^\circ_\theta$ for $\theta \in (\mathbb{Z} / n \mathbb{Z})^r$ and are defined below in~\eqref{eq:spherical-Whittaker}.
These spherical metaplectic Whittaker functions $\tilde \phi^\circ_\theta$ are functions on the metaplectic cover which are completely determined by their values on the set of arguments given by $g = \varpi^{\rho-\mu}$ where $\mu$ is a strict partition.
We will show that this is the same $\mu$ that specifies the top boundary of the corresponding lattice model system.

Out of the large space of unramified metaplectic spherical Whittaker functions, there is a particular Whittaker function $\tilde\phi^\circ = \sum_\theta \tilde\phi^\circ_\theta$ that has been the focus of many papers including~\cite{McNamaraCS,McNamaraDuke, PatnaikPuskasIwahori} which produce combinatorial evaluations for this particular Whittaker function. However, the metaplectic specialization for the family of lattice models of this paper evaluates each spherical metaplectic Whittaker function $\tilde \phi^\circ_\theta$ separately.

The evaluation of these spherical metaplectic Whittaker functions in terms of lattice models has previously been carried out in~\cite{mice, StatementB, BBBGMetahori}, but the novelty of the approach in this paper is the unexpected duality to nonmetaplectic Iwahori Whittaker functions and their corresponding lattice models constructed in~\cite{BBBGIwahori}.
In particular, this leads to a new Demazure-type recursion for spherical metaplectic Whittaker functions with respect to the metaplectic parameters $\theta \in (\mathbb{Z}/n\mathbb{Z})^r$ in contrast to previous Demazure recursions for nonspherical Whittaker functions in terms of Iwahori parameters $w \in S_r$.

We will highlight two explicit and surprising consequences of the Iwahori-metaplectic duality by showing that two classes of the metaplectic Whittaker functions $\tilde \phi^\circ_\theta$ are equal to nonmetaplectic Iwahori (or spherical) Whittaker functions up to overall prefactors consisting of different Gauss sums. 

The first class contains $\tilde\phi^\circ_\theta$ when $\theta := (\theta_1,\cdots, \theta_r) \in (\mathbb{Z}/n\mathbb{Z})^r$ has all parts $\theta_i$ distinct. In this setup, we use ideas introduced by Brubaker, Bump and Licata~\cite{BBL} and our lattice model results to give an evaluation of $\tilde\phi^\circ_\theta$ in terms of certain Demazure-Lusztig operators. In fact, the same Demazure-Lusztig operators were used in~\cite{BBL} to compute nonmetaplectic Iwahori Whittaker functions denoted by $\phi_w$ enumerated by Weyl elements $w \in S_r$ as defined in Section~\ref{sec:specializations} below and we will show how the two types of Whittaker functions are related.

Let $T_i$ for $i \in \{1, \ldots, r\}$ be the Demazure-Whittaker operators defined in~\eqref{eq:demazurewhittaker} which depend on $\z$ and $q$.
The operators satisfy the braid relations of $S_r$ and were previously introduced in~\cite{BBL, BBBGIwahori} as part of recurrence relations for non-metaplectic Iwahori Whittaker functions.
For a reduced word $w = s_{i_1} s_{i_2} \cdots s_{i_k} \in S_r$ where $s_i$ is a simple reflection we may therefore define the associated Demazure-Whittaker operator $T_w := T_{i_1} T_{i_2} \cdots T_{i_k}$.
Using Theorem~\ref{thm:specializations} and Corollary~\ref{maincor} we show in Proposition~\ref{prop:metaplectic-single-recurrence} that the spherical metaplectic Whittaker functions satisfy similar recurrence relations as the non-metaplectic Iwahori Whittaker functions.
By comparing base cases in the simplifying situation when $\theta$ has distinct parts this leads to the following Iwahori-metaplectic duality.

\begin{maintheorem}\label{thm:DW}
  Let $\theta \in (\mathbb{Z}/n\mathbb{Z})^r$ have distinct parts (which assumes that $n \geqslant r$). 
  Then the spherical metaplectic Whittaker function $\z^{\rho - \theta} \tilde \phi^\circ_{\theta}(\z; \varpi^{\rho - \mu})$ is a polynomial in $\z^n$ which is nonzero only if the residue classes of $\mu \bmod n$ is a permutation of $\theta$, and in this case,
  \begin{equation}
    \label{eq:DW}
    \mathbf{y}^{\rho - \theta} \tilde \phi^\circ_\theta(\mathbf{y}; \varpi^{\rho - \mu}) = C \cdot T_{w} (T_{w'})^{-1} \z^{\lambda+\rho} = C' \cdot \z^\rho \phi_{w}(\z; \varpi^{-\lambda} w')
  \end{equation}
  with $(\mathbb{C}^\times)^r \ni \mathbf{y} := w_0 \z^{1/n}$ (componentwise) independent of the choice of branch.
  Here $\phi_{w}$ is a non-metaplectic Iwahori Whittaker function, $\rho = (r-1, r-2, \ldots, 0)$, $\lambda = \floor{\tfrac{\mu}{n}} - \rho$ (componentwise), $w_0$ is the long Weyl element, and $C, C'$ are explicit products of Gauss sums specified by $\mu$ and $\theta$ as detailed in the proof.
  Lastly, $w, w' \in S_r$ are the unique permutations such that $-w_0 \theta \equiv w \mathbf{c}$ and $-\mu \equiv w' \mathbf{c} \bmod n$ where $\mathbf{c}$ is a decreasing $r$-tuple of elements in $\{1, \ldots, n\}$ fixed by the residue classes of $\theta$ and $\mu$.
\end{maintheorem}

For more details and a proof, see Section~\ref{sec:demazure-metaplectic-whittaker}. 
We note that this is a purely representation theoretical statement which is unexpected from representation theory alone, but becomes natural in the language of lattice models where it waters down to an interchange of colors and supercolors.

The second class of Whittaker functions we study consists of $\tilde\phi^\circ_\theta$ where $\theta$ has all parts $\theta_i$ equal. In this setting we relate values of these metaplectic Whittaker functions to the values of non-metaplectic spherical Whittaker functions which are Schur polynomials $s_\lambda$ together with a deformed Weyl denominator factor, extending the classical Casselman-Shalika~\cite{CasselmanShalika} formula to this metaplectic case.

\begin{maintheorem}\label{thm:Tokuyama}
  Let $\theta \in (\mathbb{Z}/n\mathbb{Z})^r$ have entries which are all equal. Then the spherical metaplectic Whittaker function in the Whittaker model associated to $\theta$ has the form:
  \begin{equation}
    \label{eq:Tokuyama-CS}
    \tilde\phi^\circ_\theta(\z; \varpi^{\rho-\mu}) = \z^{\theta + n w_0 \rho - \rho}\prod_{\alpha > 0} (1-v\z^{n\alpha}) s_\lambda(\z^n)
  \end{equation}
  if $\mu = n(\lambda + \rho) + \theta$ for some partition $\lambda$, and $0$ otherwise. 
\end{maintheorem}

The proof, which follows from connecting the two corresponding lattice models is carried out in Theorem~\ref{thm:metaplectic-CS} in Section~\ref{sec:Tokuyama}.
There we also derive a similar expression for such $\tilde\phi^\circ_\theta$ but in terms of $s_{n(\lambda+\rho)-\rho}(\z)$ instead of $s_\lambda(\z^n)$.
Note that Theorem~\ref{thm:Tokuyama} is essentially a metaplectic version of the Tokuyama formula shown in~\eqref{eq:Tokuyama} below.

\begin{remark}
  The required change of variables $\z \mapsto \z^n$ in the duality shown in Theorems~\ref{thm:DW} and~\ref{thm:Tokuyama} agrees with the local Shimura correspondence (see~\cite{KazhdanPattersonShimura, SavinHecke, Savin}), where this same change of variables appears in the isomorphism of the associated Iwahori Hecke algebras.
\end{remark}

One of the uses of the spherical Whittaker function is in the Rankin-Selberg method, where it is crucial that values of such functions satisfy a Cauchy identity (\cite{JacquetShalikaI} Proposition~2.3; for more details see~\cite[Section 3 and Proposition 1]{Bump-Rallis_volume}).
Inspired by this use of the classical Whittaker function, we employ Theorem~\ref{thm:Tokuyama} to prove the following Cauchy identity for this class of metaplectic Whittaker functions in Section~\ref{sec:Tokuyama}. 

\begin{maincorollary}
\label{cor:Cauchy}
     Let $\theta \in (\mathbb{Z}/n\mathbb{Z})^r$ have entries which are all equal. Then the spherical metaplectic Whittaker function satisfies the following Cauchy type identity
\[ 
\sum_{\mu}  \tilde\phi^\circ_\theta(\mathbf{x}; \varpi^{\rho-\mu}) \tilde\phi^\circ_\theta(\mathbf{y}; \varpi^{\rho-\mu}) q^{-|\mu|s} = 
(\mathbf{xy})^{\theta - n \rho + w_0 \rho}\prod_{\alpha > 0} (1-v\mathbf{x}^{n\alpha})(1-v\mathbf{y}^{n\alpha})  \prod_{\mathclap{1 \leq i,j \leq r}} (1-x_i^n y_j^n q^{-s})^{-1}. 
\]       
\end{maincorollary}

Lascoux~\cite{LascouxDoubleCrystal} generalized the Cauchy identity to expand
$\prod_{(i,j)\in \text{YD}(\lambda)}(1-\alpha_i\beta_j)^{-1}$ (where $\text{YD}(\lambda)$ is the Young diagram of a partition $\lambda$) in terms of Demazure characters, and it is probable that such an expansion can be applied to Rankin-Selberg integrals of Iwahori Whittaker functions. If so, the above considerations show that there should be similar metaplectic Rankin-Selberg integrals.

Theorems~\ref{thm:DW} and ~\ref{thm:Tokuyama} relate certain metaplectic Whittaker functions depending on an element $\theta \in (\mathbb{Z} / n \mathbb{Z})^r$ to their non-metaplectic counterparts. The element $\theta$ decides what type of non-metaplectic Whittaker function appears on the other side of our duality: if $\theta$ has all parts equal as in Theorem~\ref{thm:Tokuyama}, then we obtain spherical Whittaker functions, while if the element $\theta$ has all parts distinct, we obtain Iwahori Whittaker functions. The proofs of such dualities work by relating the lattice models corresponding to the two sides. 

One might ask if something can be said in between these two extremal cases.
On the nonmetaplectic side, the spherical and Iwahori Whittaker functions are the two extremal cases of the more general parahoric Whittaker functions which are defined with respect to a standard parabolic subgroup $P$ of $\GL_r$.
In the Iwahori case $P = B$ the Borel subgroup and in the spherical case $P = \GL_r$.
As discussed in Section~\ref{sec:specializations} the Iwahori specialization of our family of lattice models do in fact obtain all the different types of parahoric functions where $P$ is determined by which colors (with multiplicity) appear in the boundary conditions.
More specifically, $P$ is determined by the fact that the Weyl group of its Levi subgroup is the stabilizer of the decreasing $r$-tuple $\mathbf{c}$ from Theorem~\ref{thm:DW} describing the colors of the boundary.

By Theorem~\ref{thm:specializations} if we instead had considered the metaplectic specialization of our family of lattice models with the same boundary condition with repeated colors we would instead have obtained a spherical metaplectic Whittaker function with metaplectic parameter $\theta$ with some repeated entries.

However, it is still an open question if these two specializations can be related directly when we have some repeated colors. 
That is, if spherical metaplectic Whittaker functions $\tilde \phi_\theta^\circ$ where $\theta$ has some repeated entries can be expressed in terms of nonmetaplectic parahoric Whittaker functions $\phi^P_w$ similar to Theorems~\ref{thm:DW} and~\ref{thm:Tokuyama} for distinct entries.
See Remark~\ref{rem:parahoric-conjecture} for some ideas in this direction.

\subsection{Results for infinite systems}
We will show that the row transfer matrices in the Iwahori models with palette $\mathcal{P}_r$ can be regarded
as $U_q (\widehat{\mathfrak{s}\mathfrak{l}}(r))$-module endomorphisms of the
$q$-Fock space of Kashiwara, Miwa and Stern~{\cite{KMS}}. Indeed, we will
prove such a fact for the entire family of interpolating models with palette $\mathcal{P}_n$ of any size $n$.\footnote{Here we use $n$ to denote a generic palette size instead of $m$ which has another historical connotation for Fock spaces.}

To do this, we generalize results of our earlier
paper~{\cite{BBBGVertex}}. This makes use of the {\emph{$q$-Fock space}} of
Kashiwara, Miwa and Stern~{\cite{KMS}}, which is the space $\mathfrak{F}_q$
of semi-infinite monomials
\[ u = u_{i_m} \wedge u_{i_{m - 1}} \wedge \cdots \]
in terms of a $q$-dependent quantum wedge defined in Section~\ref{sec:Fock-space} and where $i_m > i_{m - 1} > \cdots$\quad are integers and $i_k = k$ for $k$ sufficiently
negative. Here $m$ can be any integer, sometimes called the {\it level}. We may define the \emph{energy} of
the monomial to be $\sum_k{i_k-k}$. Thus the vacuum state denoted $| m \rangle$,
for which $i_k=k$ for all $k \leq m$, has level $m$ and energy zero while other states of level $m$ have positive energy.

In~\cite{KMS}, an action of $U_q(\widehat{\mathfrak{sl}}(n))$ on $\mathfrak{F}_q$ is defined via
the tensor product of so-called natural modules for $U_q(\widehat{\mathfrak{sl}}(n))$. In \cite{BBBGVertex}, we show how to construct a generalized
$q$-Fock space such that it remains a module for any Drinfeld twist $U_q^\alpha(\widehat{\mathfrak{sl}}(n))$ with parameters $\alpha_{i,j}$ as above of our original Hopf algebra.
We denote the resulting Fock space by $\mathfrak{F}^\alpha_q$.

We say that an operator on $\mathfrak{F}^\alpha_q$ is
{\emph{right-moving}} (resp. left-moving) if it replaces $u$ by a sum of terms of
the form $u_{j_m} \wedge u_{j_{m - 1}} \wedge \cdots$ with $j_k \leqslant i_k$
(resp. $j_k \geqslant i_k$). Thus right-moving operators lower the energy
and left-moving operators raise it. 

The Fock space carries a natural action of a Heisenberg Lie algebra 
whose generators are denoted $B_k$
in {\cite{KMS}} and $J_k$ in {\cite{BBBGVertex}}. The operator $J_0$ acts by the scalar $m$ on Fock space vectors of level $m$. For $k \ne 0$, these so-called ``current operators'' $J_k$ 
acts as an endomorphism on level $m$ Fock space vectors by
\begin{multline}
  \label{eq:J-def}
  J_k  (u_{i_m} \wedge u_{i_{m-1}} \wedge u_{i_{m-2}} \wedge \cdots) ={} \\ {}= (u_{i_m - nk}
  \wedge u_{i_{m-1}} \wedge u_{i_{m-2}} \wedge \cdots) + (u_{i_m} \wedge u_{i_{m-2} - nk}
  \wedge u_{i_{m-2}} \wedge \cdots) + \cdots \hspace{0.17em} .
\end{multline}
It is shown in {\cite[Lemma 2.1]{KMS}} that this defines an action on the quantum Fock space consistent with the
quantum wedge, in particular resulting in a finite sum of wedges. Note that this action is expressible in a form independent of the Drinfeld twist.
The sets $\{ J_k \mid k \geqslant 0 \}$ and $\{ J_{- k} \mid k \geqslant 0 \}$ are commuting
families of $U_q^\alpha (\widehat{\mathfrak{s}\mathfrak{l}} (n))$-module
endomorphisms of $\mathfrak{F}_q^\alpha$, consisting of (respectively) right-moving
and left-moving operators. They also commute with the
action of $U_q^\alpha (\widehat{\mathfrak{s}\mathfrak{l}} (n))$. From these current operators, we build natural Hamiltonian operators
\begin{equation} 
\label{eq:hamiltonian}
H^{\pm} (z) = \sum_{k = 1}^{\infty} (1 - q^{2k}) z^{\pm n k} J_{\pm k}.
\end{equation}

On the other hand, by associating a spin to each basis vector in the natural module, we may view each element of 
Fock space as an infinite sequence of spins. Arranging these on an infinite string of vertical edges in a lattice model, then any row-transfer matrix $T(z)$ (i.e., one-row partition function) for this system with infinitely many columns may be viewed as an endomorphism of the Fock space. For a complete description of this passage from Fock space to lattice models, see Section~\ref{sec:Fock-space}.

Generalizing \cite{BBBGVertex}, we show the following equality of right-moving operators on $\mathfrak{F}_q^\alpha$.

\begin{maintheorem}\label{thm:T-hamiltonian}
  The $U^\alpha_q(\widehat{\mathfrak{sl}}(n))$ quantum Fock space operator $H^+(z)$ of \eqref{eq:hamiltonian} is a Hamiltonian associated to the one-row transfer matrix $T(z)$ for the infinite system defined by the weights in Table~\ref{tab:boltzmann-weights} with palette $\mathcal{P}_n$.
  Specifically,
  \begin{equation}
    \label{eq:T-hamiltonian}
    T(z) = \Bigl( -\frac{\Phi}{q} \Bigr)^{J_0+1} e^{H^+(z)}.
  \end{equation}
\end{maintheorem}

In \cite{BBBGVertex}, we proved a special case of this theorem for the particular choice of Drinfeld twist needed in the
{\it metaplectic specialization} described above. Moreover we defined two systems of solvable lattice models on infinite grids,
called {\emph{Delta (metaplectic) ice}} and {\emph{Gamma (metaplectic) ice}}, whose row-transfer matrices
are right-moving (energy lowering) and left-moving 
(energy raising) operators. Let
$T_{\Delta} (z)$ and $T_{\Gamma} (z)$ be these respective row-transfer
matrices. They depend on a spectral parameter $z \in \mathbb{C}^{\times}$ that
is built into the Boltzmann weights. 
The main result of~{\cite{BBBGVertex}} is
then:
\begin{equation}
  \label{eq:tghvo}  
  \begin{gathered}
    T_{\Delta} (z) = e^{H^+ (z)}, \qquad T_{\Gamma} (z) =
    e^{H^- (z)}.
  \end{gathered}
\end{equation}
In the application to metaplectic groups, $v$ is the reciprocal of the residue
cardinality, and in the quantum group the deformation parameter $q =
\sqrt{v}$. See Hardt~{\cite{Hardt}} for other versions of this result.

The result (\ref{eq:tghvo}) has several implications and ramifications.
\begin{enumerate}
  \item \label{itm:endomorphism} It shows that $T_{\Gamma}$ and $T_{\Delta}$ are $U_q^\alpha
  (\widehat{\mathfrak{s}\mathfrak{l}} (n))$-module endomorphism of
  $\mathfrak{F}_q^\alpha$, because the $J_k$ are $U_q^\alpha
  (\widehat{\mathfrak{s}\mathfrak{l}} (n))$ equivariant.
  
  \item \label{itm:commute} It demonstrates the identities $T_{\Delta} (z) T_{\Delta} (w) = T_{\Delta} (w)
  T_{\Delta} (z)$ and $T_{\Gamma} (z) T_{\Gamma} (w) = T_{\Gamma} (w)
  T_{\Gamma} (z)$, and it also shows that $T_{\Delta} (z) T_{\Gamma} (w) =
  (\ast) T_{\Gamma} (w) T_{\Delta} (z)$, where $(\ast)$ is a computable
  constant (\cite[(7.8)]{BBBGVertex}) may be proved by the Yang-Baxter
  equation. However (\ref{eq:tghvo}) gives a different proof using the
  Heisenberg relations of the operators $J_k$. This puts these results in a
  different context of the boson-fermion correspondence~{\cite{LamBoson,LamRibbon}}.
  
  \item \label{itm:vertex-op} The operator $T_{\Gamma} (z) T_{\Delta} (z)$ then may be recognized as
    a \emph{vertex operator}. Compare, for example {\cite[(2.2)]{JingBoson}}. Thus we
  may describe the row transfer matrices $T_{\Delta} (z)$ and $T_{\Gamma} (z)$
  as {\emph{half-vertex operators}}. The identities in point~\ref{itm:commute} are recognized
  as {\emph{locality conditions}} that appear in the definition of a vertex
  algebra (\cite{KacBeginners}).
  
  \item \label{itm:LLT} Due to the relationship with the boson-fermion correspondence
  mentioned in point~\ref{itm:commute}, we may recognize the polynomials $\langle \mu
  |T_{\Delta} (z_r) \cdots T_{\Delta} (z_1) | \lambda \rangle$ as super-LLT
  polynomials, defined in Lam~{\cite{LamRibbon}}. These are symmetric
  polynomials closely related to metaplectic Whittaker functions.

  \item \label{itm:vertical} As mentioned above, in the usual paradigm for solvable lattice models, edges of the grid
  are associated with modules for a quantum group. Interpreting the row transfer 
  matrices $T_{\Delta} (z)$ and $T_{\Gamma} (z)$ as 
  $U_q^\alpha (\widehat{\mathfrak{s}\mathfrak{l}} (n))$-module endomorphisms of
  the quantum group module $\mathfrak{F}_q^\alpha$ is a step towards a quantum
  group interpretation of the vertical edges which is an open question.
\end{enumerate}
  
Theorem~\ref{thm:T-hamiltonian} shows that the identity (\ref{eq:tghvo}) can be
generalized to the interpolated models, and hence to the Iwahori
models. Therefore the four points \ref{itm:endomorphism}--\ref{itm:LLT} above also apply to the models
representing Iwahori Whittaker functions: the row transfer matrices
may be regarded as half-vertex operators on a $q$-Fock space,
and they are $U_q(\widehat{\mathfrak{sl}}(r))$ module endomorphisms
of it. They therefore have relationships with supersymmetric
LLT polynomials.

Let us return to point \ref{itm:vertical}. Ideally, and in many examples such as
the six-vertex model, we may associate a module of a quantum group to
every edge of the grid. The Boltzmann weights at a
vertex of the grid are then to be the matrix coefficients in a morphism 
$V \otimes U \rightarrow U \otimes V$, where $U$ and $V$ are the modules
associated with the vertical and horizontal edges meeting at the vertex; the
Yang-Baxter equation reflects the braiding of the category. For the models
at hand there are two equivalent versions of the model: an expanded
``monochrome'' version and a fused version. The R-matrix for these models
tells us that the quantum group in this case is $U_q
(\widehat{\mathfrak{g}\mathfrak{l}} (1| n))$. For the horizontal edges, the
corresponding module will be the standard evaluation module $\mathbb{C}^{1|
n}_z$ depending on the parameter $z \in \mathbb{C}^{\times}$. But for the
vertical edges, if we use the monochrome version of the models alluded to above, the vertical
edges have no clear interpretation as $U_q
(\widehat{\mathfrak{g}\mathfrak{l}} (1| n))$-modules, only as $U_q
(\widehat{\mathfrak{g}\mathfrak{l}} (1|1))$-modules. The expectation is that
the edges in the fused module will be identified with a Kac module for
$\mathfrak{g}\mathfrak{l} (1| n)$. Although we will not explore this point
in this paper, identifying $\mathfrak{F}_q^\alpha$ as a $U_q^\alpha
(\widehat{\mathfrak{s}\mathfrak{l}} (n))$-module appears consistent with
this conjectural identification. As a related point, we conjecture that the
$q$-Fock space $\mathfrak{F}_q^\alpha$ has the structure of a 
$U_q^\alpha (\widehat{\mathfrak{s}\mathfrak{l}} (1| n))$-module. (Originally,
it was defined in~\cite{KMS} as a
$U_q^\alpha (\widehat{\mathfrak{s}\mathfrak{l}} (n))$-module.) 

\section{The family of lattice models}
\label{sec:family}

In this section we define the family of lattice models and describe the fusion process in more detail. We also explain how these lattice models simultaneously generalize models considered in previous papers, in particular Iwahori ice of~\cite{BBBGIwahori} and Delta metaplectic ice of~\cite{mice, BBB, StatementB}. 

The family of lattice models is parametrized by an integer $m$ corresponding to the size of the color palette and nonzero complex parameters $\Phi$ and $\alpha_{i,j}$ with $i,j \in \{1, \ldots, m\}$ such that $\alpha_{i,j} \alpha_{j,i} = \alpha_{i,i} = 1$.
For convenience we extend the parameters $\alpha_{i,j}$ to all $i,j \in \mathbb{Z}$ by $m$-periodicity in both indices.
We denote the palette of ordered colors $c_1 < c_2 < \cdots < c_m$ by $\mathcal{P}_m$.
Throughout the paper we will use the notation $\res_m(x)$ for the least nonnegative residue of $x \bmod m$ and $\res^m(x)$ for the least strictly positive residue of $x \bmod m$.
That is 
\begin{equation}
  \label{eq:res}
  \{0,1,\ldots,m-1\} \ni \res_m(x) \equiv x \equiv \res^m(x) \in \{1,2,\ldots, m\}.
\end{equation}

We will in this section only consider finite grids.
Let $r$ be the number of rows and $N$ the number of columns labeled from the top down $1,\ldots,r$ and from the right to left $0, \ldots, N-1$. The set of admissible states depends on $r$, $N$ and the palette size $m$ but are independent of~$\Phi$ and~$\alpha_{i,j}$.
The states are described by assigning colors to the edges of the grid according to given vertex configurations shown in Table~\ref{tab:boltzmann-weights}.
To each row $i$ we also assign a nonzero complex parameter $z_i$.

As mentioned in Section~\ref{sec:results} the horizontal edges can be assigned either no color (making it unoccupied) or any one of the colors in the palette $\mathcal{P}_m$. Vertical edges can only be assigned either no color (making it unoccupied) or a predetermined color in $\mathcal{P}_m$ which is decided by the column number in a repeating pattern such that $c_1, \ldots, c_m$ are the colors for the last columns to the right.
The number of columns $N$ may be enlarged to a multiple of $m$ without affecting the admissible states (up to a trivial identification) or their Boltzmann weights, and the partition function is therefore also unchanged. 

The boundary conditions are given by a strict partition $\mu$ with $r$ nonnegative parts determining the occupied vertical edges on the top boundary, and an element $\sigma \in (\mathbb{Z}/m\mathbb{Z})^r$ determining the color indices for the horizontal edges on the right boundary which are all occupied.
The left and bottom boundaries consists of unoccupied edges which we will sometimes also denote by plus signs.
See Figure~\ref{fig:boundary-conventions}.
We denote the ensemble of admissible states with these boundary conditions by $\mathfrak{S}^m_{\mu,\sigma}$.

The partition function $Z(\mathfrak{S})$ for any ensemble $\mathfrak{S}$ of admissible states is computed from the Boltzmann weights for the vertex configurations shown in Table~\ref{tab:boltzmann-weights} using~\eqref{eq:partition-function} where $z$ is replaced by each row parameter $z_i$ at row $i$.
These Boltzmann weights, and therefore also the partition function, depend on $\Phi$ and $\alpha_{i,j}$.
Note that the vertex configurations and weights in Table~\ref{tab:boltzmann-weights} depend on the predetermined column colors $c_j$.

\begin{lemma}
  \label{lem:non-zero}
  The partition function $Z(\mathfrak{S}^m_{\mu,\sigma})(\z)$ is non-zero only if the residue classes of $-\mu \bmod m$ are a permutation of $\sigma \in (\mathbb{Z}/m\mathbb{Z})^r$.
\end{lemma}

\begin{proof}
  The condition on the residue classes of $\mu$ is by definition equivalent to the condition that the colors of the top boundary should be a permutation of the colors on the right boundary which is necessary for admissible states.
\end{proof}

\subsection{Equivalent models by fusion}
\label{sec:fusion}
There is an equivalent description of the lattice model where the vertex weights and possible vertical edge colors do not depend on the column number.
Because the column dependence is $m$-periodic it is natural to group columns into blocks of $m$ columns starting with columns $0$ to $m-1$ which we label as block $0$ increasing to the left.
We will now describe a process called \emph{fusion} which groups together such a one-row block of vertices into a single fused vertex, and the resulting fused vertex configurations and weights do not depend on the new (fused) block column numbers.
The fused model then describes colored paths in a grid where paths of \emph{different} colors may overlap along vertical edges.
In contrast, because of how the colors for the vertical edges in the unfused model are constrained to the column color, we will also call the unfused models \emph{monochrome}.

For given boundary conditions of a one-row block there is at most one admissible state.
We assign the Boltzmann weight of the corresponding fused vertex to be the partition function (or weight) of this state or zero if there is no admissible state.
More specifically, let $A, C \in \mathcal{P}_n \cup \{\oplus\}$ where $\oplus$ denotes an unoccupied (uncolored) edge describe the left and right horizontal edges of a one-row block. Let $b_1, \ldots, b_m, d_1, \ldots, d_m \in \{\oplus, \ominus\}$ describe top and bottom vertical edges in the order shown in~\eqref{eq:unfusedequalfused} below where $\oplus$ again signifies an unoccupied edge and $\ominus$ denotes an occupied edge with the predetermined column-dependent color $c_j$.
Then there is at most one possible configuration for the interior edges in the unfused one-row block as determined by the vertex configurations in Table~\ref{tab:boltzmann-weights} and the weight of the corresponding state is the product of these vertex Boltzmann weights.
Let $B := \{c_j : b_j = \ominus\} \subseteq \mathcal{P}_m$ and $D := \{c_j : d_j = \ominus\} \subseteq \mathcal{P}_m$ denote the top and bottom edge configurations of the corresponding fused vertex.
We then have the following relations for the Boltzmann weights
{\vspace{1em}
\begin{equation}
\label{eq:unfusedequalfused}
  \operatorname{wt} \left(
    \begin{tikzpicture}[scale=0.9, baseline=-1mm, every node/.append style={scale=1}, round/.style={draw, circle, fill=white, inner sep=0pt, minimum size=18pt}, dot/.style={circle, fill, inner sep=0pt, minimum size=3pt}] 
      \useasboundingbox (-1.5,-1.5) rectangle (1.5,1.5);
      \draw [thick] (0,-1) node[round] {$D$} -- (0,0) node[dot] {} -- (0,1) node[round] {$B$};
      \draw [thick] (-1,0) node[round] {$A$} -- (1,0) node[round] {$C$};
      \node at (0,-1.85) {\small fused vertex};
    \end{tikzpicture}
   \right) :=
   \operatorname{wt} \left( 
    \begin{tikzpicture}[scale=0.9, baseline=-1mm, every node/.append style={scale=1}, round/.style={draw, circle, fill=white, inner sep=0pt, minimum size=18pt}, dot/.style={circle, fill, inner sep=0pt, minimum size=3pt}] 
      \useasboundingbox (-1.5,-1.5) rectangle (4.5,1.5);
      \draw [thick] (0,-1) node[round] {$d_1$} -- (0,0) node[dot] {} -- (0,1) node[round, label={[label distance=1mm]above:$c_1$}] {$b_1$};
      \draw [thick] (1,-1) node[round] {$d_2$} -- (1,0) node[dot] {} -- (1,1) node[round, label={[label distance=1mm]above:$c_2$}] {$b_2$};
      \draw [thick] (3,-1) node[round] {$d_m$} -- (3,0) node[dot] {} -- (3,1) node[round, label={[label distance=1mm]above:$c_m$}] {$b_m$};
      \draw [thick] (-1,0) node[round] {$A$} -- (4,0) node[round] {$C$};
      \draw [draw=none] (2,-1) node {$\cdots$} -- (2,0) node[fill=white] {$\cdots$} -- (2,1) node[label={[label distance=2mm]above:$\cdots$}] {$\cdots$};
      \node at (1.5,-1.85) {\small block of unfused vertices};
    \end{tikzpicture}
   \right) 
\end{equation}
\vspace{1em}}

An important difference between the two descriptions is thus what data the vertical edges are assigned: either binary information (i.e. whether the edge is colored with the prescribed column color or not) or a subset of the palette describing which colors are assigned to the edge.
For the fused vertical edges $B$ and $D$ we will also often identify $\emptyset$, (i.e.\ an unoccupied edge) with~$\oplus$.
The two descriptions have different advantages. For example, for Iwahori ice it is the fused description that has a more natural connection to Iwahori Whittaker functions, while for metaplectic ice it is the unfused description which is more natural.

\subsection{Equivalent models using supercolors}
\label{sec:supercolors}
In the next subsection we will compare a member of our family of lattice models with several metaplectic ice models.
The original metaplectic ice models were not described by paths with a single, global attribute such as color. Instead, horizontal edges along a path were assigned elements in $\mathbb{Z}/n\mathbb{Z}$ called charges which increase along the path.
In~\cite{BBBGMetahori} we developed another formulation based on attributes that are constant along the path.
In fact, the models in~\cite{BBBGMetahori} compute the more general metaplectic Iwahori Whittaker functions (and not only spherical) by introducing two sets of paths: one set taking colors from an ordered palette $\mathcal{P}_r$ describing the Iwahori data discussed in Section~\ref{sec:Iwahori} below and the other set taking colors from another ordered palette $\bar{\mathcal{P}}_n$ describing the metaplectic data earlier captured by the charge attributes.
We gave the elements of the second palette $\bar{\mathcal{P}}_n$ the name \emph{supercolors} because of how these connect to the odd part of the super quantum group $U_q(\widehat{\mathfrak{gl}}(r|n))$ associated to the lattice model in~\cite{BBBGMetahori}.

In this subsection we will give an equivalent description of our family of lattice models which uses $m$ supercolors instead of $m$ colors.
An important difference between colors and supercolors are how we assign them to each column of vertices.
Recall that in the unfused model a vertical edge may only be assigned no color or a predetermined color depending on the column number.
Indeed, as seen in Figure~\ref{fig:boundary-conventions}, colors are arranged in blocks of $c_1 < c_2 < \cdots < c_m$ increasing to the right, while supercolors will be increasing to the left.
It will be convenient to label the supercolors by indices starting from zero instead as $\bar c_{m-1} > \cdots > \bar c_1 > \bar c_0$.

\begin{remark}
In~\cite{BBBGMetahori}, the supercolored paths were left-moving, exiting at the left boundary, but this is not a feature of supercolored paths \emph{per~se}, but rather a feature of that model being a generalization of Gamma metaplectic ice instead of Delta metaplectic ice where the supercolored paths are right moving.
The paths in our family of lattice models will still be right moving in both the description using colors as well as the one using supercolors.
\end{remark}

By relabeling $c_i$ as $\bar c_{m-i}$ in our family of lattice models with colored paths defined earlier in Section~\ref{sec:family} we obtain an equivalent supercolored version where the column blocks are labeled $\bar c_{m-1}, \bar c_{m-2}, \ldots, \bar c_0$ from left to right and the weights are as in Table~\ref{tab:super-boltzmann-weights} where we draw supercolored paths with dotted lines.

\begin{table}[htpb]
  \centering
  \caption{Supercolor version of Table~\ref{tab:boltzmann-weights} of T-vertex configurations and their Boltzmann weights where the supercolored edges are drawn with dotted lines.}
  \label{tab:super-boltzmann-weights}
\begin{tabular}{|c|c|c|c|c|c|}
  \hline
  $\texttt{a}_1$ & $\texttt{a}_2$ & $\texttt{b}_1$ & $\texttt{b}_2$ & $\texttt{c}_1$ & $\texttt{c}_2$ \\\hline\hline
  
  \begin{tikzpicture} 
    \draw [] (0,-1) -- (0,1) node[label=above:$\bar c_j$] {};
    \draw [] (-1,0) -- (1,0);
  \end{tikzpicture}
 
  &
  
  \begin{tikzpicture}
    \draw [blue, ultra thick, dotted] (0,-1) -- (0,1) node[label=above:$\bar c_j$] {}; 
    \draw [red, ultra thick, dotted] (-1,0) -- (1,0) node[label=right:$\bar c_i$] {}; 
  \end{tikzpicture}
 
  &
  
  \begin{tikzpicture}
    \draw [] (-1,0) -- (1,0); 
    \draw [red, ultra thick, dotted] (0,-1) -- (0,1) node[label=above:$\bar c_j$] {};
  \end{tikzpicture}
 
  &
  
  \begin{tikzpicture}
    \draw [] (0,-1) -- (0,1) node[label=above:$\bar c_j$] {};
    \draw [red, ultra thick, dotted] (-1,0) -- (1,0) node[label=right:$\bar c_i$] {}; 
  \end{tikzpicture} 
 
  &
  
  \begin{tikzpicture}
    \draw [] (1,0) -- (0,0) -- (0,1) node[label=above:$\bar c_j$] {};
    \draw [red, ultra thick, dotted] (0,-1) -- (0,0) -- (-1,0);
  \end{tikzpicture}
  
  &
  
  \begin{tikzpicture}
    \draw [] (-1,0) -- (0,0) -- (0,-1);
    \draw [red, ultra thick, dotted] (0,1) node[label=above:$\bar c_j$] {} -- (0,0) -- (1,0);
  \end{tikzpicture}
  
  \\\hline
  $1$ & \scriptsize $\Phi \times
  \begin{cases}
    q z & i = j\\
    \alpha_{i, j} & i \neq j
    \end{cases}
    $ & $-\frac{\Phi}{q}$ & \scriptsize $\begin{cases}
    z & i = j\\
    1 & i \neq j
  \end{cases}$ & $ -\frac{\Phi}{q} (1 - q^2) z$ & $1$ \\\hline
\end{tabular}
\end{table}

For the boundary data we keep the partition $\mu$ to describe the top boundary, and continue to use an element of $(\mathbb{Z}/m\mathbb{Z})^r$ to describe the right boundary.
However, we will let it index the supercolors instead of the colors and therefore denote it $\theta$ instead of $\sigma$ with row $i$ assigned the supercolor $\bar c_{\res_m(\theta_i)}$.
See Figure~\ref{fig:super-boundary-conventions}.
Since the translation from colors to supercolors is the relabeling of $c_i$ to $\bar c_{m-i}$, the relationship to the color boundary data $\sigma$ is then $\theta = - \sigma$.
We will use the notation 
\begin{equation}
  \label{eq:S-bar-notation}
  \bar{\mathfrak{S}}^m_{\mu,\theta} := \mathfrak{S}^m_{\mu,-\theta} = \mathfrak{S}^m_{\mu,\sigma}, 
\end{equation}
where the bar indicates that the boundary data $\theta$ is given in terms of supercolors.

\begin{figure}[htpb]
  \centering

\tikzstyle{state}=[circle, draw, fill=white, minimum size=16pt, inner sep=0pt, outer sep=0pt]
\tikzstyle{selection}=[thick, color=gray, fill=lightgray!30, rounded corners, text=black]
\tikzstyle{selection-line}=[selection, rounded corners=0pt, fill=none, shorten <=2pt, shorten >=2pt]

\begin{tikzpicture}[baseline=0.5cm, scale=0.75, every node/.append style={scale=0.8}]
  
  \draw[selection] (7,-.5) rectangle (8,3.2);
  \draw[selection] (0,3.75) rectangle (6.75,2.75);

  \draw (0,0) node[state] {$+$} -- (0.75,0) node[label=right:$z_3$] {};
  \draw (0,1) node[state] {$+$} -- (0.75,1) node[label=right:$z_2$] {};
  \draw (0,2) node[state] {$+$} -- (0.75,2) node[label=right:$z_1$] {};
  \foreach \x in {0,...,5}{
    \draw (\x+1.25, -1.25) node[state] {$+$} -- (\x+1.25,-.5);
    \node at (6.25-\x,5) {$\x$};
  }

  \node[align=right] at (1.25, 5) {\llap{column number\hspace{1cm}}\phantom{$c_1$}};
  \node at (-2,3) {row};
  \node at (-2,2) {$1$};
  \node at (-2,1) {$2$};
  \node at (-2,0) {$3$};
  
  \draw (0+1.25, 3.25) node[state, label={[label distance=2mm]left:{\color{black}$\mu$}}, label={[label distance=4mm]above:\llap{\color{black}column supercolor\hspace{1cm}}$\color{green}\bar c_2$ \color{black}\rlap{$>$}}] {$+$} -- (0+1.25, 2.5);
  \draw[blue, densely dotted, very thick] (1+1.25, 3.25) node[state, label={[label distance=4mm]above:$\color{blue}\bar c_1$ \color{black}\rlap{$>$}}] {$\bar c_1$} -- (1+1.25, 2.5);
  \draw[red, densely dotted, very thick] (2+1.25, 3.25) node[state, label={[label distance=4mm]above:$\color{red}\bar c_0$}] {$\bar c_0$} -- (2+1.25, 2.5);
  \draw (3+1.25, 3.25) node[state, label={[label distance=4mm]above:$\color{green}\bar c_2$ \color{black}\rlap{$>$}}] {$+$} -- (3+1.25, 2.5);
  \draw (4+1.25, 3.25) node[state, label={[label distance=4mm]above:$\color{blue}\bar c_1$ \color{black}\rlap{$>$}}] {$+$} -- (4+1.25, 2.5);
  \draw[red, densely dotted, very thick] (5+1.25, 3.25) node[state, label={[label distance=4mm]above:$\color{red}\bar c_0$}] {$\bar c_0$} -- (5+1.25, 2.5);
  \draw[red, densely dotted, very thick] (7.5, 0) node[very thick, state] {$\bar c_0$} -- (6.75, 0);
  \draw[blue, densely dotted, very thick] (7.5, 1) node[very thick, state] {$\bar c_1$} -- (6.75, 1);
  \draw[red, densely dotted, very thick] (7.5, 2) node[very thick, state, label={[label distance=2mm]above:{\color{black}$\theta$}}] {$\bar c_0$} -- (6.75, 2);
  \draw[very thick] (0.75,-.5) rectangle (6.75, 2.5);
\end{tikzpicture}

  \caption{Conventions for the grid and boundary data for a state with supercolored paths.
  The strict partition $\mu$ describes the occupied column numbers at the top boundary and $\theta \in (\mathbb{Z}/m\mathbb{Z})^r$ denotes the supercolors at the right boundary by assigning the supercolor $\bar c_{\res_m(\theta_i)}$ to row $i$ taking representatives in~$\{0, \ldots, m-1\}$.
  For readability, the boundary edges are marked by a circled label with their supercolor $\bar c_i$ or with a plus sign if its not supercolored.
  The interior of the state is not shown, but represented by a white box.} 
  \label{fig:super-boundary-conventions}
\end{figure}
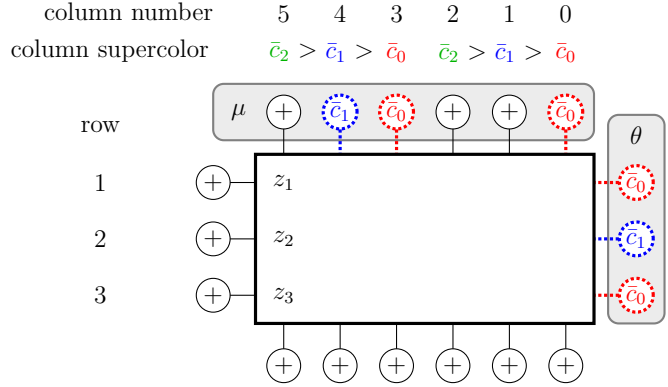

\subsection{Specializations}
\label{sec:specializations}

In this section we will show how the family of lattice models depending on the parameters $\Phi$ and $\alpha_{i,j}$ specialize to two important lattice models: the Iwahori ice model constructed in~\cite{BBBGIwahori} and the Delta metaplectic ice model constructed in~\cite{mice} and further studied in~\cite{BBB}.
Their partition functions compute a basis for Iwahori Whittaker functions for the principal series of $\GL_r(F)$ for the former and for spherical metaplectic Whittaker functions of the metaplectic $n$-cover of $\GL_r(F)$ for the latter where $F$ is a non-archimedean field.

As mentioned in Section~\ref{sec:results} the Iwahori specialization is
\begin{equation}
  \label{eq:Iwahori-data-specialization}
  \Phi=q^{-1} \text{ and } \alpha_{-i,-j} = \left\{\begin{smallmatrix*}[l] 1/q &\ i<j\\ q  &\ i>j\end{smallmatrix*}\right\}
  \text{ for } 1\leqslant i,j\leqslant m,
\end{equation}
and the metaplectic specialization is 
\begin{equation}
  \label{eq:metaplectic-data-specialization}
  \Phi = -q \text{ and } \alpha_{i,j} = -g(i-j)/q.
\end{equation}
We will show that the Boltzmann weights in Table~\ref{tab:boltzmann-weights} reduce to the ones for Iwahori ice in the Iwahori specialization.
For the metaplectic specialization the Boltzmann weights reduce to a model equivalent to metaplectic ice which we will describe below.

For all members of the family we use the same ensemble of admissible states $\mathfrak{S}^m_{\mu,\sigma}$, but with different Boltzmann weights for the partition function.
However, the same states and the boundary data were described differently in the original Iwahori ice and Delta metaplectic ice models of~\cite{BBBGIwahori} and~\cite{BBB} using data that were natural for the respective Whittaker functions.
We will describe how these models were originally constructed in Sections~\ref{sec:Iwahori} and~\ref{sec:Delta} below.

In the mean time, we will describe a dictionary between the different boundary data so that we may state the relationship between the partition functions for the specializations of our family of lattice models and those for the Iwahori and metaplectic ice models in Theorem~\ref{thm:specializations-prime} below.
For this purpose, fix a strict partition $\mu$ and $\sigma \in (\mathbb{Z}/m\mathbb{Z})^r$.

For Iwahori ice we have that the palette size~$m$ equals the rank~$r$, which is also the number of rows in the lattice and the number of components in $\mu$.
We define the composition $\lambda$ by the componentwise integer division
\begin{equation}
  \label{eq:lambda-mu-r}
  \lambda = \left\lfloor \frac{\mu}{r} \right\rfloor - \rho \qquad \text{where } \rho = (r-1, r-2, \ldots, 0).
\end{equation}
The (nonstrict) partition $\lambda+\rho$ will describe the \emph{block} (not column) numbers for where paths start at the top boundary, which is why the integer division by the number of colors $r$ for Iwahori ice appears in~\eqref{eq:lambda-mu-r}.
In other words, $\lambda+\rho$ describe the occupied columns in the \emph{fused} model counted with multiplicity.

The Iwahori ice model contains further boundary data which describe the colors of the paths entering at these fused top boundary edges, as well as the order of these colors exiting at the right boundary.
As described above, in the family of lattice models of this paper this data is described by $\sigma$ (from top to bottom) and the residues of $-\mu \bmod m$ (from left to right at columns $\mu$), where the minus sign is needed because the column colors increase to the right, while the column numbers increase to the left.

For the system to have any admissible states the colors on the top boundary must be a permutation of the colors on the right boundary.
Let $\mathbf{c}$ be the $r$-tuple of these colors (with multiplicities) appearing on the boundary arranged in weakly decreasing order.
For convenience we will also identify $\mathbf{c}$ with its $r$-tuple of color indices in $\{1, \ldots, m\}$.

Let $w_\sigma, w_{-\mu} \in S_r$ be the shortest permutations such that
\begin{equation}
  \label{eq:w-sigma-mu}
  \sigma \equiv w_\sigma \mathbf{c} \mod{m} \qquad \text{and} \qquad {-\mu} \equiv w_{-\mu} \mathbf{c} \mod{m}.
\end{equation}
The uniqueness of these permutations follows from Proposition~2.4.4 of~\cite{BjornerBrenti} and its corollary.
Let $W_P \subseteq S_r$ be the stabilizer of $\mathbf{c}$.
Then $W_P$ determines a parabolic subgroup $P$ of $\GL_r$ as the Weyl group of the corresponding Levi subgroup.
This describes the relations between the boundary data $P, \lambda, w_\sigma$ and $w_{-\mu}$ of Iwahori ice, and the boundary data $\mu$ and $\sigma$ we use in this paper.
We denote the partition function for this Iwahori ice model by $Z^\text{Iwahori ice}_{P,\lambda, w_\sigma, w_{-\mu}}$.

We will also be comparing with the $n$-metaplectic ice models of~\cite{BBB, mice}, and in particular with the Delta ice version, called $\Delta$-ice.
As will be described in more detail in Section~\ref{sec:Delta}, the edges along paths in this model are assigned charges which are elements in $\mathbb{Z}/n\mathbb{Z}$ and vary along each path.
We will later explain how these relate to supercolors in a palette of size~$m=n$.
The paths are still entering at the top boundary and exiting at the right boundary and we still describe the occupied top edges by column numbers in a strict partition~$\mu$.
Note that in contrast to Iwahori ice described above where $\lambda+\rho$ describes block numbers or fused columns, here $\mu$ again describes the unfused column numbers.
The right boundary data consists of charges $\gamma \in (\mathbb{Z}/n\mathbb{Z})^r$ read from top to bottom.
We denote the corresponding partition function for the charged Delta ice $Z^\text{charged $\Delta$}_{\mu,\gamma}$.

Since the partition functions for the Iwahori and metaplectic ice models are related to Whittaker functions we will also get resulting relations to the specializations of our family of partition functions. 
To make a precise statement we will need to introduce the following objects. 
Let $F$ be a non-archimedean field containing the $n$-th roots of unity with uniformizer~$\varpi$ and cardinality of the residue field equal to~$v^{-1}$. 
We will be working with the $p$-adic group $G:=\GL_r(F)$ and its metaplectic $n$-cover $\tilde G$ as defined in~\cite[Section~1]{BBBGMetahori}. 
The Weyl group is $W = S_r$. 

\begin{remark}
\label{rem:v-q}
There is a parameter $q$ that appears in the lattice models. 
The meaning of this parameter is different in the metaplectic and Iwahori interpretations.
In discussing Iwahori models, $q^2$ is the residue cardinality, while in the
metaplectic models $q^{-2}$ is the residue cardinality. In both cases we
will take $v^{-1}$ to be the residue cardinality, so the relationship
between $v$ and $q$ will depend on the context. This discrepancy
is forced on us by the nature of the duality.
\end{remark}

For $\mathbf{z} = (z_1, \ldots, z_r) \in (\mathbb{C}^\times)^r$, $\theta \in (\mathbb{Z}/n\mathbb{Z})^r$ , $w \in W$ and $g \in \tilde G$, let $\phi^{(n)}_{\theta, w}(\z; g)$ be the metaplectic Iwahori Whittaker function defined in~\cite[Section~4.1]{BBBGMetahori} for the (contragredient of the) principal series of $\tilde G$ with Langlands parameters $\mathbf{z}$. 
Here~$w$ enumerates a basis of Iwahori fixed vectors in the principal series representation and~$\theta$ enumerates a basis of Whittaker functionals.
Together we obtain a basis for the metaplectic Iwahori Whittaker functions.

Two special cases will be featuring in this paper.
Firstly, for $n=1$ (i.e.\ $\tilde G = G$ and $\sigma$ is trivial), the above Whittaker functions are identical to the non-metaplectic Iwahori Whittaker functions $\phi_w(\z;g) := \phi^{(1)}_{0, w}(\z; g)$ defined in equation (2) of~\cite{BBBGIwahori}.
More generally, we have the non-metaplectic parahoric Whittaker functions which are defined with respect to a parabolic subgroup $P \subseteq G$.
Let $W_P \subseteq W$ be the Weyl group of the corresponding Levi subgroup.
There is then a basis of parahoric Whittaker functions $\phi^P_{w}(\z; g)$ enumerated by $w \in W/W_P$ and obtained as the sum $\sum_{w' \in W_P} \phi_{ww'}(\z;g)$.

Note that the boundary data $(\mu, \sigma)$ encodes a parabolic subgroup $P$ which allows us to connect our current model to all the parahoric Whittaker functions discussed in~\cite{BBBGIwahori}, including Iwahori Whittaker functions (for $P=B$ the Borel subgroup) and the non-metaplectic spherical one ($P=G$).

Secondly, we shall also work with the metaplectic spherical function
\begin{equation}
  \label{eq:spherical-Whittaker}
  \tilde\phi^\circ_\theta(\z; g) := \sum_{w\in W} \phi^{(n)}_{\theta, w}(\z; g).
\end{equation}
This function has been studied extensively in many papers such as~
\cite{KazhdanPatterson, wmd5book, McNamaraCS, ChintaGunnells, BBB} although in some cases with slightly different conventions.

\begin{customtheorem}{\ref*{thm:specializations}$'$}[Refinement of Theorem~\ref{thm:specializations}]
  \label{thm:specializations-prime}
  The lattice model with color palette $\mathcal{P}_m$ defined by the weights in Table~\ref{tab:boltzmann-weights} specializes to the Iwahori ice model defined in~\cite{BBBGIwahori}, and to the Delta metaplectic ice model defined in~\cite{mice, BBB}.
  Specifically, let $\mu \in \mathbb{Z}^r_{\geqslant0}$ be a strict partition, $\sigma \in (\mathbb{Z}/m\mathbb{Z})^r$ and $\lambda = \left\lfloor \frac{\mu}{r} \right\rfloor - \rho$.
  Let $w_\sigma$ and $w_{-\mu}$ be the Weyl elements defined in \eqref{eq:w-sigma-mu} and $P$ the parabolic subgroup defined by $\sigma$ as described above.
  Then for $m=r$
  \begin{align}
    \label{eq:Iwahori-specialization}
    Z(\mathfrak{S}^{m=r}_{\mu,\sigma})(\z)|_\mathrm{Iwahori} &= Z^\text{\rm Iwahori ice}_{P, \lambda, w_\sigma, w_{-\mu}}(\z) = \mathbf{z}^\rho \phi^P_{w_\sigma}(\z; \varpi^{-\lambda}w_{-\mu}) 
  \intertext{and for $m=n$}
    \label{eq:metaplectic-specialization}
    \mathbf{z}^{\sigma} Z(\mathfrak{S}^{m=n}_{\mu,-w_0\sigma})(w_0\mathbf{z}^n)|_\mathrm{metaplectic} &=  Z^\text{\rm charged $\Delta$}_{\mu, \gamma}(w_0\mathbf{z}) = \mathbf{z}^{\rho} \tilde\phi^\circ_{\sigma}(\mathbf{z}; \varpi^{\rho - \mu}),
  \end{align}
  where $w_0 \in S_r$ is the long Weyl element, $\gamma = w_0\sigma+1$ and $\mathbf{z}^n = (z_1^n, \ldots, z_r^n)$.
  Furthermore, in each case there is a bijection of (admissible) boundary data.
\end{customtheorem}

\begin{remark}
  \label{rem:color-relabel}
  In fact, \eqref{eq:Iwahori-specialization} holds for any $m$ if we let $\lambda = \floor{\frac{\mu}{m}} - \rho$, but it is only for $m=r$ that we get a bijection for the boundary data.
  Indeed, if we consider a system $\mathfrak{S}^m_{\mu, \sigma}$ for any palette size $m$ it still contains $r$ rows and thus there are at most $r$ distinct colors on the right boundary determined by $\sigma$.
  The colors appearing in $\mathfrak{S}^m_{\mu, \sigma}$ can therefore be described by an $r$-tuple $\mathbf{c}$ of colors in a palette $\mathcal{P}_m$ of size $m$ with the same definitions for $w_\sigma, w_{-\mu}$ and $P$ as above.

  Fix an $r$-tuple of colors $\mathbf{c}$ and change the size $m$ of the palette by introducing or removing columns in the unfused system corresponding to colors not in $\mathbf{c}$ and which therefore do not appear in $\mathfrak{S}^m_{\mu, \sigma}$.
  This may require some corresponding shifts in $\mu$ which describes the unfused columns numbers, but $\lambda$ remains unchanged with the above redefinition and describes the color block numbers, or the columns of the fused model.
  Adding or removing unfused columns introduces extra vertex configurations of type $\texttt{a}_1$ and of type $\texttt{b}_2$ with $i \neq j$ in Table~\ref{tab:boltzmann-weights} but their weights are trivial.
  Furthermore, in the Iwahori specialization the color dependence for the $\texttt{a}_2$ weight, (which is governed by $\alpha_{i,j}$) is given only by the internal order of $i$ and $j$ and is thus unchanged under the addition or removal of other colors.
  In summary, the partition function remains unchanged.
  This fact will later be used when proving Theorem~\ref{thm:DW}.
\end{remark}

Since the proof of Theorem~\ref{thm:specializations-prime} is technical but otherwise straightforward we postpone the proof to Section~\ref{sec:specializations-proof} which can be skipped at a first reading.
Along the way the proof relates many different variants of lattice models that have appeared in the literature and forms a web of different dualities.
We summarize these relations in Figure~\ref{fig:web} together with forward references to the specific statements.
The figure also provides an overview of the proof of Theorem~\ref{thm:specializations-prime}.

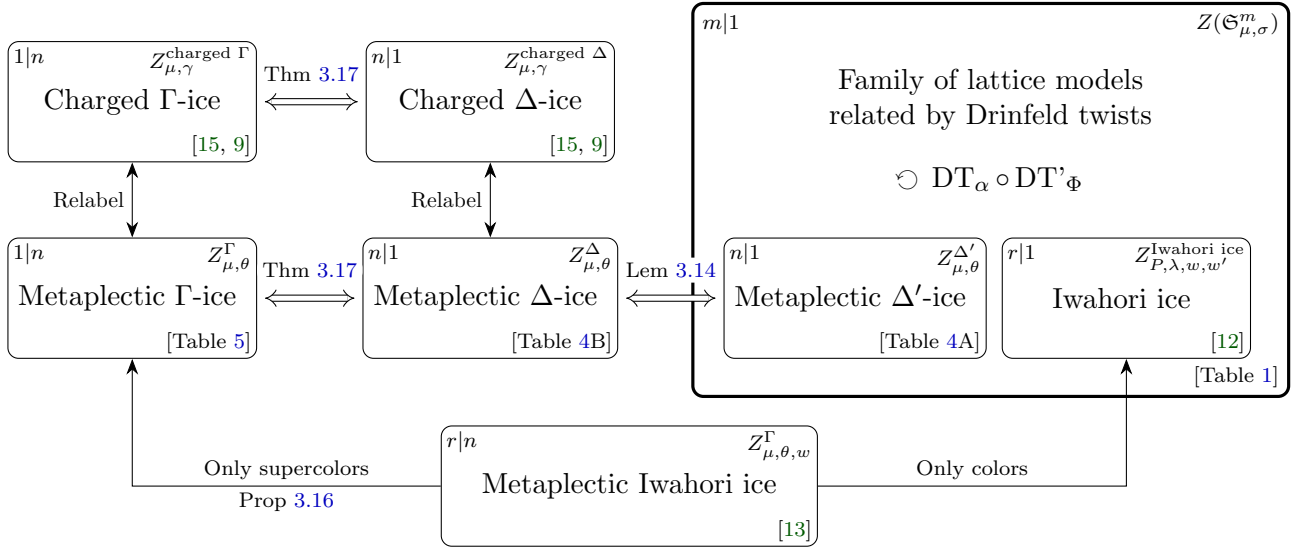
\begin{figure}[htpb]
  \centering
  \usetikzlibrary{fit, arrows}
  \begin{tikzpicture}[font=\small, 
    sqr/.style={rectangle, rounded corners,align=center,draw,minimum height=16mm}, 
    implies/.style={implies-implies, shorten <=3pt, shorten >=3pt, double equal sign distance}, 
    cornertext/.style={inner sep=2pt, font=\tiny}]
    \matrix (m) [matrix of nodes, row sep = 10mm, nodes in empty cells, every node/.style={minimum width=33mm, anchor=center},
    column 1/.style={column sep = 14mm},
    column 2/.style={column sep = 14mm},
    column 3/.style={column sep = 2mm}]
    {
      \node[sqr, align=center] (chargedgamma) {Charged $\Gamma$-ice}; & |[sqr](chargeddelta)| Charged $\Delta$-ice& |[sqr,opacity=0]| & \\
        |[sqr](gamma)| Metaplectic $\Gamma$-ice & |[sqr](delta)| Metaplectic $\Delta$-ice & |[sqr](deltap)| Metaplectic $\Delta'$-ice & |[sqr](iwahori)| Iwahori ice \\
       & & & \\
    };
    \node[very thick, fit=(m-1-3)(iwahori), sqr, inner sep=5mm] (family) {}; 
    \node[below = 7.5mm of family.north, align=center] {Family of lattice models \\ related by Drinfeld twists \\[1em] $\rotatebox[origin=c]{90}{$\circlearrowleft$}\ \operatorname{DT}_{\alpha} \circ \operatorname{DT'}_\Phi$};
    \node[sqr, minimum width=50mm] at ($(gamma)!0.5!(iwahori)+(0,-2.5)$) (metahori) {Metaplectic Iwahori ice};

    \draw[-{Stealth[scale=1.25]}] (metahori) -| (gamma) node[pos=0.25, label={[label distance=-2mm, font=\tiny]above:Only supercolors}, label={[label distance=-2mm, font=\tiny]below:Prop~\ref{prop:metahori-is-parahoric}}] {};
    \draw[-{Stealth[scale=1.25]}] (metahori) -| (iwahori) node[pos=0.25, label={[label distance=-2mm, font=\tiny]above:Only colors}] {};
    
    \node[above left, cornertext] at (delta.south east) {[Table~\ref{tab:metaplectic-weights}B]};
    \node[below left, cornertext] at (delta.north east) {$Z^\Delta_{\mu,\theta}$};
    \node[below right, cornertext] at (delta.north west) {$n|1$};

    \node[above left, cornertext] at (deltap.south east) {[Table~\ref{tab:metaplectic-weights}A]};
    \node[below left, cornertext] at (deltap.north east) {$Z^{\Delta'}_{\mu,\theta}$};
    \node[below right, cornertext] at (deltap.north west) {$n|1$};

    \node[above left, cornertext] at (iwahori.south east) {\cite{BBBGIwahori}};
    \node[below left, cornertext] at (iwahori.north east) {$Z^\text{Iwahori ice}_{P, \lambda, w, w'}$};
    \node[below right, cornertext] at (iwahori.north west) {$r|1$};

    \node[above left, cornertext] at (gamma.south east) {[Table~\ref{tab:gamma-boltzmann-weights}]};
    \node[below left, cornertext] at (gamma.north east) {$Z^\Gamma_{\mu,\theta}$};
    \node[below right, cornertext] at (gamma.north west) {$1|n$};

    \node[above left, cornertext] at (chargedgamma.south east) {\cite{mice, BBB}};
    \node[below left, cornertext] at (chargedgamma.north east) {$Z^\text{charged $\Gamma$}_{\mu, \gamma}$};
    \node[below right, cornertext] at (chargedgamma.north west) {$1|n$};

    \node[above left, cornertext] at (chargeddelta.south east) {\cite{mice, BBB}};
    \node[below left, cornertext] at (chargeddelta.north east) {$Z^\text{charged $\Delta$}_{\mu, \gamma}$};
    \node[below right, cornertext] at (chargeddelta.north west) {$n|1$};

    \node[above left, cornertext] at (metahori.south east) {\cite{BBBGMetahori}};
    \node[below left, cornertext] at (metahori.north east) {$Z^\Gamma_{\mu,\theta,w}$};
    \node[below right, cornertext] at (metahori.north west) {$r|n$};
    
    \node[below left, cornertext, inner sep = 4pt] at (family.north east) {$Z(\mathfrak{S}^m_{\mu,\sigma})$};
    \node[above left, cornertext, inner sep = 4pt] at (family.south east) {[Table~\ref{tab:boltzmann-weights}]};
    \node[below right, cornertext, inner sep = 4pt] at (family.north west) {$m|1$};
    
    \draw[{Stealth[scale=1.25]}-{Stealth[scale=1.25]}] (gamma) -- (chargedgamma) node[midway, left, scale=0.75] {Relabel};
    \draw[implies] (gamma) -- (delta) node[midway, above = 2mm, font=\tiny, fill=white, inner sep=2pt, rounded corners] {Thm~\ref{thm:StatementB-refinement}};
    \draw[implies] (chargedgamma) -- (chargeddelta) node[midway, above = 2mm, font=\tiny, fill=white, inner sep=2pt, rounded corners] {Thm~\ref{thm:StatementB-refinement}};
    \draw[{Stealth[scale=1.25]}-{Stealth[scale=1.25]}] (delta) -- (chargeddelta) node[midway, left, scale=0.75] {Relabel};
    \draw[implies] (delta) -- (deltap) node[midway, above = 2mm, font=\tiny, fill=white, inner sep=2pt, rounded corners] {Lem~\ref{lem:Delta-Deltaprime}};
  \end{tikzpicture}
  \caption{Web of dualities and relations between different lattice models as proved in the above sections.
  Each box corresponds to a solvable lattice model, or, in the case of the boldly outlined box, a family of lattice models.
  In each box we have codified the following information for each lattice model: In the lower right corner we give references for the definition of the lattice model.
  In the upper left corner we write the quantum group associated to the R-matrix in the Yang-Baxter equation where $m|n$ denotes $U_q(\widehat{\mathfrak{gl}}(m|n))$.
  In the upper right corner we write the notation we use in this paper for the partition function of the lattice model.}
  \label{fig:web}
\end{figure}

In the figure, the boldly outlined box depicts the family of lattice models in this paper.
In Section~\ref{sec:YBE} we will show that the members of this family are related by Drinfeld twists $\operatorname{DT}_\alpha \circ \operatorname{DT'}_\Phi$ of $U_q(\widehat{\mathfrak{gl}}(m|1))$ evaluation modules. For particular choices of $m$ and particular Drinfeld twists, two of those family members recover Iwahori ice \cite{BBBGIwahori} and a variant of metaplectic ice here labeled $\Delta'$, and this forms the foundation of the Iwahori-metaplectic duality. The equivalence of $\Delta'$-ice and other flavors of metaplectic ice in \cite{mice, BBB} are depicted in the figure, with double arrows denoting results requiring proof (e.g., using solvability of the model) and single arrows depicting simpler relabelings (reflecting notational conventions from earlier works). First we show the wanted equivalence to the standard $\Delta$ metaplectic ice.
Then, using a refinement of a $\Delta$-$\Gamma$ equivalence originally from~\cite{BBB} we arrive at metaplectic $\Gamma$-ice at one end and Iwahori ice at the other, these are both special cases of the metaplectic Iwahori ice in~\cite{BBBGMetahori}.
The arrows from metaplectic Iwahori ice arrows travel only in one direction since they are restrictions rather than equivalences as in other arrows. 
This allows us to recover the different kinds of Whittaker functions we defined in Section~\ref{sec:specializations} in terms of analogous restrictions of the metaplectic Iwahori Whittaker functions of~\cite{BBBGMetahori}.

\begin{remark}
  The metaplectic $\Delta$-ice models are associated to the quantum group $U_q(\widehat{\mathfrak{gl}}(n|1))$ while the $\Gamma$-ice models are associated to $U_q(\widehat{\mathfrak{gl}}(1|n))$ and the highly non-trivial $\Delta$-$\Gamma$ equivalence shows that their partition functions are equal.
  Note that the superalgebras $\mathfrak{gl}(1|n)$ and $\mathfrak{gl}(n|1)$ are related by Remark~1.6 in~\cite{ChengWangBook}, and this relationship persists for the corresponding affine quantum supergroups. 
\end{remark}

From \eqref{eq:metaplectic-specialization} in Theorem~\ref{thm:specializations-prime} we immediately get the following result.

\begin{corollary}
  \label{cor:n-th-power}
  The spherical metaplectic Whittaker function $\z^{\rho-\theta}\tilde{\phi}^\circ_\sigma(\z; \varpi^{\rho-\mu})$ is a polynomial in $\z^n = (z_1^n, \ldots, z_r^n)$.
\end{corollary}

\subsection{Yang-Baxter solvability}
\label{sec:YBE}

In this section we will prove that the lattice models of this paper are solvable.
Recall that we introduced two types of vertices in Section~\ref{sec:results}: T-vertices and R-vertices.
We say that a lattice model is \emph{Yang-Baxter solvable} if it satisfies both so called RTT and RRR Yang-Baxter equations. The RTT Yang-Baxter equations can be expressed as the equality of the partition functions for the following systems
\begin{subequations}
  \label{eq:YBE}
\begin{equation}
  \label{eq:RTT2}
    \begin{tikzpicture}[baseline, scale=0.8, round/.style={draw, circle, fill=white, inner sep=0pt, minimum size=14pt}, dot/.style={circle, fill, inner sep=0pt, minimum size=3pt}]
      \draw (-1,-1) node [round, label={[label distance=5mm]left:$z_1$}] {$a$} -- (0,0) node[dot, label={[label distance=2mm]above:$R_{12}$}] {} -- (1,1) node [round] {$i$} -- (2,1) node[dot, label=above right:${T_1}$] {} -- (3,1) node[round, label={[label distance=5mm]right:$z_1$}] {$d$};
    \draw (-1,1) node [round, label={[label distance=5mm]left:$z_2$}] {$b$} -- (1,-1) node [round] {$j$} -- (2,-1) node[circle, fill, inner sep=0pt, minimum size=3pt, label=above right:${T_2}$] {} -- (3,-1) node[round, label={[label distance=5mm]right:$z_2$}] {$e$};
    \draw (2,-2) node[round] {$f$} -- (2,0) node [round] {$k$} -- (2,2) node[round] {$c$};
  \end{tikzpicture}
  \qquad = \qquad
  \begin{tikzpicture}[xscale=-1, baseline, scale=0.8, round/.style={draw, circle, fill=white, inner sep=0pt, minimum size=14pt}, dot/.style={circle, fill, inner sep=0pt, minimum size=3pt}]
    \draw (-1,-1) node [round, label={[label distance=5mm]right:$z_2$}] {$e$} -- (0,0) node[dot, label={[label distance=2mm]above:$R_{12}$}] {} -- (1,1) node [round] {$l$} -- (2,1) node[dot, label=above right:${T_2}$] {} -- (3,1) node[round, label={[label distance=5mm]left:$z_2$}] {$b$};
    \draw (-1,1) node [round, label={[label distance=5mm]right:$z_1$}] {$d$} -- (1,-1) node [round] {$m$} -- (2,-1) node[circle, fill, inner sep=0pt, minimum size=3pt, label=above right:${T_1}$] {} -- (3,-1) node[round, label={[label distance=5mm]left:$z_1$}] {$a$};
    \draw (2,-2) node[round] {$f$} -- (2,0) node [round] {$n$} -- (2,2) node[round] {$c$};
  \end{tikzpicture}
\end{equation}
for fixed boundary edges $a,b,c,d,e$ and $f$ while summing over interior edges $i,j,k$ on the left-hand side and $l,m,n$ on the right-hand side.
The indices for the R and T-vertices denote data such as the row parameters $z_1$ and $z_2$ that the weights of the vertices depend on.
Note that the T-vertices have swapped rows in the process.
Although the equation can be applied to any pair of consecutive row parameters we have chosen $z_1$ and $z_2$ for simplicity.

One may consider an RTT equation as in~\eqref{eq:RTT2} for both a fused and an unfused model.
For an unfused model the T-vertices depend on a column or vertex color $c_\ell$ which is the same for the left- and right-hand sides of~\eqref{eq:RTT2}.
The R-matrix is then also dependent on a vertex color which is $c_\ell$ for the left-hand side and $c_{\ell'}$ for the right-hand side where $\ell' \equiv \ell + 1 \mod{m}$ is the corresponding color for the next column.
We showed in~\cite[Section~5]{BBBGIwahori} that the Yang-Baxter equation for the unfused model implies a Yang-Baxter equation for the fused model by repeatedly applying~\eqref{eq:RTT} for a whole block of T-vertices.
In fact, the R-matrix for the fused model is then the R-matrix for the unfused model specialized to vertex color $c_1$ which attaches to the left of a block of T-vertices.

The RRR Yang-Baxter equation can be expressed in the same way by a deformation of the above lines as follows
\begin{equation}
  \label{eq:RRR}
  \begin{tikzpicture}[baseline={(0,-0.8)}, scale=0.8, round/.style={draw, circle, fill=white, inner sep=0pt, minimum size=14pt}, dot/.style={circle, fill, inner sep=0pt, minimum size=3pt}]
    \draw (-1,1) node [round] {$c$} -- (0,0) node[dot, label={[label distance=2mm]above:$R_{13}$}] {} -- (1,-1) node[round] {$k$} -- (2,-2) node[dot, label={[label distance=2mm]above:$R_{23}$}] {} -- (3,-3) node[round] {$f$};
    \draw (1,1) node [round] {$d$} -- (-1,-1) node[round]{$i$} -- (-2,-2) node[dot, label={[label distance=2mm]above:$R_{12}$}] {} -- (-3,-3) node[round] {$a$};
    \draw (-3,-1) node [round] {$b$} -- (-1,-3) -- (1,-3) node[midway, round] {$j$} -- (3,-1) node[round] {$e$};
  \end{tikzpicture}
  \qquad = \qquad
  \begin{tikzpicture}[scale=0.8, yscale=-1, baseline={(0,.8)}, round/.style={draw, circle, fill=white, inner sep=0pt, minimum size=14pt}, dot/.style={circle, fill, inner sep=0pt, minimum size=3pt}]
    \draw (-1,1) node [round] {$a$} -- (0,0) node[dot, label={[label distance=2mm]above:$R_{13}$}] {} -- (1,-1) node[round] {$m$} -- (2,-2) node[dot, label={[label distance=2mm]above:$R_{12}$}] {} -- (3,-3) node[round] {$d$};
    \draw (1,1) node [round] {$f$} -- (-1,-1) node[round]{$n$} -- (-2,-2) node[dot, label={[label distance=2mm]above:$R_{23}$}] {} -- (-3,-3) node[round] {$c$};
    \draw (-3,-1) node [round] {$b$} -- (-1,-3) -- (1,-3) node[midway, round] {$l$} -- (3,-1) node[round] {$e$};
  \end{tikzpicture}
\end{equation}
\end{subequations}

Our first main result shows that all members of the family of lattice models defined above in the beginning of Section~\ref{sec:family} satisfy Yang-Baxter equations both for the fused and unfused descriptions.
Furthermore, we make the observation that the R-matrix for the unfused model can be derived from that of the fused model via a color palette shift (which is a statement in the opposite direction compared to the statement above from~\cite{BBBGIwahori}).

\begin{customtheorem}{\ref*{thm:YBE}$'$}[Refinement of Theorem~\ref{thm:YBE}]
  \label{thm:YBE-prime}
  Both the unfused and the resulting fused lattice models defined by the weights in Table~\ref{tab:boltzmann-weights} with a palette of $m$ colors are Yang-Baxter solvable, the latter model with an R-matrix for the $U_q(\widehat{\mathfrak{gl}}(m|1))$ evaluation module under a Drinfeld twist with parameters $\alpha_{i,j}$ and $\Phi$. Furthermore, the R-matrix for the unfused model at vertex color $c_k$ is a color palette shift (with wrapping) of the R-matrix for the fused model such that $c_k$ becomes the first and smallest color of the palette. 
\end{customtheorem}

The proof is at the end of this subsection.
In Table~\ref{tab:R-matrix} we show the weights and configurations for the R-vertices used in the Yang-Baxter equations for the fused model which are obtained by the fusion process described in Section~\ref{sec:fusion} from the unfused weights in Table~\ref{tab:boltzmann-weights}.
Recall that these weights do not depend on a column color.

However, an R-vertex for the unfused model depends on the column or vertex color $c_k$ for the T-vertices attached to the right of the R-vertex.
The configurations and weights are the same as those in Table~\ref{tab:R-matrix} except with the addition of the attribute $c_k$ for the column or vertex color and the weight for the third configuration on the first row is changed to
\begin{equation}
  \label{eq:R-matrix-unfused-exceptional-weight}
  (1 - q^2) \times \begin{cases}
  z_1 & i < j < k, \; k \leqslant i < j, \; j < k \leqslant i\\
  z_2 & j < i < k, \; k \leqslant j < i, \; i < k \leqslant j.
\end{cases}
\end{equation}

\begin{table}[htpb]
  \centering
  \caption{R-vertex configurations and weights for the fused model obtained by the fusion process described in Section~\ref{sec:fusion} from the unfused Boltzmann weights in Table~\ref{tab:boltzmann-weights}.}
  \label{tab:R-matrix}
  {\tabulinesep=1.2mm
\begin{tabu}{|c|c|c|c|}
    \hline
    
    \begin{tikzpicture}
      \draw (-1,-1) -- (1,1);
      \draw (-1,1) -- (1,-1);
    \end{tikzpicture}
    
    & 
    
    \begin{tikzpicture}
      \draw[ultra thick, red] (-1,-1) -- (1,1);
      \draw[ultra thick, red] (-1,1) -- (1,-1);
    \end{tikzpicture}
    
    & 
    
    \begin{tikzpicture}
    \draw[ultra thick, red] (-1,-1) node[label=above:$c_i$] {}
    -- (0,0)-- (1,-1);
    \draw[ultra thick, blue] (-1,1) node[label=below:$c_j$] {}
    -- (0,0) -- (1,1);
    \end{tikzpicture}
    
    & 
    
    \begin{tikzpicture}
    \draw[ultra thick, red] (-1,-1) node[label=above:$c_i$] {}
    -- (1,1);
    \draw[ultra thick, blue] (-1,1) node[label=below:$c_j$] {}
    -- (1,-1);
    \end{tikzpicture}

    \\ \hline

    $z_1 - q^2 z_2$ & $z_2 - q^2 z_1$ & \scriptsize $(1 - q^2) \times
    \begin{cases}
      z_1 & i < j\\
      z_2 & i > j
    \end{cases}$ & $-q (z_1 - z_2) \alpha_{-i, -j}$\\
    \hline\hline
    
    \begin{tikzpicture}
      \draw[ultra thick, red] (-1,-1) -- (0,0)-- (1,-1);
      \draw (-1,1) -- (0,0) -- (1,1);
    \end{tikzpicture}
    
    & 
    
    \begin{tikzpicture}
      \draw (-1,-1) -- (0,0)-- (1,-1);
      \draw[ultra thick, red] (-1,1) -- (0,0) -- (1,1);
    \end{tikzpicture}

    & 
    
    \begin{tikzpicture}
      \draw (-1,-1) -- (1,1);
      \draw[ultra thick, red] (-1,1) -- (1,-1);
    \end{tikzpicture}
    
    & 
    
    \begin{tikzpicture}
      \draw[ultra thick, red] (-1,-1) -- (1,1);
      \draw (-1,1) -- (1,-1);
    \end{tikzpicture}
    
    \\ \hline
    
    $(1 - q^2) z_1$ & $(1 - q^2) z_2$ & $-q (z_1 - z_2) \Phi$ & $-q (z_1 - z_2) / \Phi$\\
    \hline
\end{tabu}}
\end{table}

As mentioned above, the Yang-Baxter equation for the fused vertices follows from repeated use of the Yang-Baxter equation for the unfused vertices after passing through a full block of colors. (See also \cite[Lemma~5.4]{BBBGIwahori}.) The Yang-Baxter equations for the unfused system in the Iwahori case $\Phi = 1/q$ and $\alpha_{-i,-j} = \left\{\begin{smallmatrix*}[l] 1/q &\ i<j\\ q  &\ i>j\end{smallmatrix*}\right\}$ is proved in~\cite[Proposition~6.4]{BBBGIwahori}.

We will now show that the Yang-Baxter equation remains true under the change to different $\Phi$ and $\alpha_{i,j}$.
However, we will show this for a greater generality than what is needed for this paper by considering any lattice model $\mathbb{L}$ for which the spinset of the horizontal edges is described by a finite set $\mathcal{E}$ and the vertical edges are described by multisets which we consider as maps $\mathcal{E} \to \mathbb{Z}_{\geqslant 0}$.
This will also allow us to give and prove parallel statements for both the fused and unfused models at the same time and treat both the RRR and RTT Yang-Baxter equations simultaneously.

Consider the ring $\mathcal{E}^*$ of maps $\mathcal{E} \to \mathbb{Z}$ and let $\mathcal{E}^*_\mathrm{ms}$ be the subset of multisets $\mathcal{E} \to \mathbb{Z}_{\geqslant 0}$.
It will be convenient to describe both horizontal and vertical edges by elements in $\mathcal{E}^*$. 

Recall that for the family of unfused lattice models of this paper edges are assigned either a color in the palette $\mathcal{P}_m$ or no color (denoted by a plus sign) with the color of a vertical edge being constrained by the column number.
Thus, $\mathcal{E}$ is here $\{\oplus\} \cup \mathcal{P}_m$ and all edges in the unfused model may in fact be described as elements of this set.
In the fused model we may use the same underlying set $\mathcal{E}$ but here we describe the vertical edges by subsets of $\mathcal{P}_m$.

As another example, consider the lattice model described in~\cite{BBBGMetahori}.
Here unfused vertical edges are assigned both a color and a supercolor, and horizontal edges are assigned either a color or a supercolor. In this case $\mathcal{E}$ would here be the union of all colors and supercolors, and all fused or unfused edges could be described by multisets $\mathcal{E}^*_\mathrm{ms}$.

We will assume the following canonical conservation law (which holds for all the examples mentioned above).
\begin{assumption}
  \label{assumption:conservation}
  Consider a lattice model whose edges can be described by the multisets $\mathcal{E}^*_\mathrm{ms} := \{\mathcal{E} \to \mathbb{Z}_{\geqslant 0}\}$.
  Let $a,b,c,d \in \mathcal{E}^*_\mathrm{ms}$ describe the edges of a T- or R-vertex using the following naming convention
  \begin{equation}
    \label{eq:edge-labels}
    \begin{tikzpicture}[scale=0.8, baseline, round/.style={draw, circle, fill=white, inner sep=0pt, minimum size=14pt}, dot/.style={circle, fill, inner sep=0pt, minimum size=3pt}]
      \draw (-1,0) node[round] {$a$} -- (0,0) node[dot, label={above right:$T$}] {} -- (1,0) node[round] {$c$};
      \draw (0,1) node[round] {$b$} -- (0,-1) node[round] {$d$};
    \end{tikzpicture}
    \qquad\qquad
    \begin{tikzpicture}[scale=0.8, baseline, round/.style={draw, circle, fill=white, inner sep=0pt, minimum size=14pt}, dot/.style={circle, fill, inner sep=0pt, minimum size=3pt}]
      \draw (-1,-1) node[round] {$a$} -- (0,0) node[dot, label={[label distance=2mm]above:$R$}] {} -- (1,1) node[round] {$c$};  
      \draw (-1,1) node[round] {$b$} -- (1,-1) node[round] {$d$};
    \end{tikzpicture}
  \end{equation}
  Then we will assume that all admissible vertex configurations for the lattice model satisfies the conservation law
  \begin{equation}
    \label{eq:conservation}
    a + b = c + d.
  \end{equation}
\end{assumption}

The labels in~\eqref{eq:RTT2} and~\eqref{eq:RRR} are such that they satisfy the same conservation laws.

\begin{remark}
  \label{rem:sub-paths}
  Because of this conservation law and the fact that $a,b,c$ and $d$ are all integer valued we may represent each admissible vertex configuration as a collection of sub-paths from $a$ or $b$ to $c$ or $d$ each carrying an element of $\mathcal{E}$. 
\end{remark}

Let us introduce some notation that will be used to describe the different Drinfeld twists.
Fix a subset $\mathcal{P}$ of $\mathcal{E}$, which for the lattice models of this paper will be the palette of colors $\mathcal{P}_m$.
For $f \in \mathcal{E}^*$ define the $\mathcal{P}$-average 
\begin{equation}
  \label{eq:P-average}
  \langle f \rangle = \langle f \rangle_\mathcal{P} := \sum_{x \in \mathcal{P}} f(x). 
\end{equation}

Let $\phi : \mathcal{E} \times \mathcal{E} \to \mathbb{C}^\times$ such that $\phi(x,y) \phi(y,x) = 1$ and $\phi(x,x) = 1$ for $x,y \in \mathcal{E}$ and define the antisymmetric bilinear form $\langle \; , \rangle_\phi : \mathcal{E}^* \times \mathcal{E}^* \to \mathbb{C}$ by
\begin{equation}
  \label{eq:antisymmetric-form}
  \langle f, g \rangle = \langle f, g \rangle_\phi := \sum_{x,y \in \mathcal{E}} \log \phi(x,y) f(x) g(y).
\end{equation}
Since $\phi$ takes a finite number of values we may choose a branch cut that makes the above logarithm well-defined.
The Yang-Baxter equation will contain factors on the form $\exp(\langle f, g \rangle)$ which is independent of the choice of branch.
Note that if $f$ and $g$ are multisets describing single elements $a$ and $b$ in $\mathcal{E}$ respectively, then $\exp(\langle f, g \rangle) = \phi(a,b)$.

Let $\mathbb{L}$ be a lattice model satisfying Assumption~\ref{assumption:conservation} with edges described by maps $\mathcal{E} \to \mathbb{Z}_{\geq 0}$, let $\phi : \mathcal{E} \times \mathcal{E} \to \mathbb{C}^\times$ as above and $\Phi \in \mathbb{C}^\times$.
We will show that Yang-Baxter solvability is an invariant property under the following operations of the Boltzmann weights which are generalizations of Drinfeld twists.

Let $\DT_\phi(\mathbb{L})$ be the lattice model obtained from $\mathbb{L}$ by multiplying the Boltzmann weights for both T- and R-vertices by $\exp(\frac{1}{4}\langle a+c, b+d\rangle_\phi)$ with edges labeled as in~\eqref{eq:edge-labels}.

\begin{remark}
  \label{rem:standard-DT}
  The operation $\DT_\phi$ amounts to multiplying a weight by $\phi(x,y)$ whenever we have a crossing of two sub-paths carrying $x$ and $y$ in $\mathcal{E}$.
  Indeed if $x$ is carried by both $a+c$ and $b+d$ it will be cancelled out since $\langle \, , \, \rangle$ is antisymmetric, and if a sub-path carrying $x$ goes from $a$ to $c$ crosses a sub-path carrying $y$ from $b + d$ we get a contribution $\exp\bigl(\frac{1}{4} \log\phi(x,y) \bigl(a(x) + c(x) \bigr)\bigl(b(y) + d(y)\bigr) \bigr) = \phi(x,y)$.
  This is the standard Drinfeld twist of evaluation modules that also appears in~\cite[Proposition~4.2]{BBB} which we prove here in a more general setting.
\end{remark}

\begin{proposition}
  \label{prop:standard-DT}
  Yang-Baxter solvability is an invariant under the standard Drinfeld twist~$\DT_\phi$.
\end{proposition}

We postpone the proof to the end of this section and introduce another transformation of the weights using a parameter $\Phi \in \mathbb{C}^\times$ and the $\mathcal{P}$-average defined in \eqref{eq:P-average}.
Here the R-weights transform as a special case of the standard Drinfeld twist above, but the T-weights transform in a non-standard way.

Specifically, let $\DT'_\Phi(\mathbb{L})$ be the lattice model obtained from $\mathbb{L}$ by multiplying the R-weights by $\Phi^{\langle b - c\rangle_\mathcal{P}}$ and the T-weight is multiplied by $\Phi^{\langle d \rangle_\mathcal{P}}$ with edge labels as in~\eqref{eq:edge-labels}.

\begin{remark}
If one only considers the R-weight transformation, and not the T-weight transformation, the non-standard Drinfeld twist $\DT'_\Phi$ is a special case of a standard Drinfeld twist $\DT_\phi$ with $\phi(x,y) = \Phi = \phi(y,x)^{-1}$ for $x\notin\mathcal{P}$ and $y \in \mathcal{P}$, and $1$ otherwise.
Indeed, since the R-vertex only has horizontal edges attached to it which are assigned elements of $\mathcal{E}$ with multiplicity one, we obtain two sub-paths of Remark~\ref{rem:sub-paths} carrying $x\in\mathcal{E}$ from $a$ to one of the outputs $c$ or $d$ and $y \in \mathcal{E}$ from $b$ to the other output.
By Remark~\ref{rem:standard-DT}, the standard Drinfeld twist only gives a contribution if $x$ and $y$ cross, which is equivalent to $b \neq c$, and the same is true for $\Phi^{\langle b-c \rangle_\mathcal{P}}$.
Assuming $x$ and $y$ cross, both the standard and non-standard Drinfeld twist $\DT_\phi$ and $\DT'_\Phi$ then transform the R-weight by a factor of 
\begin{equation}
  \phi(x,y) =
  \begin{cases*}
    \Phi & if $x \notin \mathcal{P}, y \in \mathcal{P}$, \\
    \Phi^{-1} & if $x \in \mathcal{P}, y \notin \mathcal{P}$, \\
    1 & otherwise.
  \end{cases*}
\end{equation}
\end{remark}

\begin{proposition}
  \label{prop:non-standard-DT}
  Yang-Baxter solvability is an invariant under the non-standard Drinfeld twist $\DT'_\Phi$.
\end{proposition}

The proofs are the same for both the RTT equation~\eqref{eq:RTT2} and the RRR equation~\eqref{eq:RRR} and we refer to a fixed choice of either one as the Yang-Baxter equation~\eqref{eq:YBE}.

\begin{proof}[Proof of Proposition~\ref{prop:standard-DT}] 
  We will show that all terms in the Yang-Baxter equation~\eqref{eq:YBE} for given boundary edges $a,b,c,d,e$ and $f$ are multiplied with the same factors of $\phi(x,y)$.
  
  From the conservation law~\eqref{eq:conservation} and the antisymmetry of the bilinear form one can show that
  \begin{equation}
    \frac{1}{4} \langle a+c, b+d \rangle = \frac{1}{2} \langle a,b \rangle + \langle c, d \rangle .
  \end{equation}
  
  Then, the $\phi$-factor for the left-hand side in~\eqref{eq:YBE} can be expressed as $\exp(\frac{1}{2}\operatorname{LHS})$ where
  \begin{equation}
    \operatorname{LHS} = \langle a, b \rangle + \langle i, j \rangle + \langle i, c \rangle + \langle
d, k \rangle + \langle j, k \rangle + \langle e, f \rangle.
  \end{equation}
  Using the conservation laws $a + b = i + j$, $j + k = e + f$, and $a + b + c = d + e + f$ for each vertex we obtain
  \begin{equation}
    \begin{split}
      \operatorname{LHS} 
      &= \langle a, b \rangle + \langle a + b, j \rangle + \langle a + b - j, c \rangle + \langle d, e + f - j \rangle + \langle j, e + f \rangle + \langle e,f \rangle \\
      &= \langle a, b \rangle + \langle a + b, c \rangle + \langle d, e + f\rangle + \langle e, f \rangle +\langle a + b + c - d - e - f, j \rangle 
    \end{split}
  \end{equation}
  where the last term is zero because of the total conservation of color for the whole left-hand side configuration.
  
  Similarly, the $\phi$-factor for the right-hand side is $\exp(\frac{1}{2}\operatorname{RHS})$ where
  \begin{equation}
    \operatorname{RHS} = \langle a, n \rangle + \langle m, f \rangle + \langle b, c \rangle + \langle
l, n \rangle + \langle m, l \rangle + \langle d, e \rangle
  \end{equation}
  which, using the conservation laws $b + c = n + l$, $m + l = d + e$ and $a + b + c = d + e + f$ equals
  \begin{equation}
    \begin{split}  
      \operatorname{RHS} 
      &= \langle a, b + c - l \rangle + \langle d + e - l, f \rangle + \langle b, c\rangle + \langle l, b + c \rangle + \langle d + e, l \rangle + \langle d, e \rangle \\
      &= \langle a, b + c \rangle + \langle d + e, f \rangle + \langle b, c\rangle + \langle d, e \rangle + \langle l, a + b + c - d - e - f\rangle
    \end{split}
  \end{equation}
  where again the last term is zero because of total conservation of color.
  Thus $\operatorname{LHS} = \operatorname{RHS}$.
\end{proof}

\begin{proof}[Proof of Proposition~\ref{prop:non-standard-DT}]
  We will again show that all terms in the Yang-Baxter equation in~\eqref{eq:YBE} for given boundary edges $a,b,c,d,e$ and $f$ are multiplied with the same factors of $\Phi$.

  The total power of $\Phi$ on the left-hand side is $\langle b - i + k + f \rangle$ and $\langle l - d + n + f \rangle$ on the right-hand side.
  Using the conservation laws $i + c = d + k$ and $l + n = b + c$ we obtain that the relative power between the left-hand side and the right-hand side is
  \begin{equation}
  \langle b - i + k + f \rangle - \langle l - d + n + f \rangle = \langle b + c - d + f \rangle - \langle b + c - d + f \rangle = 0
  \end{equation}
\end{proof}

\begin{proof}[\textbf{Proof of Theorem~\ref{thm:YBE-prime}}]
  The family of lattice models can be expressed as standard $\DT_\phi$ and non-standard $\DT'_\Phi$ Drinfeld twists of a lattice model~$\mathbb{L}$ where, for colors $c_i$ and $c_j$, $\phi(c_i,c_j) = \alpha_{-i,-j}$ and otherwise $1$. As a shorthand notation we may write $\DT_\phi$ of this form as $\DT_\alpha$.
  
  There are two special cases of this family: Iwahori ice and Delta metaplectic ice.
  The Yang-Baxter equations for both the unfused and fused versions of Iwahori ice were proved in Proposition~6.4 and Theorem~6.5 in~\cite{BBBGIwahori}.
  The unfused version of Delta metaplectic ice of this paper is equivalent to the Delta ice model of~\cite{BBB}.
  That this model is Yang-Baxter solvable is proved in~\cite[Theorem~A.3]{BBB}.

  Thus it follows from Propositions~\ref{prop:standard-DT} and ~\ref{prop:non-standard-DT}, which can be applied to both the fused and unfused models, and both the RTT and the RRR equations, that all members of the family of lattice models are Yang-Baxter solvable.
  Alternatively, we may apply the propositions to only Yang-Baxter equations for the unfused models which imply the Yang-Baxter equations for the fused models.

  The last statement that the R-matrix with vertex color $c_k$ for the unfused model can be obtained from the R-matrix of the fused by shifting the color palette such that $c_k$ becomes the first, smallest color in the new palette is seen by inspection from Table~\ref{tab:R-matrix} and~\eqref{eq:R-matrix-unfused-exceptional-weight}.
  Recall that the R-matrix for the fused model is obtained from the unfused model by letting $c_k = c_1$.
  
  It remains to show that the R-matrix for a fused model is a Drinfeld twist of the R-matrix for the $U_q(\widehat{\mathfrak{gl}}(m|1))$ evaluations module.
  The statement was proven for Iwahori ice in~\cite{BBBGIwahori}. 
  Since the R-matrices for all members of the family of lattice models are related by (standard) Drinfeld twists according to Propositions~\ref{prop:standard-DT} and ~\ref{prop:non-standard-DT} this proves the statement.
\end{proof}

\subsection{Proof of Theorem~\ref{thm:specializations-prime}}
\label{sec:specializations-proof}

We now come back to the proof of Theorem~\ref{thm:specializations-prime} which is a refinement of Theorem~\ref{thm:specializations} and which we have split into two subsections corresponding the two different specializations.
The proof is technical, but otherwise straightforward, and relates many different variants of lattice models in the literature painting a web of dualities.
The reader may skip this section at a first reading without affecting the understanding of later sections.
The web of dualities is pictured in Figure~\ref{fig:web} at the end of Section~\ref{sec:specializations} and also gives an overview of the proof.

\subsubsection{Iwahori ice}
\label{sec:Iwahori}

Iwahori ice was first constructed in~\cite{BBBGIwahori} where it was showed that its partition functions compute Iwahori (and parahoric) Whittaker functions for $\GL_r(F)$ for a non-archimedean field $F$ of residue cardinality $v^{-1}$. 
It was there described using colored paths with a palette of size $m=r$, where $r$ is the rank of $\mathbf{GL}_r$. 
The geometry and the set of admissible states of colored paths for the Iwahori model is the same as that of the models introduced in this paper; the only difference is how the boundary data is described.
For the Iwahori ice model the boundary data is specified using the fused description (see Section~\ref{sec:fusion}), while for the ensembles $\mathfrak{S}^m_{\mu,\sigma}$ in this paper it is specified using the unfused description. 

As mentioned in Section~\ref{sec:specializations}, the boundary data for Iwahori ice consists of tuples $(P, \lambda, w_1, w_2)$, where $P$ is a parabolic subgroup of $\GL_r$, and $w_1, w_2 \in W$ are the shortest representatives of some $W/W_P$-cosets where $W_P$ is the Weyl group of the associated Levi subgroup.
Here $\lambda$ is $w_2$-almost dominant (see~\eqref{eq:almost-dominant} below or~\cite[Definition 3.4]{BBBGIwahori} for a definition of $w$-almost dominant). For the parabolic subgroup $P$, there is an associated weakly decreasing $r$-tuple of (possibly nondistinct) colors $\mathbf{c}_P$ which is stabilized by $W_P$. The edges on the right boundary of the system will all be colored according to the $r$-tuple $w_1 \mathbf{c}_P$. 
The boundary condition on top depends on $\lambda$ and $w_2$. The vertical edges in fused column numbers, i.e.\ block numbers, $(\lambda + \rho)_j$ will be colored with the color $(w_2 \mathbf{c}_P)_j$ while the rest of the edges will be uncolored.

\begin{proof}[\textbf{Proof of Theorem~\ref{thm:specializations-prime} Part 1}]
  We will now prove~\eqref{eq:Iwahori-specialization} which shows that the Iwahori specialization of our family of lattice models is the Iwahori ice model of~\cite{BBBGIwahori}.
  The geometry and the set of admissible states of the two lattice models to be compared are the same as explained before, so in order to finish the proof of this statement, we must match the Boltzmann weights and the admissible boundary conditions of the two models.

The monochrome (or unfused) Boltzmann weights for the Iwahori model appear in~\cite[Table 7]{BBBGIwahori} and depend on $v$. These weights are easily seen to be equal to the Iwahori specialization of the weights in Table~\ref{tab:boltzmann-weights} with $q = \sqrt{1/v}$, $\Phi=q^{-1}$, $\alpha_{-i,-j} = \left\{\begin{smallmatrix*}[l] 1/q &\ i<j\\ q  &\ i>j\end{smallmatrix*}\right\}$ for $1\leqslant i,j \leqslant m=r$.
Recall that $\alpha_{i,i}$ should always be $1$.

As explained in the beginning of Section~\ref{sec:specializations}, given $\mu$ and $\sigma$, we may define $\lambda = \left\lfloor \frac{\mu}{r} \right\rfloor - \rho$, and an $r$-tuple $\mathbf{c}$ of color indices (in $\{1, 2, \ldots, m=r\}$ and with multiplicities) for the colors of the paths in weakly decreasing order.
Recall that these colors are given by $\sigma$ for the right boundary and the residues of $-\mu \bmod m$ for the top boundary, and that the colors of the two boundaries have to agree for there to be any admissible states.
We also defined in~\eqref{eq:w-sigma-mu} the shortest Weyl group elements $w_\sigma$ and $w_{-\mu}$ such that $\sigma \equiv w_\sigma \mathbf{c}$ and $-\mu \equiv w_{-\mu} \mathbf{c} \bmod m$, and the stabilizer $W_P \subseteq S_r$ of $\mathbf{c}$.
Then $W_P$ is the Weyl group of the Levi subgroup of a parabolic subgroup $P$ of $G = \GL_r(F)$.  

To understand the equivalence between the two sets of admissible boundary data, we first note that we may relabel colors without loss of generality so long as we keep their internal ordering. 
Thus the parabolic subgroup $P$ contains the same data as the color tuple $\mathbf{c}$, and the $W/W_P$ cosets represented by $w_{\sigma}$ and $w_{-\mu}$ are in bijection with those for $w_1$ and $w_2$ and determine the colors on the boundary in the same way.

From our description of the Iwahori boundary data above and how it specifies the boundary edges and their colors it is easy to see that we have an injective map from the Iwahori data $(P, \lambda, w_1, w_2)$ to the boundary data $(\mu, \sigma)$ of this paper by going from the fused to the unfused description.

For surjectivity it remains to show that $\lambda$ and $w := w_{-\mu}$ obtained from $\mu$ and $\sigma$ as described above satisfy the property that $\lambda$ is $w$-almost dominant meaning that
\begin{equation}
  \label{eq:almost-dominant}
  \lambda_i - \lambda_{i + 1} \geqslant 
  \begin{cases}
    0 & \text{if } w^{- 1} \alpha_i > 0 \\
    - 1 & \text{if } w^{- 1} \alpha_i < 0
  \end{cases}
\end{equation}
where $\alpha_i$ is the simple root $\varepsilon_i - \varepsilon_{i+1}$ in Bourbaki notation.

Note that $\rho_i - \rho_{i + 1} = 1$. For $\lambda + \rho = \left\lfloor\frac{\mu}{r} \right\rfloor$ we thus have that
\begin{equation*}
  \lambda_i - \lambda_{i + 1} = - 1 + \left\lfloor \frac{\mu_i}{r}\right\rfloor - \left\lfloor \frac{\mu_{i + 1}}{r} \right\rfloor \in \mathbb{Z} 
\end{equation*} 
Since $\mu_i > \mu_{i + 1}$ this means that $\lambda_i - \lambda_{i + 1} \geqslant - 1$ with equality if and only if $\left\lfloor \frac{\mu_i}{r} \right\rfloor = \left\lfloor \frac{\mu_{i + 1}}{r} \right\rfloor$.
Recall the definition of the least positive and least nonnegative residues $\res^m(x)$ and $\res_m(x) \bmod{m}$ respectively from \eqref{eq:res}. We have that
\begin{equation*}
  \left\lfloor \frac{\mu_i}{r} \right\rfloor r + \res_r (\mu_i) = \mu_i > \mu_{i + 1} = \left\lfloor \frac{\mu_{i + 1}}{r} \right\rfloor r + \res_r (\mu_{i + 1}). 
\end{equation*}
Thus $\lambda_i - \lambda_{i+1} = -1$ implies that $\res_r (\mu_i) > \res_r (\mu_{i + 1})$.
Since $w\mathbf{c} = w_{-\mu}\mathbf{c} = \res^r(-\mu)$ and $\res^r (-x) + \res_r (x) = r$ this implies in turn that $(w\mathbf{c})_i = \res^r (- \mu_i) < \res^r (-\mu_{i + 1}) = (w\mathbf{c})_{i + 1}$.
Together with the fact that $\mathbf{c}$ is weakly decreasing this implies that $s_i$ is a left
descent of $w$, or equivalently that $w^{-1} \alpha_i$ is a negative root,
and hence that $\lambda$ is $w$-almost dominant.
(See Lemma~\ref{lem:ascent-or-descent}.)

The second equality in equation~\eqref{eq:Iwahori-specialization} is essentially the main theorem of~\cite{BBBGIwahori}. With this, the proof is finished.
\end{proof}

\begin{remark}
  The sequence of colors at the right boundary of the system is an important part of the boundary data of the model, and if the colors are distinct their order as an element of the Weyl group $W = S_r$ specifies a basis element for the Iwahori Whittaker functions.
  If the colors are not distinct they specify a basis element for parahoric Whittaker functions which can be obtained as a sum over Weyl elements in a coset of $S_r$ by the stabilizer of the non-distinct colors.
  On the lattice model side this is expressed as the property that summing over the permutations of a subset of $r$ distinct colors on the right boundary is equivalent to identifying these colors, and this follows from two easily verifiable properties of the Boltzmann weights of the $T$-vertices (called Property A and Property B) as shown in~\cite[Section 8]{BBBGIwahori}.

  These properties do not however hold in general for the family of Boltzmann weights in Table~\ref{tab:boltzmann-weights}; in fact the Iwahori ice specialization in this subsection is the only model in our family for which the properties hold.
  See also Remark~\ref{rem:parahoric-conjecture}.
\end{remark}

\subsubsection{Delta metaplectic ice}
\label{sec:Delta}
In this section we will relate several different, but similar metaplectic lattice models.
All these relations are overviewed in Figure~\ref{fig:web} at the end of Section~\ref{sec:specializations} and we recommend to keep this figure at hand while reading this subsection.

Metaplectic ice was first constructed in~\cite{mice} with further study in~\cite{BBB, StatementB} and it computes spherical Whittaker functions for a metaplectic $n$-cover of $\GL_r(F)$ for a non-archimedean field $F$ containing the $n$-th roots of unity.
There are two versions of metaplectic ice called Gamma and Delta.
In both cases paths start at the top boundary, but they exit at the left boundary for Gamma and at the right boundary for Delta.
There is a relationship between the two models as shown in~\cite{StatementB, wmd5book} which is further refined in this paper in Theorem~\ref{thm:StatementB-refinement}.

Let $v^{-1} = q^{-2}$ be the cardinality of the residue field of $F$.
Note that this is different from the Iwahori case --- see Remark~\ref{rem:v-q}.
To relate to the Delta metaplectic ice model of \cite{BBB} we will from now on until the end of this section consider the supercolor description from Section~\ref{sec:supercolors} of our family of lattice models with palette cardinality $m = n$ and~$r$ the number of rows.
We will use the metaplectic specialization of the weights in Table~\ref{tab:super-boltzmann-weights} for which $\Phi = -q$ and $\alpha_{i,j} = -g(i-j)/q$. 
In addition, according to the statement in Theorem~\ref{thm:specializations-prime} we should also replace the row parameters $z_i$ by $z_i^n$. 
The corresponding weights are shown on row A in Table~\ref{tab:metaplectic-weights} where we note that $g(0) = -v$.
As mentioned in Section~\ref{sec:specializations} the original metaplectic ice models were described using charges.
We will show below that, if translated in terms of charges, the weights in row A of Table~\ref{tab:metaplectic-weights} are similar to a Delta version of the so-called \emph{modified weights} in~\cite[Figure~6]{BBB}.

To obtain the \emph{standard} charged Delta metaplectic ice weights of~\cite[Figure~A.1]{BBB} we need to make a change of basis for the vector space of horizontal edges which changes the Boltzmann weights for both the T- and R-vertices.
The Yang-Baxter solvability of the model is invariant under a change of basis, but the partition function may change.
In our case the partition function will change by a known factor of $z$'s as shown in Lemma~\ref{lem:Delta-Deltaprime}.

For the horizontal edges we start with a basis enumerated by the possible edge assignments (the so called spinset) which, in the supercolor framework is: unoccupied (no supercolor, which we will denote by a plus) or a single supercolor $\bar c_i$.
We multiply the basis element of a horizontal edge $A$ attached to the left of a vertex with column supercolor $\bar c_j$ and row parameter $z$ with the function ${f(A, \bar c_j, z) = \left\{\begin{smallmatrix*}[l] z^{1-\res^n(i-j)} &\ \text{if } A = \bar c_i \\ 1  &\ \text{otherwise}\end{smallmatrix*}\right\}}$. 
Then the T-vertex Boltzmann weight transforms as 
\begin{equation}
  \operatorname{wt} \left(
    \begin{tikzpicture}[scale=0.5, baseline=-1mm, every node/.append style={scale=0.8}, round/.style={draw, circle, fill=white, inner sep=0pt, minimum size=16pt}, dot/.style={circle, fill, inner sep=0pt, minimum size=3pt}] 
      \useasboundingbox (-1.5,-1.5) rectangle (1.5,1.5);
      \draw [thick] (0,-1) node[round] {$D$} -- (0,0) node[dot] {} -- (0,1) node[round, label=above:$\bar c_j$] {$B$};
      \draw [thick] (-1,0) node[round] {$A$} -- (1,0) node[round] {$C$};
    \end{tikzpicture}
   \right) \mapsto \frac{f (A, \bar{c}_j, z)}{f (C,
 \bar{c}_{\res_n(j - 1)}, z)} 
   \operatorname{wt} \left( 
    \begin{tikzpicture}[scale=0.5, baseline=-1mm, every node/.append style={scale=0.8}, round/.style={draw, circle, fill=white, inner sep=0pt, minimum size=16pt}, dot/.style={circle, fill, inner sep=0pt, minimum size=3pt}] 
      \useasboundingbox (-1.5,-1.5) rectangle (1.5,1.5);
      \draw [thick] (0,-1) node[round] {$D$} -- (0,0) node[dot] {} -- (0,1) node[round, label=above:$\bar c_j$] {$B$};
      \draw [thick] (-1,0) node[round] {$A$} -- (1,0) node[round] {$C$};
    \end{tikzpicture}
   \right)
\end{equation}
where we note that $C$ is attached to the left of a vertex with column supercolor $\bar c_{\res_n(j - 1)}$.
The resulting weights after this change of basis are shown in row B of Table~\ref{tab:metaplectic-weights}.

\begin{table}[htpb]
  \centering
  \caption{A: Metaplectic specialization of supercolored Boltzmann weights with $z^n$ which we will call $\Delta'$-weights.
  B:~The same weights after a change of basis which we will call $\Delta$-weights.}
  \label{tab:metaplectic-weights}
  {\tabulinesep=1.2mm
\begin{tabu}{c|c|c|c|c|c|c|}
  \cline{2-7}
  & $\texttt{a}_1$ & $\texttt{a}_2$ & $\texttt{b}_1$ & $\texttt{b}_2$ & $\texttt{c}_1$ & $\texttt{c}_2$ \\\cline{2-7}\noalign{\vskip\doublerulesep}\cline{2-7}

  &
  
  \begin{tikzpicture} 
    \draw [] (0,-1) -- (0,1) node[label=above:$\bar c_j$] {};
    \draw [] (-1,0) -- (1,0);
  \end{tikzpicture}
 
  &
  
  \begin{tikzpicture}
    \draw [blue, ultra thick, dotted] (0,-1) -- (0,1) node[label=above:$\bar c_j$] {}; 
    \draw [red, ultra thick, dotted] (-1,0) -- (1,0) node[label=right:$\bar c_i$] {}; 
  \end{tikzpicture}
 
  &
  
  \begin{tikzpicture}
    \draw [] (-1,0) -- (1,0); 
    \draw [red, ultra thick, dotted] (0,-1) -- (0,1) node[label=above:$\bar c_j$] {};
  \end{tikzpicture}
 
  &
  
  \begin{tikzpicture}
    \draw [] (0,-1) -- (0,1) node[label=above:$\bar c_j$] {};
    \draw [red, ultra thick, dotted] (-1,0) -- (1,0) node[label=right:$\bar c_i$] {}; 
  \end{tikzpicture} 
 
  &
  
  \begin{tikzpicture}
    \draw [] (1,0) -- (0,0) -- (0,1) node[label=above:$\bar c_j$] {};
    \draw [red, ultra thick, dotted] (0,-1) -- (0,0) -- (-1,0);
  \end{tikzpicture}
  
  &
  
  \begin{tikzpicture}
    \draw [] (-1,0) -- (0,0) -- (0,-1);
    \draw [red, ultra thick, dotted] (0,1) node[label=above:$\bar c_j$] {} -- (0,0) -- (1,0);
  \end{tikzpicture}
  
  \\\hline
  \multicolumn{1}{|c|}{A} &
  $1$ & \scriptsize $g(i-j) 
  \begin{cases}
    z^n & i = j\\
    1 & i \neq j
  \end{cases} $
      & $1$ & \scriptsize $\begin{cases}
    z^n & i = j\\
    1 & i \neq j
  \end{cases}$ & $ (1 - v) z^n$ & $1$ \\\hline

  \multicolumn{1}{|c|}{B} &
  $1$ & $g(i-j)z$ & $1$ & $z$ & $ (1 - v) z$ & $1$ \\\hline
\end{tabu}}
\end{table} 

Before stating how this change of basis affects the partition function, let us name the two different sets of weights in row A and B of Table~\ref{tab:metaplectic-weights}.
This will be done by comparing with the charged metaplectic ice models of~\cite{BBB}.

The translation from supercolors to charges is done as follows.
In the charged Delta ice model the horizontal edges are either assigned a plus with charge $0$, or a minus sign with a charge $a\in \mathbb{Z}/n\mathbb{Z}$. That is, the spinset is $\{\oplus\} \cup \{\ominus^a : a \in \mathbb{Z}/n\mathbb{Z}\}$.
Vertical edges are assigned either a plus or minus without any charge.
In both cases the plus sign corresponds to an uncolored edge.

Charges are not constant along a path, but instead grow by one unit for each step to the right.
In the unfused supercolor description the column supercolors also change with each step.
A horizontal edge with supercolor $\bar c_i$ attached to the left of a vertex with column supercolor $\bar c_j$ is equivalent to a charged edge with charge $a \equiv i - j \bmod{n}$:
\begin{equation}
  \label{eq:charge-scolor}
  \begin{tikzpicture}[scale=0.8, baseline=-1mm, round/.style={draw, circle, fill=white, inner sep=0pt, minimum size=16pt}, dot/.style={circle, fill, inner sep=0pt, minimum size=3pt}] 
    \draw [thick, gray!50] (0,-1) node[round, label=below:{\color{black}supercolored}] {$$} -- (0,1) node[round, label=above:$\color{black}\bar c_j$] {};
    \draw [thick, gray!50] (0,0) -- (1,0) node[round] {};
    \draw [ultra thick, dotted, red] (0,0) node[dot] {} -- (-1,0) node[round] {$\bar c_i$};
  \end{tikzpicture}
  \qquad \longleftrightarrow \qquad
  \begin{tikzpicture}[scale=0.8, baseline=-1mm, round/.style={draw, circle, fill=white, inner sep=0pt, minimum size=16pt}, dot/.style={circle, fill, inner sep=0pt, minimum size=3pt}] 
    \draw [thick, gray!50] (0,-1) node[round, label=below:{\color{black}charged}] {$$} -- (0,1) node[round] {};
    \draw [thick, gray!50] (0,0) -- (1,0) node[round] {};
    \draw [thick] (0,0) node[dot] {} -- (-1,0) node[round, label=above:$a$] {$-$};
  \end{tikzpicture}
  \qquad \qquad
  a \equiv i - j \mod{n}.
\end{equation}

With this dictionary it is easy to see that the weights of row B in Table~\ref{tab:metaplectic-weights} agree with the Delta metaplectic ice weights of~\cite[Figure~A.1]{BBB}.
We will call the weights of row B in Table~\ref{tab:metaplectic-weights} the supercolored $\Delta$-weights and those in row A we will call the $\Delta'$-weights because of the similarity to the modified weights in~\cite{BBB} (although there only the modified weights for $\Gamma$ ice were presented).

\begin{lemma}
  \label{lem:Delta-Deltaprime}
  Let $\bar{\mathfrak{S}}^n_{\mu,\theta}$ be an ensemble of states for our family of lattice models in the supercolor description with top boundary occupancy given by the strict partition $\mu$ and the supercolors on the right boundary given by $\theta \in (\mathbb{Z}/n\mathbb{Z})^r$.
  Denote the partition functions of $\bar{\mathfrak{S}}^n_{\mu,\theta}$ using the $\Delta'$-weights and $\Delta$-weights from row A and B in Table~\ref{tab:metaplectic-weights} by $Z^{\Delta'}(\bar{\mathfrak{S}}^n_{\mu,\theta})$ and $Z^\Delta(\bar{\mathfrak{S}}^n_{\mu,\theta})$ respectively. 
  Then,
  \begin{equation}
    \label{eq:Delta-Deltaprime}
    Z^{\Delta}(\bar{\mathfrak{S}}^n_{\mu, \theta})(\mathbf{z}) = \mathbf{z}^{\theta} Z^{\Delta'} (\bar{\mathfrak{S}}^n_{\mu, \theta})(\mathbf{z}) = \mathbf{z}^{\theta} Z(\bar{\mathfrak{S}}^n_{\mu,\theta})(\mathbf{z}^n)|_\text{\rm metaplectic}, 
\end{equation}
where $\mathbf{z}^n = (z_1^n, \ldots, z_r^n)$ and $Z(\bar{\mathfrak{S}}^n_{\mu,\theta})|_\text{\rm metaplectic}$ denotes the metaplectic specialization of our family of lattice models in the supercolor description.
\end{lemma}

\begin{proof}
  The last equality follows from the definition of the $\Delta$'-weights in Table~\ref{tab:metaplectic-weights}.

  For the first equality we will show that the total weight for each state is multiplied by the same factor $\mathbf{z}^\theta$ when going from the $\Delta$-weights to the $\Delta$'-weights.
  Fix a state $\mathfrak{s} \in \bar{\mathfrak{S}}^n_{\mu,\theta}$.
  Comparing the weights in row A and B of Table~\ref{tab:metaplectic-weights} we note that only the $\texttt{a}_2$, $\texttt{b}_2$ and $\texttt{c}_1$ weights changed, and these are the only vertex configurations with a supercolored left-edge.
  We also note that they all replaced a factor of $\left\{\begin{smallmatrix*}[l] z^n &\ i=j\\ 1  &\ i\neq j\end{smallmatrix*}\right.$ with $z$ recalling that the $\texttt{c}_1$ configuration requires that $i=j$.

  Thus, when comparing the total weights of $\mathfrak{s}$ we only need to consider the segments of supercolored horizontal edges along each path for each row.
  We will now argue that all such segments have lengths which are multiples of $n$ except possibly for one segment in each row which is the segment reaching the right boundary.
  Because the left boundary is unoccupied the paths all enter at the top of each row and all but one path in each row continues down to the next row.
  Since the supercolors for vertical edges are restricted by the column supercolors which are repeated every $n$ columns this means that the lengths of the segments that do not exit to the right boundary are indeed multiples of $n$.
  
  Each subsegment of length $n$ at row $i$ contributes to the total weight for $\mathfrak{s}$ with the same factor of~$z_i^n$ for both the $\Delta$- and $\Delta'$-weights.
  Thus, it remains to compare the trailing subsegments exiting at the right boundary modulo subsegments of length $n$.
  Each occupied left-edge in these trailing subsegments at a row $i$ contributes with an extra $z_i$-factor for the $\Delta$-weights compared to the $\Delta'$-weights.
  The number of these left-edges is determined by the column number for where the segment started modulo $n$ which is given by the boundary supercolor index $\theta_i$ and is the same for all states in $\bar{\mathfrak{S}}^n_{\mu,\theta}$.
\end{proof}

As a shorthand notation we will from now on write the above partition functions as
\begin{equation}
  \label{eq:Delta-partition-shorthand}
  Z^\Delta_{\mu,\theta}(\mathbf{z}) := Z^\Delta(\bar{\mathfrak{S}}^n_{\mu,\theta})(\mathbf{z}) \qquad \text{and} \qquad Z^{\Delta'}_{\mu,\theta}(\mathbf{z}) := Z^{\Delta'}(\bar{\mathfrak{S}}^n_{\mu,\theta})(\mathbf{z})
\end{equation}
suppressing the palette cardinality $n$.

We now return to the dictionary between charged and supercolored metaplectic ice models.
Besides the Boltzmann weights we also need to compare the boundary conditions and row parameters.
For the charged Delta ice model of~\cite{BBB} the row parameters $z_1, \ldots, z_r$ match those in this paper and in~\cite{BBBGMetahori,BBBGIwahori}.
The right boundary condition in the charged model is given by the charges $\gamma \in (\mathbb{Z}/n\mathbb{Z})^r$ for the horizontal edges to the right of column $0$ read from the top down.
Note that this is the opposite order compared to the Delta ice model in~\cite{BBBGVertex}.
The right-edges at column $0$ are left-edges for column $-1$ and we get that the corresponding column supercolor $\bar c_j$ of~\eqref{eq:charge-scolor} for these edges is then $\bar c_{n-1}$.
Using~\eqref{eq:charge-scolor} this means that the relationship between the right boundary data for the supercolored description and the charge description is $\gamma \equiv \theta + 1 \bmod n$.

Thus, letting $Z^\text{charged $\Delta$}_{\mu, \gamma}$ denote the partition function for the charged $\Delta$-ice model from~\cite{BBB} with row parameters labeled from the top down, the top boundary positions of the minus signs given by $\mu$ and the right boundary charges given by $\gamma \in (\mathbb{Z}/n\mathbb{Z})^r$ ordered from the top down we have that
\begin{equation}
  \label{eq:Delta-scolors-charged}
  Z^\Delta_{\mu, \theta}(\mathbf{z}) = Z^\text{charged $\Delta$}_{\mu, \theta+1}(\mathbf{z})\, .
\end{equation}

We next want to relate these lattice models to spherical metaplectic Whittaker functions.
To start with, recall that the lattice models in~\cite{BBBGMetahori} called metaplectic Iwahori ice compute a basis of metaplectic Iwahori Whittaker functions by Theorem~A of the same paper.
These models consist of both right-moving paths and left-moving paths (from the top boundary to the left boundary) with the two types of paths required to overlap along vertical edges according to certain rules.
Each right-moving path is assigned a color from a palette of $r$ colors and each left-moving path is assigned a supercolor from a separate palette of $n$ supercolors.
For the purpose of this section let us call this the $n|r$ model.

In a nutshell, we may obtain lattice models for spherical metaplectic
Whittaker functions by taking the model in~\cite{BBBGMetahori} with one
minor modification: instead of taking $n$ supercolors and $r$ colors,
we will take $n$ supercolors and a single color and we will prove that
the resulting lattice model, which we will call the $n|1$ model, represents spherical Whittaker functions.
In this model all horizontal edges carry either a supercolor or the single color, and without loss of any information one can erase the colored paths and only draw the supercolored paths which is how we are able to relate the $n|1$ model to our family of lattice models in Theorem~\ref{thm:StatementB-refinement} below. 

To show that the $n|1$ partition functions compute metaplectic spherical Whittaker functions we will use the fact that the latter are sums over $w \in S_r$ of Iwahori Whittaker functions according 
to~\eqref{eq:spherical-Whittaker}.
As mentioned earlier, the latter are in turn computed by the $n|r$ partition functions.
On the lattice model side each $w\in S_r$ in this sum corresponds to permutations of the colors on the right boundary.
We then prove an analog of~\cite[Theorem~8.3]{BBBGIwahori} in Proposition~\ref{prop:metahori-is-parahoric} which says that the sum over color permutations of the $n|r$ partition function equals a corresponding $n|1$ partition function; that is, the sum over color permutations amounts to equating the colors.
Note that the admissible states of the $n|1$ model are also admissible states of the $n|r$ model of~\cite{BBBGMetahori} and that the $n|1$ model can be obtained from the $n|r$ model by only considering boundary conditions with a fixed single color.
In~\cite{BBBGMetahori} we mostly limited ourselves to boundary conditions with $r$ distinct colors which excludes such states.

\begin{remark}
In principle we could also connect our family of lattice models to spherical Whittaker functions by reinterpreting the models in~\cite{BBB} as
models using supercolors instead of charges. However the representation theoretical foundations in
\cite{BBB} are different from this paper and it is both more convenient
and more general to start with the results of~\cite{BBBGMetahori} and
to observe that an analog of Theorem~8.3 in~\cite{BBBGIwahori} is
valid in the metaplectic case.
\end{remark}

To carry out the above plan we need to first describe the models of~\cite{BBBGMetahori} in more detail.
We note that in~\cite{BBBGMetahori} there are three equivalent descriptions of the $n|r$ models, called \textit{unfused} (or monochrome), 
\textit{color-fused}, and (fully) \textit{fused}. We will use the color-fused description, where the columns are sequentially assigned a supercolor but not a color, to relate the $n|r$ models to the $n|1$ models. 
The $n|r$ models we will require are less general than the most general
models in~\cite{BBBGMetahori}. The reason is that these are
to represent the value of an Iwahori Whittaker function that
is a summand in a spherical Whittaker function, and it is
sufficient to describe these at a particular values $g=\varpi^{-\lambda}$.
Iwahori Whittaker are more generally determined by their values on a larger set of group elements $g$.

Moreover, nonvanishing of the spherical Whittaker function
requires $\lambda$ to be dominant. That is, the top boundary is specified by $\mu=\lambda+\rho$
where $\lambda$ is a dominant weight. Limiting ourselves
to such $g$ amounts to limiting ourselves to boundary conditions for the color-fused $n|1$ model such that 
the top boundary vertical edges, which are in $r$ distinct color-fused columns (or color blocks) $\mu_i$, carry single colors reading $c_r, c_{r-1}, \ldots c_1$ from left to right.
Thus in Theorem~A of \cite{BBBGMetahori} we are concerned with
the case where $w'=1$.
Note that, in contrast to the ensembles $\mathfrak{S}^{m}_{\mu,\sigma}$ for our family of lattice models, here $\mu$ specifies the color-fused column numbers and not the unfused columns.
This is as expected since each color-fused column in~\cite{BBBGMetahori} also is assigned a single supercolor which increases from right to left and will correspond to the unfused columns of this paper in the supercolor description.

The remaining boundary data is given by $\theta \in (\mathbb{Z}/n\mathbb{Z})^r$ determining the supercolors on the left boundary and a permutation $w \in W = S_r$  of the decreasing color tuple $(c_r, c_{r-1}, \ldots, c_1)$ determining the colors on the right boundary, both read from top to bottom.
We denote the corresponding partition function with Boltzmann weights as in
Figure~10 (or equivalently Figure~4 for the unfused description) of~\cite{BBBGMetahori} by~$Z^\Gamma_{\mu,\theta,w}$ suppressing the supercolor cardinality~$n$.
Then by Theorem~A of \cite{BBBGMetahori} the metaplectic Iwahori Whittaker function $\phi^{(n)}_{\theta,w}$ is given by
\begin{equation}
  \label{eq:metahori-Thm-A}
  \mathbf{z}^\rho\phi^{(n)}_{\theta,w}(\mathbf{z}; \varpi^{-\lambda}) =Z^\Gamma_{\mu,\theta,w}(\mathbf{z}).
\end{equation}
We note that these Boltzmann weights in the earlier preprint versions of~\cite{BBBGMetahori} uses a different convention for the Gauss sums compared to for example~\cite{BBB}.
To match them we replace the Gauss sum $g(x)$ in~\cite{BBBGMetahori} with $g(-x)$.

\begin{proposition}
\label{prop:metahori-is-parahoric}
Let $\mu$ be a strict partition and $Z^\Gamma_{\mu,\theta}$ be the partition function for the single-color model obtained from the $n$-metaplectic Iwahori $\Gamma$-ice model of~\cite{BBBGMetahori} with the color-fused column numbers for the occupied top boundary edges given by $\mu$ as described above, $\theta \in (\mathbb{Z}/n\mathbb{Z})^r$ denoting the supercolors on the left boundary with row parameters ordered from the top down, and where the right boundary edges are all of a single color $c = c_1$.
Then,
\begin{equation}
  \label{eq:metahori-is-parahoric}
  Z^\Gamma_{\mu,\theta}(\mathbf{z}) =
  \sum_{w \in S_r} Z^\Gamma_{\mu,\theta,w}(\mathbf{z}) = 
  \mathbf{z}^\rho \tilde \phi_\theta^\circ(\mathbf{z}; \varpi^{\rho-\mu}).
\end{equation}
\end{proposition}

\begin{proof}[Proof (sketch)]
  The second equality follows from~\eqref{eq:metahori-Thm-A} and the definition of $\tilde \phi_\theta^\circ$ in~\eqref{eq:spherical-Whittaker}.
  
  The first equality is proved by generalizing the proof Theorem~8.3 in~\cite{BBBGIwahori}
to the metaplectic case, taking $\mathbf{J}=\mathbf{I}$ in the
notation of that result, since we are interested in the spherical
case. The tree argument in that proof proceeds by replacing
a color $c'$ by $c$, where $c$ and $c'$ are adjacent colors,
that is, assuming $c>c'$, that there is no color $c''$ such that $c>c''>c'$.
This requires generalizing Properties~A and~B in Section~8
of \cite{BBBGIwahori} to the color-fused weights that are derived
from the Boltzmann weights in Figure~4 of
\cite{BBBGIwahori}. 
Note that in the color-fused weights if a vertical
edge carries two color-scolor pairs $(c,\theta)$ and
$(c',\theta')$, then $\theta=\theta'$. Hence the process of
changing $c'$ to $c$ does not affect scolors. With this
in mind, the generalization of Properties~A and~B to the
metaplectic case is straightforward.
\end{proof}

The Boltzmann weights for the $n|1$ partition function $Z^\Gamma_{\mu,\sigma}$ are shown in Table~\ref{tab:gamma-boltzmann-weights} where the colored paths have been suppressed after color-fusion.
These weights can in fact be identified directly from the unfused weights of Figure~4 of~\cite{BBBGMetahori}, because the color-fusion process is trivial if we have restricted to a single color.
Indeed every unfused vertex of color $c'\neq c$ will
be in configuration $\tt{a}_1$ or $\tt{c}_2$, contributing
a factor of $1$ to the partition function.
In other words, taking the color-fused model and restricting to boundary data
involving only a single color $c$ is equivalent to making
a model using only the weights in \cite[Figure~4]{BBBGMetahori} 
with a single color $c$. Eliminating columns corresponding
to the other colors does not change the partition function.

\begin{table}[htpb]
  \centering
  \caption{Gamma Boltzmann weights for the partition function $Z^\Gamma_{\mu,\theta}$. Except for a convention regarding the Gauss sums $g(j-i)$, these
  are the same as the Boltzmann weights from Figure~4 of~\cite{BBBGMetahori} using only a single color.
  The labels $\tt{a}_1,\tt{a}_2$ etc.\ differ from those of~\cite{BBBGMetahori} due
  to our use of scolor conventions in this paper.}
  \label{tab:gamma-boltzmann-weights}
\begin{tabular}{|c|c|c|c|c|c|}
  \hline
  $\texttt{a}_1$ & $\texttt{a}_2$ & $\texttt{b}_1$ & $\texttt{b}_2$ & $\texttt{c}_1$ & $\texttt{c}_2$ \\\hline\hline
  
  \begin{tikzpicture} 
    \draw [] (0,-1) -- (0,1) node[label=above:$\bar c_j$] {};
    \draw [] (-1,0) -- (1,0);
  \end{tikzpicture}
 
  &
  
  \begin{tikzpicture}
    \draw [blue, ultra thick, dotted] (0,-1) -- (0,1) node[label=above:$\bar c_j$] {}; 
    \draw [red, ultra thick, dotted] (-1,0) -- (1,0) node[label=right:$\bar c_i$] {}; 
  \end{tikzpicture}
 
  &
  
  \begin{tikzpicture}
    \draw [] (-1,0) -- (1,0); 
    \draw [red, ultra thick, dotted] (0,-1) -- (0,1) node[label=above:$\bar c_j$] {};
  \end{tikzpicture}
 
  &
  
  \begin{tikzpicture}
    \draw [] (0,-1) -- (0,1) node[label=above:$\bar c_j$] {};
    \draw [red, ultra thick, dotted] (-1,0) -- (1,0) node[label=right:$\bar c_i$] {}; 
  \end{tikzpicture} 
 
  &
  
  \begin{tikzpicture}
    \draw [] (-1,0) -- (0,0) -- (0,1) node[label=above:$\bar c_j$] {};
    \draw [red, ultra thick, dotted] (0,-1) -- (0,0) -- (1,0);
  \end{tikzpicture}
  
  &
  
  \begin{tikzpicture}
    \draw [] (1,0) -- (0,0) -- (0,-1);
    \draw [red, ultra thick, dotted] (0,1) node[label=above:$\bar c_j$] {} -- (0,0) -- (-1,0);
  \end{tikzpicture}
  
  \\\hline
  $z$ & $g(j-i)$ & $z$ & $1$ & $(1-v)z$ & $1$ \\\hline
\end{tabular}
\end{table}

Using a similar dictionary as in~\eqref{eq:charge-scolor} one can also show that weights in Table~\ref{tab:gamma-boltzmann-weights} are equivalent to the weights for the model in~\cite{BBB} which uses charges instead of supercolors.
They are also closely related to the weights in~\cite{BBBGVertex}, where they are multiplied by a factor of $z^{-1}$ in order to make the infinite models of that paper convergent. 

It remains relate the $n|1$ $\Gamma$-ice partition function $Z^\Gamma_{\mu,\theta}$ to the $\Delta$-ice partition functions of this paper. 
For this end, the following theorem is a refinement of \cite[Theorem~A.1]{BBB} which is recovered by
taking the sum over $\theta$ on both sides.
It is also a concretization of~\cite[Proposition~6.1]{BBBGVertex}.

\begin{theorem}\label{thm:StatementB-refinement}
  Let $Z^\Gamma_{\mu,\theta}$ be the single-colored $\Gamma$-ice partition function from Proposition~\ref{prop:metahori-is-parahoric} and $Z^\Delta_{\mu,\theta}$ be the $\Delta$-ice partition function~\eqref{eq:Delta-partition-shorthand}. Then,
\begin{equation}
  \label{eq:StatementB-refinement}
  Z^\Gamma_{\mu,\theta}(\mathbf{z}) = Z^\Delta_{\mu, w_0 \theta} (w_0 \mathbf{z}) .
\end{equation}
\end{theorem}

\begin{proof}
  We have that both sides of~\eqref{eq:StatementB-refinement} are equal to their respective charged models (with some bijective dictionaries for the left and right boundary data as in~\eqref{eq:Delta-scolors-charged}).
  When summing over these boundary data the resulting equality $\sum_\theta Z^\Gamma_{\mu,\theta}(\mathbf{z}) = \sum_{\theta'} Z^\Delta_{\mu, \theta'} (w_0 \mathbf{z})$ is equivalent to~\cite[Theorem~A.1]{BBB} which was originally expressed in terms of charged models.
  Here the long Weyl element $w_0$ in the argument $w_0 \mathbf{z}$ on the right-hand side comes from the fact that the Gamma and Delta models in~\cite{BBB} labeled the rows from the bottom up and from the top down respectively while we here (and in~\cite{BBBGMetahori}) label them both from the top down.
  
  Now, we showed in~\cite[Theorem~A and Proposition~3.14]{BBBGMetahori} that each $\theta$-component of $Z^\Gamma_{\mu,\theta}(\mathbf{z})$ is an element of $\mathbf{z}^{\theta} \mathbb{C} [\mathbf{z}^n]$ and thus that the operation of taking the sum over $\theta$ is reversible.
  
  From Lemma~\ref{lem:Delta-Deltaprime} we see that $Z^\Delta_{\mu, w_0 \theta}(w_0 \mathbf{z})$ is also an element of $\mathbf{z}^{\theta} \mathbb{C} [\mathbf{z}^n]$ and since their $\theta$-sums are equal with distinctly $\mathbf{z}$-supported terms, the two $\mathbf{z}^\theta$-terms must therefore be equal.
\end{proof}

\begin{proof}[\textbf{Proof of Theorem~\ref{thm:specializations-prime} Part 2}]
  We will now prove \eqref{eq:metaplectic-specialization}.
  Combining Proposition~\ref{prop:metahori-is-parahoric} with Theorem~\ref{thm:StatementB-refinement}, Lemma~\ref{lem:Delta-Deltaprime}, \eqref{eq:Delta-scolors-charged} and~\eqref{eq:S-bar-notation} we get that
  \begin{multline}
    \mathbf{z}^\rho \tilde \phi^\circ_\theta(\mathbf{z}; \varpi^{\rho-\mu}) = Z^\Gamma_{\mu,\theta}(\mathbf{z}) = Z^\Delta_{\mu,w_0\theta}(w_0 \mathbf{z}) = Z^\text{charged $\Delta$}_{\mu, w_0\theta + 1}(w_0 \mathbf{z}) ={} \\ {}= \mathbf{z}^\theta Z(\bar{\mathfrak{S}}^n_{\mu,w_0\theta})(w_0\mathbf{z}^n)|_\text{metaplectic} = \mathbf{z}^\theta Z(\mathfrak{S}^n_{\mu, -w_0\theta})(w_0\mathbf{z}^n)|_\text{metaplectic}
  \end{multline}
  The sets of boundary data for the charged and supercolored models are very similar and the bijection between them is clear.
\end{proof}

\section{Demazure recursions for finite systems\label{sec:demaction}}

We return now to the general setting of Section~\ref{sec:results}, that is, the full family of lattice models whose systems are described by Boltzmann weights in Table~\ref{tab:boltzmann-weights} according to a choice of color palette size $m$ and complex parameters $\Phi$ and $\alpha_{i,j}$ with $i, j \in \{1, \ldots, m\}$ such that $\alpha_{i,j} \alpha_{j,i}=1$.

Recall that in our finite systems with $r$ rows, the right-hand boundary edge is dictated by a choice of $\sigma$ in $(\mathbb{Z} / m \mathbb{Z})^r$ (see Figure~\ref{fig:boundary-conventions}). When the components $\sigma_i$ and $\sigma_{i+1}$ are not equal, then solvability and the standard train argument (see for example \cite[Figure~18]{BBBGMetahori}) give a relation on partition functions associated to $\sigma$ and $s_i (\sigma)$ in Corollary~\ref{maincor}. This relation can be rewritten in terms of familiar operators. For convenience, we define the following function on $(\mathbb{Z}/m\mathbb{Z})^r$ which renormalizes the constants $\alpha_{i,j}$ as
\begin{equation} 
  \label{eq:alphanorm} 
  \tilde{\alpha}_{i}(\sigma) := \alpha_{- \sigma_{i + 1}, - \sigma_i} \cdot 
  \begin{cases}
    q & \textrm{if } \sigma_i > \sigma_{i + 1},\\
    1 & \textrm{if } \sigma_i = \sigma_{i + 1},\\
    q^{- 1} & \textrm{if } \sigma_i < \sigma_{i + 1}.
  \end{cases}
\end{equation}
Then \eqref{eq:Demazure-op} of Corollary~\ref{maincor} can be rewritten as
\begin{equation} Z(\mathfrak{S}^m_{\mu, s_i \sigma})(\mathbf{z}) = \tilde{\alpha}_i(\sigma) \cdot \begin{cases} T_i \cdot Z(\mathfrak{S}^m_{\mu,\sigma})(\mathbf{z}) & \textrm{if } \sigma_i > \sigma_{i+1}, \\ T_i^{-1} \cdot Z(\mathfrak{S}^m_{\mu,\sigma})(\mathbf{z}) & \textrm{if } \sigma_i < \sigma_{i+1}, \end{cases} 
\label{eq:modified-op} \end{equation}
where the operators $T_i^{\pm 1}$ are given by
\begin{align} 
  \label{eq:demazurewhittaker}
  T_i\cdot f(\z) &= \frac{z_{i} -q^{-2} z_{i+1}}{z_{i+1} - z_{i}} f(s_i \z) + \frac{(q^{-2}-1)z_{i+1}}{z_{i+1} - z_{i}} f(\z) \\
  \intertext{and}
  \label{eq:demazurewhittakershort}
  T_i^{-1} \cdot f(\z) &= \frac{z_{i} - q^{-2} z_{i+1}}{q^{-2}(z_{i+1} - z_{i})} f(s_i \z) + \frac{(q^{-2}-1) z_{i}}{q^{-2}(z_{i+1} - z_{i})} f(\z) \, .
\end{align}
For $\mathbf{z}$ taking values in the complex dual torus $\hat T \cong (\mathbb{C}^\times)^r$, these $T_i^{\pm 1}$ act on $\mathcal{O}(\hat T)$, the ring of regular polynomial functions on $\hat T$. These operators exactly match Equations (30) and (31) of \cite{BBBGIwahori} upon setting the parameter $v$ there to our $q^{-2}$. The above relation \eqref{eq:modified-op} on partition functions is then analogous to Proposition 7.1 of \cite{BBBGIwahori}. It was demonstrated in \cite{BBBGIwahori} that the $T_i$ give a Hecke algebra action on $\mathcal{O}(\hat T)$; in particular, the operators $T_i$ satisfy the braid relations. Thus, to any reduced expression for $w = s_{i_1} \cdots s_{i_k}$, we may set $T_w := T_{i_1} \cdots T_{i_k}$ and the result is well-defined.

Similarly, if one defines the composition of $\tilde \alpha_i$'s for $i \in \{1, \ldots, r-1\}$ by 
\begin{equation}
  (\tilde{\alpha}_{i_1} \circ \tilde{\alpha}_{i_2} \circ \cdots \circ \tilde{\alpha}_{i_k}) (\sigma) :=
  \tilde{\alpha}_{i_1} (s_{i_2} \cdots s_{i_k} \sigma) \tilde{\alpha}_{i_2} (s_{i_3} \cdots s_{i_k} \sigma) \cdots \tilde{\alpha}_{i_k} (\sigma)
\end{equation}
then one can show that $\tilde \alpha_i$ satisfies the same quadratic and braid relations as the simple reflection $s_i \in S_r$ under this composition.
We may thus for any Weyl word $w = s_{i_1} \cdots s_{i_k}$ define
\begin{equation}
  \tilde{\alpha}_w \assign \tilde{\alpha}_{i_1} \circ \cdots \circ \tilde{\alpha}_{i_k}.
\end{equation}

Using~\eqref{eq:modified-op} we can, given the partition function $Z(\mathfrak{S}_{\mu, \sigma}^m)$ compute the partition function for any color permutation $Z(\mathfrak{S}_{\mu, w \sigma}^m)$ with $w \in S_r$.
We may take any path from $\sigma$ to $w \sigma$ by applying different simple reflections $s_i$ and the condition for when we should apply $T_i$ or its inverse to the polynomial is determined by whether $s_i$ is an ascent or descent along the path --- see Corollary~3.9 of~\cite{BBBGIwahori} for a more detailed, similar statement for Iwahori Whittaker functions.

We will in the upcoming theorem pick a particularly simple path which ensures that we will always be in the $\sigma_i > \sigma_{i+1}$ case of~\eqref{eq:modified-op} and therefore only need to apply a product of $T_i$ operators without inverses.
To show this we first need the following lemmas.

\begin{lemma}
  \label{lem:ascent-or-descent}
  Let $\hat \sigma = (\hat \sigma_1, \ldots \hat \sigma_r) \in \mathbb{Z}^r$ be a weakly decreasing tuple, that is, $\hat \sigma_i \geqslant \hat \sigma_{i+1}$.
  Let $w \in S_r$ and $\sigma = (\sigma_1, \ldots, \sigma_r) = w \hat \sigma$.
  \begin{enumerate}[label=\textup{(}\roman*\textup{)}]
    \item If $\sigma_i < \sigma_{i+1}$ then $s_i w < w$.
    \item If $\sigma_i > \sigma_{i+1}$ then $s_i w > w$.
  \end{enumerate}
\end{lemma}

\begin{proof}
  Note that $s_i w < w$ if and only if $w^{-1} \alpha_i$ is a negative root and $s_i w > w$ if and only if $w^{-1} \alpha_i$ is a positive root.  
  We write the roots in terms of the standard Bourbaki basis $\{\varepsilon_i\}_{i=1}^r$ where $\alpha_i = \varepsilon_i - \varepsilon_{i+1}$ and any positive root can be expressed as $\varepsilon_i - \varepsilon_j$ which is positive if and only if $i < j$.
  In this basis, the inner product on the root space is then the standard Euclidean inner product $\varepsilon_i \cdot \varepsilon_j = \delta_{ij}$ and $w \in W = S_r$ acts by permutations.
  We have that
  \begin{equation}
    \sigma_i - \sigma_{i+1} = \alpha_i \cdot \sigma = \alpha_i \cdot w \hat \sigma =  w^{-1} \alpha_i \cdot \hat \sigma,
  \end{equation}
  where we have used that the inner product is invariant under permutations.
  Then, since $\hat \sigma$ is weakly decreasing $\sigma_i - \sigma_{i+1} = w^{-1} \alpha_i \cdot \hat \sigma < 0$ implies that $w^{-1} \alpha_i$ is a negative root and similarly $\sigma_i - \sigma_{i+1} = w^{-1} \alpha_i \cdot \hat \sigma > 0$ implies that $w^{-1} \alpha_i$ is a positive root.
\end{proof}

We will often apply the above lemma to $\sigma,\hat{\sigma} \in (\mathbb{Z}/m\mathbb{Z})^r$ with their parts identified by representatives in $\{1, \ldots, m\}$.
Fix $\hat \sigma \in (\mathbb{Z}/m\mathbb{Z})^r$ and let $\mathbf{J}$ denote the indices for the simple reflections in $W = S_r$ that stabilize $\hat{\sigma}$. They generate the stabilizer $W_{\mathbf{J}}$ of $\hat{\sigma}$. Let $W^{\mathbf{J}} := \{ w \in W : w s_j > w \text{ for all } j \in \mathbf{J} \}$. Then $W$ factorizes uniquely as $W^{\mathbf{J}} W_{\mathbf{J}}$ and every coset $w W_{\mathbf{J}}$ has a unqiue shortest representative $w \in W^{\mathbf{J}}$ by Proposition~2.4.4 of~\cite{BjornerBrenti} and its corollary.

\begin{lemma}
  \label{lem:ascent}
  Let $w \in W^{\mathbf{J}}$ and $s_i, w' \in W$ such that $w = s_i w' > w'$. Then $w' \in W^{\mathbf{J}}$ and $(w' \hat{\sigma})_i > (w'\hat{\sigma})_{i + 1}$.
\end{lemma}
\begin{proof}
  Let $\sigma = w' \hat \sigma$.
  By Lemma~\ref{lem:ascent-or-descent} $\sigma_i < \sigma_{i+1}$ implies $s_i w' < w'$ which is a contradiction.
  If $\sigma_i = \sigma_{i+1}$ then $w\hat\sigma = w'\hat\sigma$, but $w$ is by assumption the shortest permutation of $\hat\sigma$ to $w\hat\sigma$ and $w = s_iw' > w'$ which is also a contradiction.
  Thus $\sigma_i > \sigma_{i+1}$ proving the second statement.

  For the first statement assume that there exists $j \in \mathbf{J}$ such that $w's_j < w'$ which would be necessary for $w' \not\in W^\mathbf{J}$.
  Then $w' s_j < w' < s_i w' = w < w s_j = s_i w' s_j$ where we have used that $w \in W^\mathbf{J}$.
  Each inequality adds $+1$ to the length which means that $\ell(w's_j) = \ell(s_iw's_j) + 3$ but this is a contradiction since $\ell(s_i y) = \ell(y) \pm 1$ for any $y \in W$.
  Hence $w' \in W^\mathbf{J}$. 
\end{proof}

\begin{proposition} \label{prop:sigma-recursion}
  Let $\sigma \in (\mathbb{Z}/m\mathbb{Z})^r$ and let $w \in W$ be the shortest permutation such that $\hat \sigma := w^{-1} \sigma$ is a weakly decreasing $r$-tuple as representatives in $\{1, \ldots, m\}$.
  Then  
  \begin{equation}
    \label{eq:sigma-recursion}
    Z (\mathfrak{S}_{\mu, \sigma}^m) (\mathbf{z}) = \tilde{\alpha}_w (\hat{\sigma}) T_w Z (\mathfrak{S}_{\mu, \hat{\sigma}}^m)(\mathbf{z})
  \end{equation}
\end{proposition}

\begin{proof}
  Let $w = s_{i_k} \cdots s_{i_1}$ be a reduced word.
  We can compute $Z(\mathfrak{S}_{\mu, \sigma}^m)$ from $Z(\mathfrak{S}_{\mu, \hat \sigma}^m)$ by repeated applications of the Demazure operators using~\eqref{eq:modified-op} for each simple reflection~$s_{i_j}$.
  For $1 \leqslant j \leqslant k$ enumerating these steps let $w_j := s_{i_{j-1}} \cdots s_{i_1}$ with $w_1 = 1$ and let $w_{k+1} := w$.
  By repeated use of Lemma~\ref{lem:ascent} (with $w = w_{j+1}$, $w' = w_j$ and $s_i = s_{i_j}$) we get that $(w_j \hat \sigma)_{i_j} > (w_j \hat \sigma)_{i_j + 1}$.
  This means that each step corresponds to the first case in~\eqref{eq:modified-op}, that is, we should apply $T_{i_j}$ at each step and not its inverse.
  The statement now follows from the definitions of $\tilde \alpha_w$ and~$T_w$.
\end{proof}

Since $T_w$ is an invertible operator we may use Proposition~\ref{prop:sigma-recursion} to relate any two partitions in the same $\sigma$-orbit.
The theorem becomes especially powerful when one (or more) of the partition functions in a $\sigma$-orbit is easily computed.
We will show that this is the case when $\sigma$ has distinct parts.
Let us call a system $\mathfrak{S}$ {\emph{monostatic}} if it has a unique (admissible) state. According to our weights in Table~\ref{tab:boltzmann-weights}, if $\mathfrak{S}$ is monostatic, then the partition function $Z(\mathfrak{S}^m_{\mu, \sigma})(\mathbf{z})$ will be a monomial in $\mathbf{z}$ multiplied by a complex constant. We next determine boundary conditions depending on $\mu$ and $\sigma$ such that the system is monostatic. 

\begin{proposition} \label{prop:monostatic} 
  Let $\sigma_{\mathrm{dist}} \in (\mathbb{Z}/ m\mathbb{Z})^r$ have distinct parts. The system $\mathfrak{S}_{\mu, \sigma_{\mathrm{dist}}}^m$ is monostatic if $\sigma_{\mathrm{dist}} \equiv - \mu \bmod m$ and in this case the partition function equals
  \begin{equation}
    \Phi^{r (r - 1) / 2} \mathbf{z}^{\left\lfloor \frac{\mu}{m} \right\rfloor} \cdot \prod_{1 \leqslant i < j \leqslant r} \alpha_{- (\sigma_{\mathrm{dist}})_i, -(\sigma_{\mathrm{dist}})_j} 
  \end{equation}
\end{proposition}

\begin{proof}
  Because the parts of $\mu$ and $\sigma_\mathrm{dist}$ are distinct, our admissible states consist of distinctly colored paths that move down and to the right from the top boundary to the right-hand boundary. The choice of boundary conditions $ \sigma_\mathrm{dist} \equiv -\mu$ forces forces each colored path to travel straight down along the column number $\mu_i$ with column color $(\sigma_\mathrm{dist})_i$ and then turn at row $i$ and continue straight, exiting out the right boundary. The condition that no two paths occupy a horizontal edge guarantees that this is the unique such state. The Boltzmann weight of the state follows by comparison with Table~\ref{tab:boltzmann-weights}, noting that all crossings involve colors $c_i$ and $c_j$ with $i \ne j$, according to our assumption on distinct colors. 
\end{proof}

In earlier work of the authors, the unique state was called the ``ground state'' and illustrations of these monostatic systems were provided in Figure~14 of~{\cite{BBBGIwahori}} and Figure~12 of~{\cite{BBBGMetahori}}. 
Combining the previous two results, we get the following very general result.

\begin{theorem}
  \label{thm:recursion}
  Let $\sigma_{\mathrm{dist}} \in (\mathbb{Z}/ m\mathbb{Z})^r$ have distinct
  parts. The partition function $Z (\mathfrak{S}_{\mu,
  \sigma_{\mathrm{dist}}}^m)$ is nonzero only if the residue classes of $- \mu
  \bmod m$ are a permutation of $\sigma_\mathrm{dist}$, and in this case
  \[ Z (\mathfrak{S}_{\mu, \sigma_{\mathrm{dist}}}^m)(\z) = q^{\ell (w_0 w) - \ell
     (w)} \Phi^{r (r - 1) / 2} \frac{\tilde{\alpha}_w
     (\hat{\sigma})}{\tilde{\alpha}_{w_0 w'} (\hat{\sigma})} T_w (T_{w'})^{-
     1} \mathbf{z}^{\left\lfloor \frac{\mu}{m} \right\rfloor} \]
     where $w, w' \in S_r$ are the unique permutations such that $\sigma_\mathrm{dist} = w
  \hat{\sigma}$ and $- \mu \equiv w' \hat{\sigma}  \bmod m$ where
  $\hat{\sigma}$ is a decreasing $r$-tuple as representatives in $\{ 1,
  \ldots, m \}$.
\end{theorem}

\begin{proof}
  The condition on $\mu$ follows from Lemma~\ref{lem:non-zero}.  
  For such $\mu$, define $\sigma_{\mu} \in (\mathbb{Z}/ m\mathbb{Z})^r$ by $\sigma_{\mu} \assign w' \hat{\sigma} \equiv - \mu \bmod m$. Using Proposition~\ref{prop:monostatic} we get that
  \begin{equation*}
    Z(\mathfrak{S}_{\mu, \sigma_{\mu}}^m)(\z) = q^{\ell (w_0 w) - \ell (w)} \frac{\Phi^{r (r - 1) / 2}}{\tilde{\alpha}_{w_0} (\sigma_{\mu})}\mathbf{z}^{\left\lfloor \frac{\mu}{m} \right\rfloor}     
  \end{equation*}
  where we have used the fact that $\alpha_{ij} = 1/\alpha_{ji}$ and
  $
    \tilde{\alpha}_{w_0} (\sigma ) = q^{\ell (w_0 w) - \ell (w)} \prod_{1
\leqslant i < j \leqslant r} \alpha_{- \sigma_j, - \sigma_i}
$
  because $w_0$ reverses $\sigma$ producing this $\alpha$-product and each ascent introduces a factor of $q$.
  
  On the other hand, using Proposition~\ref{prop:sigma-recursion} we have that
  \begin{equation*}
    Z (\mathfrak{S}_{\mu, \sigma_{\mu}}^m)(\z) = 
    \tilde{\alpha}_{w'} (\hat{\sigma}) T_{w'} Z (\mathfrak{S}_{\mu, \hat{\sigma}}^m)(\z) \quad \text{and} \quad Z (\mathfrak{S}_{\mu, \sigma_{\mathrm{dist}}}^m)(\z) =
    \tilde{\alpha}_w (\hat{\sigma}) T_w Z (\mathfrak{S}_{\mu,\hat{\sigma}}^m)(\z).
  \end{equation*}
  Thus,
  \begin{equation*}
    Z(\mathfrak{S}_{\mu, \sigma_{\mathrm{dist}}}^m)(\z) = 
    \frac{\tilde{\alpha}_w(\hat{\sigma})}{\tilde{\alpha}_{w'} (\hat{\sigma})} T_w (T_{w'})^{- 1} Z(\mathfrak{S}_{\mu, \sigma_{\mu}}^m)(\z) = 
    q^{\ell (w_0 w) - \ell (w)} \Phi^{r (r - 1) / 2} \frac{\tilde{\alpha}_w(\hat{\sigma})}{\tilde{\alpha}_{w_0 w'} (\hat{\sigma})} T_w (T_{w'})^{-1} \mathbf{z}^{\left\lfloor \frac{\mu}{m} \right\rfloor}
  \end{equation*}
  since $\tilde{\alpha}_{w_0} (\sigma_{\mu}) \tilde{\alpha}_{w'}(\hat{\sigma}) = \tilde{\alpha}_{w_0} (w' \hat{\sigma}) \tilde{\alpha}_{w'}(\hat{\sigma}) = \tilde{\alpha}_{w_0} \circ \tilde{\alpha}_{w'}(\hat{\sigma}) = \tilde{\alpha}_{w_0 w'} (\hat{\sigma})$.
\end{proof}

We are principally interested in two specializations of our model --- the Iwahori specialization and the metaplectic specialization. For the Iwahori specialization of our weights, similar results to the preceding two were used in Theorem~7.2 of~\cite{BBBGIwahori} to show that partition functions of a solvable lattice model satisfied the same Demazure-like recursions as (non-metaplectic) Iwahori Whittaker functions at certain special values that uniquely determine them. Since spherical metaplectic Whittaker functions are described by the metaplectic specialization, we may apply the preceding results in this special case and conclude, somewhat surprisingly, that spherical metaplectic Whittaker functions are also described by the Demazure-like operators $T_w$ applied to monomials in $\mathbf{z}$. 

This is what is meant by Iwahori-metaplectic duality; because the two families of special functions are special cases of the same solvable lattice model, and because so many properties of these functions may be derived from their rendering on this lattice model, then we may think of these two families as ``dual'' in this light. 
In other words, they are two sides of the same coin.
In particular they are in many cases (as will be shown in detail below) very similar as polynomials in $\z$ or $\z^n$ obtained by applying the same Demazure-like operators $T_w$ to a monomial with the only difference being an overall factor of Gauss sums.

\begin{remark}
\label{rem:monostatic}
In fact, we may generalize the above arguments in Theorem~\ref{thm:recursion} to any time a system in the same $\sigma$-orbit is monostatic.
Indeed, the partition function of any monostatic system is a monomial in $\z$ and any partition function in the same $\sigma$-orbit is obtained by applying the Demazure operator $T_w$ together with an $\tilde\alpha_w$ factor.
Since only the coefficient of the original monomial and the $\tilde\alpha_w$ factor depend on $\Phi$ and $\alpha_{i,j}$ the Iwahori and metaplectic specializations of these partition functions agree up to an overall factor of Gauss sums.
We work out the specific constants appearing in the case where the parts of $\sigma$ are distinct in the next subsection.
\end{remark}

\subsection{Demazure recursions for spherical metaplectic Whittaker functions}
\label{sec:demazure-metaplectic-whittaker}

In the metaplectic specialization, we choose the size of the palette $m$ to be equal to the degree $n$ of the metaplectic cover, as well as let the parameters $\alpha_{-\sigma_i, -\sigma_j} = -g(\sigma_j - \sigma_i) / q$ and $\Phi = -q$.
We recall from Theorem~\ref{thm:specializations-prime} that the metaplectic specialization of our partition functions are in bijection with the values of spherical metaplectic Whittaker functions by
\begin{equation}
  \z^{\rho-\theta}\tilde{\phi}^\circ_\theta(\z; \varpi^{\rho-\mu}) = Z(\mathfrak{S}_{\mu, -w_0\theta})(w_0 \z^n)|_\text{metaplectic}.
\end{equation}
As also stated by Corollary~\ref{cor:n-th-power} this is a polynomial in $\z^n$ and we may thus replace $\z$ with $w_0 \z^{1/n}$ to recover the standard partition function on the right-hand side without introducing any ambiguity.
Combined with Corollary~\ref{maincor} and~\eqref{eq:modified-op} we get the following recurrence relations for the spherical metaplectic Whittaker function.
\begin{proposition} 
  \label{prop:metaplectic-single-recurrence}
  Let $(\mathbb{C}^\times)^r \ni \mathbf{y} := w_0 \z^{1/n}$ (componentwise) and $\sigma = -w_0\theta$.
  Then 
  \begin{equation*}
    Z(\mathfrak{S}_{\mu, \sigma})(\z)|_\mathrm{metaplectic} = \mathbf{y}^{\rho - \theta} \tilde \phi^\circ_{\theta}(\mathbf{y}; \varpi^{\rho - \mu})
  \end{equation*}
  which is well-defined and independent of the choice of branch.
  Furthermore, the spherical metaplectic Whittaker functions satisfy the recurrence relations
  \begin{equation*}
    Z(\mathfrak{S}_{\mu, s_i \sigma})(\z)|_\mathrm{metaplectic} = - \frac{g(\sigma_i - \sigma_{i + 1})}{q^{1-\operatorname{sgn}(\sigma_i - \sigma_{i + 1})}} \cdot
    \begin{cases*}
      T_i \cdot Z(\mathfrak{S}_{\mu, \sigma})(\z)|_\mathrm{metaplectic} & if $\sigma_i > \sigma_{i+1}$ \\
      T_i^{-1} \cdot Z(\mathfrak{S}_{\mu, \sigma})(\z)|_\mathrm{metaplectic} & if $\sigma_i < \sigma_{i+1}$,
    \end{cases*} 
  \end{equation*}
  where $T_i$ is defined in~\eqref{eq:demazurewhittaker}.
\end{proposition}

Note that these recurrence relations using the Demazure operators $T_i$ are very similar to the ones for non-metaplectic Iwahori Whittaker functions which were first shown in~\cite{BBL, BBBGIwahori} and follow from the same arguments but using the Iwahori specialization.
They are different however, from the Demazure operator relations of~\cite{ChintaGunnellsPuskas,BBBGMetahori} for both the spherical and Iwahori metaplectic Whittaker functions appearing there which instead use a recursion on the Iwahori basis data of the representation.
Here the recursion is on the Whittaker model data $\sigma$.

We will now restrict to the case where $\sigma \in (\mathbb{Z}/n\mathbb{Z})^r$ has distinct parts which necessarily means that the number of rows $r \leqslant n$. In this case we have a monostatic base case for the recurrence from Proposition~\ref{prop:monostatic} and may use the subsequent results of Theorem~\ref{thm:recursion} to draw conclusions about spherical metaplectic Whittaker function values.
Let us start with two simple examples.

\begin{example}\label{example:thefirst} Applying Proposition~\ref{prop:monostatic} to the metaplectic specialization in this case gives the monostatic ground state partition function:
  \begin{equation} Z (\mathfrak{S}_{\mu, \sigma_{\mu}}^n) (\mathbf{z})|_\mathrm{metaplectic} = \prod_{1 \leqslant i < j \leqslant r} g(j-i)
     \mathbf{z}^{\left\lfloor \frac{\mu}{n} \right\rfloor}, \label{eq:metgroundstate} 
  \end{equation}
  where $\sigma_{\mu} \in (\mathbb{Z}/ n\mathbb{Z})^r$ is defined by $\sigma_{\mu} \equiv - \mu \bmod n$.
\end{example}

\begin{example} We may apply Theorem~\ref{thm:recursion} to the metaplectic specialization with $w=w_0$ and $w' = 1$.
That is, the residue classes of $-\mu \equiv \hat \sigma \bmod n$ are already in decreasing order, and $\sigma = w_0 \hat \sigma$ is in increasing order.
Then,
  \[ Z (\mathfrak{S}_{\mu, w_0 \hat \sigma}^n) (\mathbf{z})|_\mathrm{metaplectic} = (-1)^{r (r -
     1) / 2} T_{w_0}
     \mathbf{z}^{\left\lfloor \frac{\mu}{n} \right\rfloor}. \]
\end{example}

These examples present two extremes --- one partition function from a monostatic state expressible without Demazure-like operators at all, but with all possible Gauss sums, and the other with no Gauss sums but with the Demazure-like operators for each simple reflection appearing in the long word $w_0$. More general expressions for any choice of $\mu$ with parts in distinct residue classes mod $n$ and any permutation of these colors along the right-hand boundary will involve a mix of Demazure-like operators and Gauss sums; we state here a corollary that follows directly from Theorem~\ref{thm:recursion} by taking the metaplectic specialization.

\begin{corollary}
  \label{cor:metaplectic-recursion}
  Let $\sigma_{\mathrm{dist}} \in
  (\mathbb{Z}/ n\mathbb{Z})^r$ have distinct parts. The partition function $Z
  (\mathfrak{S}_{\mu, \sigma_{\mathrm{dist}}}^n)$ is nonzero only if
  the residue classes of $- \mu \bmod n$ are a permutation of $\sigma_\mathrm{dist}$ and in
  this case
  \[ Z (\mathfrak{S}_{\mu, \sigma_{\mathrm{dist}}}^n)(\z)|_{\mathrm{metaplectic}} = C \cdot T_w (T_{w'})^{- 1}
     \mathbf{z}^{\left\lfloor \frac{\mu}{n} \right\rfloor} \]
  where
  \begin{equation}
    \label{eq:gauss-factor}
    C = q^{\ell (w_0 w) - \ell (w)} (- q)^{r (r - 1) / 2} \left. \left(
     \frac{\tilde{\alpha}_w (\hat{\sigma})}{\tilde{\alpha}_{w_0 w'}
     (\hat{\sigma})} \right) \right|_{\mathrm{metaplectic}}
  \end{equation}
  is a product of Gauss sums, and $w, w' \in S_r$ are the unique
  permutations such that $\sigma = w \hat{\sigma}$ and $- \mu \equiv w'
  \hat{\sigma}  \bmod n$ where $\hat{\sigma}$ is a decreasing r-tuple as
  representatives in $\{ 1, \ldots, n \}$.
\end{corollary}

\begin{proof}[\textbf{Proof of Theorem~\ref{thm:DW}}]
  Following Proposition~\ref{prop:metaplectic-single-recurrence} we get from Theorem~\ref{thm:specializations-prime} that
  \begin{equation}
    \label{eq:metaplectic-zn}
    \z^{\rho-\theta}\tilde{\phi}_\theta(\z; \varpi^{\rho-\mu}) = Z(\mathfrak{S}_{\mu, -w_0\theta})(w_0 \z^n)|_\text{metaplectic}
  \end{equation}
  which is a polynomial in $\z^n$.
  We may then replace $\z$ by $\mathbf{y} := w_0\z^{1/n}$ in \eqref{eq:metaplectic-zn} which becomes a polynomial in $\z$ independent of the choice of branch for the componentwise $n$-th root of $\z$.
  
  Then Corollary~\ref{cor:metaplectic-recursion} with $\lambda + \rho := \left\lfloor \frac{\mu}{n} \right\rfloor$ and $\theta := -w_0 \sigma_\mathrm{dist}$ gives that the right-hand side of~\eqref{eq:metaplectic-zn} (but with argument $w_0\mathbf{y}^n$ which equals $\z$) is nonzero only if the residue classes of $-\mu \bmod n$ is a permutation of $\sigma_\mathrm{dist} = -w_0 \theta$ and then it equals $C \cdot T_w (T_{w'})^{-1} \z^{\lambda + \rho}$.
  This is the first statement and equality of Theorem~\ref{thm:DW}.
  Note that the operator $T_w$ has a hidden dependence on $\z$ which is why we introduced $\mathbf{y}$ for \eqref{eq:metaplectic-zn} instead of replacing $\z$ by $\z^n$ in $C \cdot T_w (T_{w'})^{-1} \z^{\lambda + \rho}$.

  The second equality follows from an Iwahori specialization of Theorem~\ref{thm:recursion}, similar to Corollary~\ref{cor:metaplectic-recursion}, together with the relation to non-metaplectic Iwahori Whittaker functions in Theorem~\ref{thm:specializations-prime} (and Remark~\ref{rem:color-relabel} which deals with the cases $n > r$).

  In particular we get that $C$ is defined as in~\eqref{eq:gauss-factor} and $C/C'$ is defined similarly but with the Iwahori specialization.
\end{proof}

\begin{remark}
  \label{rem:parahoric-conjecture}
  In the introduction we mentioned that it is an open problem to write down similar direct duality relations between metaplectic and non-metaplectic Whittaker functions when the parts of $\sigma$ (or $\theta$) are not distinct (or all equal, which is treated in Section~\ref{sec:Tokuyama}).
  As noted in Remark~\ref{rem:monostatic}, we may easily generalize the arguments for distinct $\sigma$ parts to apply anytime we have a $\sigma$-orbit with a monostatic system, and thus a monomial partition function, to obtain a duality relationship with only an overall factor of Gauss sums.
  More generally, when the parts of $\sigma$ are not distinct we expect that the Whittaker functions on the non-metaplectic side of the duality are parahoric instead of Iwahori, but it seems that a duality relationship may be more complicated than having an overall factor of Gauss sums.
  
  The straightforward approach of using the fact that a parahoric Whittaker function $\phi^P_w$ can be expressed as a specific sum of Iwahori Whittaker functions $\phi_{w'} := \phi^B_{w'}$ for certain $w'$ (for which we do have an explicit Iwahori-metaplectic duality in Theorem~\ref{thm:DW}) immediately runs into some difficulties.
  Indeed, we do not have a corresponding simple decomposition of spherical metaplectic Whittaker functions $\tilde \phi^\circ_\theta$ where $\theta$ has repeated entries in terms of $\tilde \phi^\circ_{\theta'}$ where $\theta'$ has distinct entries.
  In the lattice model perspective this statement is translated to the fact that the Boltzmann weights for the metaplectic specialization do not satisfy Property~A introduced in~\cite[Section 8]{BBBGIwahori}.
  This property says that if we have an intersection vertex with two differently colored incoming paths (travelling in a south-easterly direction) we can color the outgoing paths in two ways and the sum of the two corresponding weights equal the weight of the same vertex but where the two colors are replaced by a single color.
  In fact, out of all the members of our family of lattice models it is only the Iwahori specialization that satisfies this parahoric property.
  A possible solution for a parahoric concretization of the Iwahori-metaplectic duality would perhaps be to use a generalized version of the parahoric Property~A (weighted with Gauss sums), and is an interesting open question.  
\end{remark}

\subsection{Remarks on metaplectic Iwahori Whittaker functions}
Recall that in~{\cite{BBBGMetahori}} we considered systems made with both $m$ colors and $n$
supercolors (scolors) representing Iwahori Whittaker functions on metaplectic
covers of $\GL_r$, which provides a simultaneous generalization of the two specializations considered in this paper. 
(Here $m$ can be any integer $\geqslant r$ and in
{\cite{BBBGMetahori}} we take $m = r$.) The models are related
to the quantum supergroup $U_q(\widehat{\mathfrak{gl}}(m|n))$.

The emphasis in this paper is on two separate specializations $m=1$
and $n=1$ shown at the bottom of Figure~\ref{fig:web} giving metaplectic $\Gamma$-ice and Iwahori ice respectively.
Since colors and supercolors are related by supersymmetry, this gives further meaning to the word duality in Iwahori-metaplectic duality.

In this section, we allow $m$ and $n$ both
to be general to discuss the role of monostatic systems in
this general case.
In the scheme of {\cite{BBBGMetahori}}, the boundary conditions
parametrize the value of an Iwahori Whittaker function at an element $g$ of
the metaplectic group; the value $g$ is encoded in the top boundary
conditions, which assigns a collection of color-scolor pairs to a set of
vertical edges. Let $c_{i_1}, \ldots, c_{i_r}$ be the colors that appear in these
pairs and ${\overline{c}_{j_1}, \ldots {\overline{c}_{j_r}}}$ be the scolors
(supercolors), both sequences in the order they appear from left to right possibly with repetitions. The
$c_{i_k}$ reappear in the right boundary conditions and parametrize an
Iwahori-fixed vector in the principal series representation, and the
${\overline{c}_{j_k}}$ reappear in the left boundary conditions, parametrizing a
Whittaker model.

We consider the case where $g$ is fixed, and we ask whether there is a
monostatic system with top boundary conditions corresponding to $g$. The answer
is yes, provided the $c_{i_k}$ are distinct, which is the assumption
in~{\cite{BBBGMetahori}}, and also in~{\cite{BBBGIwahori}}, Section~7. In this
state, the $c_{i_k}$ appear on the right edge (from top to bottom) in the order
$c_{i_1}, \ldots, c_{i_r}$, and the ${\overline{c}_{j_k}}$ appear on the left edge (from
top to bottom) in the order ${\overline{c}_{j_1}, \ldots {\overline{c}_{j_r}}}$. 
Thus the $\theta$ appearing in the metaplectic Iwahori Whittaker function $\phi^{(n)}_{\theta, w}(\z; g)$, and in metaplectic Iwahori ice parametrize the left boundary, must be a permutation of the $\overline{c}_{j_k}$. Every
pair of colored lines cross but no pair of scolored lines cross in the unique
state of this model. (See Figure~12 of {\cite{BBBGMetahori}} for a picture of this ``ground state.'') Its partition function
is a multiple of $\mathbf{z}^{\lambda + \rho}$ by Lemma~2.5
of~{\cite{BBBGMetahori}},
and this agrees with the value of an Iwahori
Whittaker function by~{\cite{BBBGMetahori}} Proposition~3.8. Then it is shown
that with $g$ fixed (that is, with the top boundary conditions fixed) but
allowing the Whittaker model and Iwahori fixed vector to vary (and on the
lattice model side, the order in which ${\overline{c}_{j_1}, \ldots
{\overline{c}_{j_r}}}$ and $c_{i_1}, \ldots, c_{i_r}$ appear in the left and right
boundary conditions) that the Iwahori Whittaker functions and the partition
functions of the lattice models satisfy the same Demazure recursion relations
({\cite{BBBGMetahori}} Propositions~2.12 and~3.11). The main theorem,
identifying metaplectic Iwahori Whittaker functions with partition functions
of lattice models in the most general case is deduced from the ground state
case and these recursions.

Now an interesting fact may be observed: with $g$ fixed, in addition to the
monostatic system already described there is {\emph{another}} distinct
monostatic system provided the scolors ${\overline{c}_{j_1}, \ldots
{\overline{c}_{j_r}}}$ are all distinct. It is no longer necessary to
assume that the colors $c_{i_k}$ are distinct --- just the scolors $\overline{c}_{j_k}$.
The boundary conditions take the right
edge boundary colors to be $c_{i_r}, \ldots, c_{i_1}$ and the left edge scolors to be ${\overline{c}_{j_r}, \ldots
{\overline{c}_{j_1}}}$ both in order from the top down. There is
now a unique state to this system, in which every pair of scolored lines
cross; this is forced by the boundary conditions. The colored lines are then
forced by {\cite{BBBGMetahori}} Remark~2.4, and no colored lines cross. The
partition function for this monostatic system is, up to a power of $v$
equal to
\begin{equation}
  \label{dualbasecase} \prod_{1 \leqslant j < i \leqslant r} g (j - i)^{\varepsilon_{ij}}
  \mathbf{z}^{w_0 (\lambda + \rho)} 
\end{equation}
where $\varepsilon_{ij}=\pm1$, depending on the order of $\overline{c}_{j_k}$.
So by the main theorem of~{\cite{BBBGMetahori}}, this is a value of an Iwahori
Whittaker function. Note that this method of evaluating this Whittaker
function is very indirect, making essential use of the lattice model
interpretation. Equation (\ref{dualbasecase}) resembles (\ref{eq:metgroundstate})
but it is not the same, since in Example~\ref{example:thefirst} there
is only one color, while here there is no such assumption.

Importantly, (\ref{dualbasecase}) requires the ${\overline{c}_{j_k}}$ to be distinct,
and the
${\overline{c}_{j_k}}$ are determined by $g$. If $g = \varpi^{-\lambda}$, this
means that the components of $\lambda + \rho$ are distinct modulo $n$. In
general this is not true, so this discussion would require a modification if
there are repetitions among the parts of $\lambda + \rho$ modulo $n$, or
equivalently, among the ${\overline{c}_{j_k}}$.

\begin{remark}
We expect that the Iwahori-metaplectic duality may be generalized to the lattice models and Whittaker functions discussed in this subsection which have both colors and supercolors.
In Remark~\ref{rem:parahoric-conjecture} we conjectured that the duality is really a parahoric-metaplectic duality if we allow repetition of supercolors which are mapped to repeated colors on the non-metaplectic side.
If we consider the more general lattice model of~\cite{BBBGMetahori} with both colors and supercolors, one may speculate on a similar duality (or self-duality) where colors and supercolors are exchanged.
In terms of the corresponding metaplectic parahoric Whittaker functions this would, loosely speaking, then amount to an exchange between the data $\theta \in (\mathbb{Z}/n\mathbb{Z})^r$ that determines the Whittaker model and the parahoric data $P \subset G$ and $w \in W = S_r$ for the considered vector of the representation.
To determine the existence or exact nature of such a duality is however beyond the scope of this paper.
\end{remark}

\section{A metaplectic Tokuyama formula}\label{sec:Tokuyama}

In this section we prove a formula for values of \emph{certain} metaplectic spherical Whittaker functions in terms of Schur polynomials. The proof of this formula uses a lattice model (or equivalently writing the values of Whittaker functions as sums of certain combinatorial objects) which is reminiscent of the Tokuyama formula~\cite{hkice,Tokuyama}.  

We consider the metaplectic spherical Whittaker function $\tilde\phi^\circ_\theta(\z; g)$ defined in~\eqref{eq:spherical-Whittaker}. 
When $\theta \in (\mathbb{Z}/n\mathbb{Z})^r$ is invariant under the action of the Weyl group $S_r$, we can relate the values $\tilde\phi^\circ_\theta(\z; g)$ to the values of the non-metaplectic spherical Whittaker function $\phi^\circ(\z; g)$ by looking at the associated lattice models.
As mentioned in the introduction $\phi^\circ(\z; g)$ equals the parahoric Whittaker function $\phi^P_1(\z; g)$ defined in Section~\ref{sec:specializations} with $P = G$.
In the non-metaplectic setting, a combination of the Casselman-Shalika formula and a result of~\cite[Chapter 19]{wmd5book} (which can be combinatorially related to Tokuyama's formula~\cite{Tokuyama}) states:
\begin{equation}\label{eq:Tokuyama}
\phi^\circ(\z; \varpi^{-\lambda}) = Z^{\Delta\text{-Tokuyama}}(\mathfrak{S}_{\lambda+\rho})(\z) = \z^\rho \prod_{\alpha > 0} (1-v\z^{-\alpha}) s_\lambda(\z).   
\end{equation} 
where the product over positive roots of $\GL_r$ is called the deformed Weyl denominator, $s_\lambda(\z)$ is the Schur polynomial associated to $\lambda$ and $Z^{\Delta\text{-Tokuyama}}(\mathfrak{S}_{\lambda +\rho})(\z)$ is the partition function of the $\Delta$ lattice model introduced in~\cite{hkice} using the Tokuyama weights in their Table~2 row~2 with $t_i =-v$ and top boundary column numbers given by $\lambda + \rho$.

In Section~\ref{sec:fusion} we described a process called fusion where we obtain an equivalent description of a (super) colored system by fusing a color block of columns. 
For the models considered in this paper, the spinset of the vertical edges is then the powerset of the palette, but if we restrict to systems whose boundary edges are colored by at most a single (super) color $c$, then all its admissible vertex configurations will also be singly-colored.  
  
In particular, the singly-colored fused version of the $n$-metaplectic $\Delta'$-weights of Table~\ref{tab:metaplectic-weights} are shown in Table~\ref{tab:fused-super-boltzmann-weights}.

\begin{table}[htpb]
  \centering
  \caption{Fused metaplectic $\Delta'$-weights when restricting to one out of $n$ supercolors.}
  \label{tab:fused-super-boltzmann-weights}
\begin{tabular}{|c|c|c|c|c|c|}
  \hline
  $\texttt{a}_1$ & $\texttt{a}_2$ & $\texttt{b}_1$ & $\texttt{b}_2$ & $\texttt{c}_1$ & $\texttt{c}_2$ \\\hline\hline
  
  \begin{tikzpicture} 
    \draw [] (0,-1) -- (0,1);
    \draw [] (-1,0) -- (1,0);
  \end{tikzpicture}
 
  &
  
  \begin{tikzpicture}
    \draw [red, ultra thick, dotted] (0,-1) -- (0,1); 
    \draw [red, ultra thick, dotted] (-1,0) -- (1,0); 
  \end{tikzpicture}
 
  &
  
  \begin{tikzpicture}
    \draw [] (-1,0) -- (1,0); 
    \draw [red, ultra thick, dotted] (0,-1) -- (0,1) {};
  \end{tikzpicture}
 
  &
  
  \begin{tikzpicture}
    \draw [] (0,-1) -- (0,1);
    \draw [red, ultra thick, dotted] (-1,0) -- (1,0); 
  \end{tikzpicture} 
 
  &
  
  \begin{tikzpicture}
    \draw [] (1,0) -- (0,0) -- (0,1);
    \draw [red, ultra thick, dotted] (0,-1) -- (0,0) -- (-1,0);
  \end{tikzpicture}
  
  &
  
  \begin{tikzpicture}
    \draw [] (-1,0) -- (0,0) -- (0,-1);
    \draw [red, ultra thick, dotted] (0,1) -- (0,0) -- (1,0);
  \end{tikzpicture}
  
  \\\hline
  $1$ & $-vz^n$ & $1$ & $z^n$ & $(1 - v) z^n$ & $1$ \\\hline
\end{tabular}
\end{table}

\begin{remark}\label{remark:equal_weights} 
  We note that the weights in Table~\ref{tab:fused-super-boltzmann-weights} are exactly the Tokuyama weights in~\cite[Table 2, row 2]{hkice} if we identify the $\ominus$ signs in the Tokuyama model with the single supercolor in the metaplectic weights, replace $z$ with $z^n$ and set $t_i =-v$.
\end{remark}

By equating the partition functions mentioned above, we can write down the following formulas for certain \emph{metaplectic} Whittaker functions in terms of deformed Weyl denominators and Schur polynomials, similar to the original Casselman-Shalika formula.

\begin{theorem}[Theorem~\ref{thm:Tokuyama}]
  \label{thm:metaplectic-CS}
  Let $\theta = (\bar{c},\bar{c},\cdots,\bar{c})$ be an element of $(\mathbb{Z}/n\mathbb{Z})^r$ that is invariant under $S_r$.
  \begin{multline}
    \label{eq:metaplectic-CS}
      \z^{w_0\rho} \tilde\phi^\circ_\theta(w_0\z; \varpi^{\rho-\mu}) = \z^\theta Z(\bar{\mathfrak{S}}^n_{\mu,\theta})(\mathbf{z}^n)|_\mathrm{metaplectic} = \\
      =
      \begin{cases*}
        \z^{\theta + n \rho}\prod_{\alpha > 0} (1-v\z^{-n\alpha}) s_\lambda(\z^n) & if $\mu = n(\lambda + \rho) + \theta$ for some partition $\lambda$,\\
        0 & otherwise.
      \end{cases*}
  \end{multline} 
\end{theorem}
\begin{proof}
Note that the first equality is given by Theorem~\ref{thm:specializations-prime} and the fact that $w_0\theta = \theta$. The metaplectic specialization gives the $\Delta'$-weights of Table~\ref{tab:metaplectic-weights} which has an equivalent description using fusion. The $\theta$ boundary condition ensures that the system can only contain the single supercolor $\bar c$. This means that the partition function is zero unless the top boundary condition is such that only columns of supercolor $\bar c$ are occupied, that is $\mu = n(\lambda + \rho) + \theta$ for some partition $\lambda$. 

As shown in Remark~\ref{remark:equal_weights} the singly-colored fused metaplectic $\Delta'$-weights are exactly the Tokuyama weights with $\z$ replaced by $\z^n$, and after identifying the different boundary data we get that
\begin{equation*}
  Z(\bar{\mathfrak{S}}^n_{n (\lambda + \rho) + \theta, \theta})(\mathbf{z}^n)|_\mathrm{metaplectic} = Z ( \mathfrak{S}^{\Delta \text{-Tokuyama}}_{\lambda +\rho}) (\mathbf{\z}^n). 
\end{equation*}

The statement now follows from equation~\eqref{eq:Tokuyama}.
\end{proof}

The following relationship between Schur polynomials is easily proven from the Weyl character formula:
\begin{equation*}
    s_\lambda(\z^n) = \frac{\z^\rho \prod_{\alpha>0}(1-v\z^{-\alpha})}{\z^{n\rho}\prod_{\alpha>0}(1-v\z^{-n\alpha})} s_{n(\lambda + \rho) - \rho}(\z).
\end{equation*}

From this statement we obtain another expression for the metaplectic spherical Whittaker function.

\begin{corollary}\label{cor:Shimura}
    Let $\lambda$ be a partition and $\theta \in (\mathbb{Z}/n\mathbb{Z})^r$ invariant under $S_r$. Then,
    \begin{equation} 
      \z^{w_0\rho} \tilde\phi^\circ_\theta(w_0\z; \varpi^{-n(\lambda +\rho) + \rho - \theta}) = \z^{\theta + \rho} \prod_{\alpha>0}(1-v\z^{-\alpha}) s_{n(\lambda + \rho) - \rho}(\z) = \z^{\theta} \phi^\circ (\z; \varpi^{-n (\lambda+\rho)+\rho}), 
    \end{equation}
where $\phi^\circ$ is the non-metaplectic spherical Whittaker function as in~\eqref{eq:Tokuyama}. 
\end{corollary}

Let $\mathbf{x}=(x_1,\cdots,x_r), \mathbf{y}=(y_1,\cdots,y_r)$ be two $r$-tuples of nonzero complex numbers. A well-known property of Schur polynomials is the Cauchy identity
\[ \sum_{\lambda} s_\lambda (\mathbf{x}) s_\lambda (\mathbf{y}) q^{-|\lambda|s} = \prod_{1 \leq i,j \leq r} (1-x_iy_jq^{-s})^{-1},  \]
where the sum on the left is over all partitions $\lambda$ of length $r$ (here we introduce a factor of $q^{-s}$ to the standard Cauchy identity for number theoretic purposes, see~\cite[Proposition~1]{Bump-Rallis_volume}). In the classical setting, one may use the Casselman-Shalika formula to show a Cauchy identity for spherical Whittaker functions, which has applications in the Rankin-Selberg method (see, for example, ~\cite[Section~4]{Bump-Rallis_volume} and references therein). Theorem~\ref{thm:metaplectic-CS} allows us to show the following Cauchy identity for \emph{metaplectic} Whittaker functions:
\begin{corollary}[Corollary~\ref{cor:Cauchy}]\label{cor:Rankin-Selberg}
    Let $\theta = (\bar{c},\bar{c},\cdots,\bar{c})$. Then
    \begin{equation*}
    \sum_{\mu}  \tilde\phi^\circ_\theta(\mathbf{x}; \varpi^{\rho-\mu}) \tilde\phi^\circ_\theta(\mathbf{y}; \varpi^{\rho-\mu}) q^{-|\mu|s} = 
    (\mathbf{xy})^{\theta - n \rho + w_0 \rho}\prod_{\alpha > 0} (1-v\mathbf{x}^{n\alpha})(1-v\mathbf{y}^{n\alpha})  \prod_{\mathclap{1 \leq i,j \leq r}} (1-x_i^n y_j^n q^{-s})^{-1}. 
    \end{equation*}
\end{corollary}

\section{Fock space operators}
\label{sec:Fock-space}

We consider a Drinfeld twisted $U_q^\alpha(\widehat{\mathfrak{sl}}(m))$ quantum Fock space defined in~\cite{KMS, BBBGVertex} parametrized by $\alpha_{i,j} \in \mathbb{C}^\times$ which is $m$-periodic in $i$ and $j$ such that $\alpha_{i,j}\alpha_{j,i} = 1$ and $\alpha_{i,i} = 1$.
The Fock space $\mathfrak{F}$ is spanned by the semi-infinite monomials
\begin{equation}
  \label{eq:Fock-space-basis}
  u_{i_k} \wedge u_{i_{k - 1}} \wedge \cdots
\end{equation}
with $i_k > i_{k - 1} > \cdots$ and $i_k = k$ for $k \ll 0$ where the Drinfeld twisted quantum wedge is defined by $u_l \wedge u_k = - u_k \wedge u_l$ if $k \equiv l \bmod m$ and otherwise if $k > l$, with $i := \res_m(k-l)$,
\begin{multline*}
  u_l \wedge u_k = -q \alpha_{l,k} \, u_k \wedge u_l +{} \\ 
  {}+ (q^2 - 1)\Bigl( \;\, \sum_{\mathclap{\substack{n\in\mathbb{Z}_{\geqslant0} \\ k-i-mn > l+i+mn}}} \; q^{2n} u_{k-i-mn} \wedge u_{l+i+mn} - \alpha_{l,k} \; \sum_{\mathclap{\substack{n \in \mathbb{Z}_{>0} \\ k-mn > l+mn}}} \; q^{2n-1} u_{k-mn} \wedge u_{l+mn} \Bigr).
\end{multline*}

We will later construct systems with an infinite number of columns for our family of lattice models using the unfused supercolor description with the same weights in Table~\ref{tab:super-boltzmann-weights} as for the finite system.
Before that, we will consider modified finite systems for which the left and right boundaries are unoccupied.
Let $\lambda$ and $\mu$ be strict partitions with parts $\lambda_1 > \lambda_2 > \cdots > \lambda_\ell \geqslant 0$ and $\mu_1 > \cdots > \mu_{\ell'} \geqslant 0$ (where we will allow $0$ for convenience).
Note that if the supercolor-assignment of any three of the four edges of a vertex in Table~\ref{tab:super-boltzmann-weights} is fixed then there is at most one admissible supercolor-assignment for the fourth remaining edge.
This means that a one-row state for a finite system with unoccupied left and right boundaries is completely determined by its top and bottom boundary edges which can be encoded by $\lambda$ and $\mu$ describing which column numbers that are occupied for the top and bottom edges respectively.

Conversely, for any two strict partitions $\lambda$ and $\mu$ there is at most one admissible one-row state with $N = \max(\lambda_1, \mu_1)+1$ columns.
Let $Z^\text{finite}_{\mu,\lambda}(z)$ denote the Boltzmann weight for this finite one-row state if admissible and otherwise $0$.
This forms a matrix with coefficients enumerated by strict partitions which is called the one-row transfer matrix of the (finite) lattice model.
Note that because of the color conservation law \eqref{eq:conservation} the partition function is nonzero only if $\ell = \ell'$, that is, $\lambda$ and $\mu$ have the same length.

\begin{remark}
Note that if we increase $N$ in our one-row state this would only introduce extra $\texttt{a}_1$ vertex configurations on the left which have weight $1$ and it would not affect the admissibility of the top and bottom boundary edges specified by $\lambda$ and $\mu$.
Thus, especially when treating multiple rows, it may be convenient to instead set $N$ to be some sufficiently large number valid for all rows.
A multi-row system can then be constructed by a product of the above transfer matrices effectively summing over the possible configurations of interior vertical edges and the different resulting matrix elements give the partition function of the system for given boundary conditions.
\end{remark}

\begin{remark}
  \label{rem:n-shift}
  Note also that if $\lambda$ and $\mu$ in a finite one-row state are shifted $m$ columns to the left the Boltzmann weights stay the same since the supercolors are $m$-periodic. 
\end{remark}

We now wish to define a corresponding lattice model with an infinite number of columns, and to describe its top (and bottom) boundary edges with elements of a quantum Fock space as well as construct a Fock space operator $T := T(z)$ for the transfer matrix.

In the finite case we can, instead of using a strict partition $\lambda$, express the occupancy for the top (or bottom) edges of a one-row state by a finite quantum wedge product $u_{\lambda_1} \wedge \cdots \wedge u_{\lambda_\ell}$.
However, to be able to describe the occupancy using a Fock space element we would need an infinite wedge product such as
\begin{equation}
  \label{eq:lambda-ket}
  \ket{\lambda} := u_{\lambda_1} \wedge \cdots \wedge u_{\lambda_\ell} \wedge u_{-1} \wedge u_{-2} \wedge \cdots. 
\end{equation}
But if we were to add an infinite number of occupied top and bottom edges to the right of our finite one-row state we would get an infinite number of $\texttt{b}_1$ vertex configurations each of weight $-\Phi/q$.

In other words, while there is no problem in extending our finite system to the left (by increasing $N$ which only introduces $\texttt{a}_1$ configurations of weight $1$), extending the system to the right with a sea of occupied edges, and therefore $\texttt{b}_1$ configurations, requires some form of normalization.
Such a normalization goes hand in hand with the choice of numbering for the columns, that is, where column $0$ is located.
As suggested by our initial formulation of the finite system in terms of partitions with positive parts, we choose to disregard all $\texttt{b}_1$-weights to the right of column $0$.
Colloquially, we divide by the infinite diagonal transfer matrix element for the vacuum $\ket{\emptyset} = u_{-1} \wedge u_{-2} \wedge \cdots$.
This means that the weight for our normalized infinite one-row state with top row $\ket{\lambda}$ and bottom row $\ket{\mu}$ matches that of the finite one-row state $Z^\text{finite}_{\mu,\lambda}$.
This defines our infinite system as a normalized extension of the finite system for all Fock space elements $\ket{\lambda}$ and $\ket{\mu}$ given by strict partitions $\lambda$ and $\mu$.
Any other Fock space element $u_{i_r} \wedge u_{i_{r-1}} \wedge \cdots$ with $i_j$ not necessarily positive can be incorporated by repeatedly shifting its column positions by $m$ steps to the left.
By Remark~\ref{rem:n-shift} such an $m$-shift does not change any of the vertex weights, however the normalization has to be compensated by a factor of $(-\Phi/q)^{-m}$.

In more detail, define the shift operator $Q : \mathfrak{F} \to \mathfrak{F}$ by
\begin{equation*}
  Q(u_{i_r} \wedge u_{i_{r-1}} \wedge \cdots) = u_{i_r+m} \wedge u_{i_{r-1}+m} \wedge \cdots.
\end{equation*}
Then any basis element~\eqref{eq:Fock-space-basis} of $\mathfrak{F}$ can be expressed as $\ket{\lambda;k} := Q^{-k} \ket{\lambda}$ for some strict partition $\lambda$ and positive integer~$k$.
Note that there is a freedom in choosing $k$ and $\lambda$ since
\begin{equation*}
  Q^{-k}\ket{\lambda} = Q^{-(k+1)} \ket{\lambda^+} \qquad \text{where } \lambda^+ := (\lambda_1+m, \ldots, \lambda_\ell+m, m-1, m-2, \ldots, 0).
\end{equation*}
Using this freedom any two elements in $\mathfrak{F}$ can be expressed as $\ket{\lambda;k}$ and $\ket{\mu;k}$ with the same integer $k$ and two strict partitions $\lambda$ and $\mu$.

We also note that
\begin{equation}
  \label{eq:Z-shift}
  Z^\text{finite}_{\mu^+,\lambda^+} = \Bigl( -\frac{\Phi}{q} \Bigr)^m Z^\text{finite}_{\mu,\lambda}
\end{equation}
because of Remark~\ref{rem:n-shift} and the introduction of $m$ extra $\texttt{b}_1$ vertices.
This means that the transfer matrix $T$ for our infinite system with the above normalization is defined on the whole of $\mathfrak{F}$ by
\begin{equation}
  \label{eq:T-def}
  \bra{\mu;k} T(z) \ket{\lambda;k} := \Bigl(- \frac{\Phi}{q} \Bigr)^{-mk} Z^\text{finite}_{\mu,\lambda}(z)
\end{equation}
which does not depend on the choice of $k$ by \eqref{eq:Z-shift}.

\subsection{Proof of Theorem~\ref{thm:T-hamiltonian}}
Let $U(z) := \bigl( -\frac{\Phi}{q} \bigr)^{J_0+1} e^{H(z)}$ which is the right-hand side of \eqref{eq:T-hamiltonian}.
Since $Q^{-1}J_kQ = J_k$ for $k\neq0$ and $Q^{-1}J_0Q = J_0 + m$ we have that $Q^{-1}U(z)Q = \bigl( -\frac{\Phi}{q} \bigr)^n U(z)$ similar to $T$ in \eqref{eq:T-def}.
Thus, it is enough to show that $T = U$ on the subspace of $\mathfrak{F}$ spanned by $\ket{\lambda;0} = \ket{\lambda}$ where $\lambda$ is a strict partition.

The case $\Phi = -q$ and $\alpha_{i,j} = -g(i-j)/q$ is the $\Delta'$-ice lattice model with the modified weights of row~A in Table~\ref{tab:metaplectic-weights} where we recall that we have replaced $z$ by $\zeta = z^m$ compared to the metaplectic specialization of our lattice model.

Similar to the proof of Lemma~\ref{lem:Delta-Deltaprime} one can show that the partition function for and infinite one-row system where all paths enter at the top and exit at the bottom is the same for both $\Delta'$-ice weight (row A of Table~\ref{tab:metaplectic-weights}) and $\Delta$-ice weights (row B).
This is because the paths only travel along horizontal segments which have a length that is a multiple of $m$ since vertical edges are constrained to only carry particular supercolors determined by the column position, and thus we obtain the same factors of $z^m$ in both cases.

As mentioned in the introduction, we showed in \cite[Theorem A]{BBBGVertex} that the transfer matrix for this $\Delta$-ice model equals the operator $e^{H(\zeta)}$ on a Fock space with the above Gauss sum Drinfeld twist.
The statement for a general $\alpha_{i,j}$ replacing $-g(i-j)/q$ in both the $\texttt{a}_2$ configuration of the lattice model as well as the Fock space Drinfeld twist follows by the same arguments.

It remains to prove the statement for general $\Phi$.
Note that $(J_0+1)\ket{\lambda} = \ell \ket{\lambda}$ where $\ell$ is the length of $\lambda$ which means that $J_0+1$ counts the number of paths in the one-row system.
Thus the factor $(-\tfrac{\Phi}{q})^{J_0+1}$ in $U$ can be incorporated in the lattice model by multiplying each vertex configuration weight by $(-\tfrac{\Phi}{q})$ if its bottom edge is occupied which gives exactly the weights of Table~\ref{tab:boltzmann-weights}.\qed

\bibliographystyle{habbrv} 
\bibliography{vertex}

\end{document}